\documentclass[11pt]{article}
\usepackage{amsmath,amsfonts,amssymb, amsthm,enumerate,graphicx}
\usepackage{tikz,authblk}

\usepackage{wrapfig}
\usepackage{subcaption}
\usepackage{url}
\usepackage{hyperref}
\usepackage{stackengine,scalerel}
\usetikzlibrary {decorations.pathmorphing, decorations.pathreplacing, decorations.shapes}

\newtheorem{theorem} {{\textsf{Theorem}}}
\newtheorem{proposition}[theorem]{{\textsf{Proposition}}}

\newtheorem{conj}[theorem]{{\textsf{Conjecture}}}

\newtheorem{remark}[theorem]{{\textsf{Remark}}}

\newtheorem{lemma}[theorem]{{\textsf{Lemma}}}
\def\Z{{\mathbb{Z}^{n}_2}}
\begin{document}

\title{Regular Genus of
$\mathbb{S}^2\times\mathbb{S}^1\times\mathbb{S}^1$, the $4$-Torus, and Small Covers over $\Delta^2\times\Delta^2$}

\author{Anshu Agarwal and Biplab Basak$^1$}
	
\date{November 16, 2026}
	
\maketitle
	
\vspace{-10mm}
\begin{center}
		
\noindent {\small Department of Mathematics, Indian Institute of Technology Delhi, New Delhi, 110016, India.$^2$}

\footnotetext[1]{Corresponding author}
		
\footnotetext[2]{{\em E-mail addresses:} \url{maz228084@maths.iitd.ac.in} (A. Agarwal), \url{biplab@iitd.ac.in} (B. Basak).}
		
\medskip
	
\end{center}
	
\hrule
	
\begin{abstract}
A crystallization is an edge-colored graph encoding a contracted triangulation of a closed PL manifold, providing a combinatorial framework for studying PL-manifold invariants. In this framework, we determine the regular genus of several important $4$-manifolds. We first prove that $\mathcal{G}(\mathbb{S}^2\times\mathbb{S}^1\times\mathbb{S}^1)=6$, thereby resolving a conjecture on the regular genus of $\mathbb{S}^2\times\mathbb{S}^1\times\mathbb{S}^1$. We also establish that $\mathcal{G}(\mathbb{S}^1\times\mathbb{S}^1\times\mathbb{S}^1\times\mathbb{S}^1)=16$. These computations lead to further observations on the regular genus of the $n$-torus and motivate the conjecture $\mathcal{G}(\mathbb{T}^n)=1+\frac{(n+1)!(n-3)}{8}$ for $n\geq 5$. We then apply our combinatorial methods to small covers. For the simple polytope $\Delta^2\times\Delta^2$, we classify the corresponding small covers up to Davis-Januszkiewicz equivalence and obtain exactly seven equivalence classes. One of these is $\mathbb{RP}^2\times\mathbb{RP}^2$, while the remaining six are $\mathbb{RP}^2$-bundles over $\mathbb{RP}^2$. We prove that all seven small covers have regular genus $8$. Thus, the paper provides regular genus computations for a collection of $4$-manifolds arising from products and small covers. The approach is purely combinatorial and relies on various combinatorial moves.
\end{abstract}

\noindent {\small {\em MSC 2020\,:} Primary 05E45; Secondary 05C15, 52B11, 52B70, 57Q15.
		
}
	
\medskip

\section{Introduction}
Every closed PL $n$-manifold admits a colored triangulation, whose dual is an $(n+1)$-regular edge-colored graph. Such a graph is called a crystallization when the corresponding triangulation is contracted. Crystallizations provide a combinatorial representation of PL manifolds and allow topological invariants to be studied through the combinatorics of edge-colored graphs. In particular, every crystallization admits regular embeddings into closed surfaces. The regular genus of a crystallization is defined as the minimum genus of an orientable surface, or half the genus of a non-orientable surface, into which the graph admits a regular embedding. The regular genus $\mathcal{G}(M)$ of a closed connected PL manifold $M$ is the minimum of the regular genera of its crystallizations. Introduced by Gagliardi, regular genus extends the classical notion of surface genus and the Heegaard genus of $3$-manifolds to higher-dimensional PL manifolds \cite{g81}. Consequently, the study of regular genus translates questions about PL manifolds into combinatorial questions concerning edge-colored graphs. Recently,  the trisection genus of PL 4-manifolds arising from colored triangulations has been extensively investigated by Casali and Cristofori \cite{cc25}.

A fundamental result of Gagliardi states that a closed connected PL $n$-manifold $M$ has $\mathcal{G}(M)=0$ if and only if $M$ is PL homeomorphic to $\mathbb{S}^n$ \cite{fg82}. In particular, $\mathcal{G}(\mathbb{S}^4)=0$, and hence $\mathcal{G}(\mathbb{S}^n\times\mathbb{S}^1)\geq 1$. Explicit crystallizations of $\mathbb{S}^n\times\mathbb{S}^1$ with $2(n+2)$ vertices and regular genus one yield $\mathcal{G}(\mathbb{S}^n\times\mathbb{S}^1)=1$ \cite{ca89,fgg86}. Moreover, the closed connected orientable prime PL $n$-manifolds of regular genus one have been characterized: such a manifold is PL homeomorphic to $\mathbb{S}^{n-1}\times\mathbb{S}^1$ \cite{ca89,ch93,fgg86}. These results illustrate the effectiveness of crystallizations in determining and characterizing low regular genus PL manifolds.

For higher-dimensional PL manifolds, and in particular for closed PL $4$-manifolds, the regular genus is considerably less understood. The combinatorial approach via crystallizations has nevertheless led to several exact computations and low-genus classification results. In dimension four, the regular genera of the simply connected prime PL $4$-manifolds $\mathbb{S}^4$, $\mathbb{CP}^2$, $\mathbb{S}^2\times\mathbb{S}^2$, and the $K3$ surface are known. More precisely, $\mathcal{G}(\mathbb{S}^4)=0$, $\mathcal{G}(\mathbb{CP}^2)=2$, $\mathcal{G}(\mathbb{S}^2\times\mathbb{S}^2)=4$, and $\mathcal{G}(K3)=23$ \cite{ca99,bj16,cm93}. The real projective spaces provide another family for which the regular genus is explicitly known: $\mathcal{G}(\mathbb{RP}^n)=1+(n-3)2^{n-3}$ \cite{ab25,cs07}. In addition, $\mathcal{G}(\mathbb{RP}^2\times\mathbb{S}^2)=5$ was determined in \cite{bc17}.

The low values of regular genus in dimension four exhibit a striking classification pattern. A closed connected orientable prime PL $4$-manifold has regular genus two if and only if it is PL homeomorphic to $\mathbb{CP}^2$ \cite{ca99,g89}, while no closed connected orientable prime PL $4$-manifold has regular genus three. Furthermore, $\mathcal{G}(\mathbb{S}^2\times\mathbb{S}^2)=4$, and every closed connected orientable prime PL $4$-manifold of regular genus four is topologically homeomorphic to $\mathbb{S}^2\times\mathbb{S}^2$ \cite{bj16,cm93}. There is no closed connected orientable prime PL $4$-manifold of regular genus five. The classification of closed connected orientable prime PL $4$-manifolds of regular genus at most five is available in \cite{B19,ca89,ch93,fgg86,s99}. These results demonstrate both the strength of crystallization methods in determining regular genus and the scarcity of low regular genus $4$-manifolds, motivating the determination of the regular genus for further classes of $4$-manifolds.

The classification of closed connected orientable prime PL $4$-manifolds of regular genus greater than five remains largely open, even at the level of topological classification. In \cite{s99}, it was shown that $\mathcal{G}(\mathbb{RP}^3\times\mathbb{S}^1)=6$, and it was conjectured that every closed connected orientable prime PL $4$-manifold of regular genus six is topologically homeomorphic to $\mathbb{RP}^3\times\mathbb{S}^1$. This conjecture was disproved in \cite{B19}, where certain $4$-dimensional mapping tori with regular genus six were constructed that are not topologically homeomorphic to $\mathbb{RP}^3\times\mathbb{S}^1$. The same work conjectured that $\mathcal{G}(\mathbb{S}^2\times\mathbb{S}^1\times\mathbb{S}^1)=6$. We resolve this conjecture by proving
\[
\mathcal{G}(\mathbb{S}^2\times\mathbb{S}^1\times\mathbb{S}^1)=6.
\]

The regular genus of the $2$-torus $\mathbb{S}^1\times\mathbb{S}^1$ is $1$, while that of the $3$-torus $\mathbb{S}^1\times\mathbb{S}^1\times\mathbb{S}^1$ is $3$. This naturally raises the question of determining the regular genus of the $4$-torus $\mathbb{S}^1\times\mathbb{S}^1\times\mathbb{S}^1\times\mathbb{S}^1$. Previously, the bounds
\[
4\leq\mathcal{G}(\mathbb{S}^1\times\mathbb{S}^1\times\mathbb{S}^1\times\mathbb{S}^1)\leq28
\]
were known (cf. \cite{fg82p,gg93}). The lower bound was subsequently improved from $4$ to $16$ in \cite{bc17}, but the exact value remained unknown. We settle this problem by using crystallization techniques to prove (cf. Theorem \ref{theorem:G2}) that
\[
\mathcal{G}(\mathbb{S}^1\times\mathbb{S}^1\times\mathbb{S}^1\times\mathbb{S}^1)=16.
\]
We further investigate the regular genus of the $n$-torus and obtain observations that lead to the following conjecture: for $n\geq5$,
\[
\mathcal{G}(\underbrace{\mathbb{S}^1\times\mathbb{S}^1\times\cdots\times\mathbb{S}^1}_{n\text{ times}})
=1+\frac{(n+1)!(n-3)}{8}
\]
(cf. Conjecture \ref{conj}).

We next investigate the regular genus of small covers, a class of closed manifolds introduced by Davis and Januszkiewicz \cite{dj91}. A small cover over a simple convex $n$-polytope $P^n$ is a closed $n$-manifold equipped with a locally standard $\mathbb{Z}_2^n$-action whose orbit space is $P^n$. Various topological complexities of small covers have been studied in \cite{bss25}. Here, we consider the case $P=\Delta^2\times\Delta^2$, the Cartesian product of two $2$-simplices. Using the combinatorial data associated with the characteristic functions, we classify all small covers over $\Delta^2\times\Delta^2$ up to Davis--Januszkiewicz (D-J) equivalence and show that there are exactly seven equivalence classes (cf. Lemma \ref{lemma:seven}). One of these is $\mathbb{RP}^2\times\mathbb{RP}^2$, while the remaining six are $\mathbb{RP}^2$-bundles over $\mathbb{RP}^2$. We then determine the regular genus of all seven small covers and show that, remarkably, they all have the same value, namely $8$. In particular, we obtain (cf. Theorem \ref{thm:RP2})
\[
\mathcal{G}(\mathbb{RP}^2\times\mathbb{RP}^2)=8.
\]

Regular genus is closely related to several fundamental questions concerning the topology of PL $4$-manifolds. In particular, the additivity of regular genus under connected sum for simply-connected PL $4$-manifolds would imply the $4$-dimensional smooth Poincar\'e conjecture. In \cite{bb18}, weak semi-simple crystallizations were introduced for PL $4$-manifolds, and it was observed that a sufficiently broad existence result for such crystallizations would have significant consequences for this problem. The crystallizations constructed in this paper for $\mathbb{S}^2\times\mathbb{S}^1\times\mathbb{S}^1$, $\mathbb{S}^1\times\mathbb{S}^1\times\mathbb{S}^1\times\mathbb{S}^1$, $\mathbb{RP}^2\times\mathbb{RP}^2$, and the non-trivial $\mathbb{RP}^2$-bundles over $\mathbb{RP}^2$ are weak semi-simple. Since the class of PL $4$-manifolds admitting weak semi-simple crystallizations is closed under connected sum, the results of this paper enlarge the class of PL $4$-manifolds for which the regular genus can be determined explicitly. In particular, this class contains $\mathbb{S}^4$, $\mathbb{CP}^2$, $\mathbb{S}^2\times\mathbb{S}^2$, the $K3$ surface, $\mathbb{S}^3\times\mathbb{S}^1$, $\mathbb{RP}^4$, $\mathbb{RP}^3\times\mathbb{S}^1$, $\mathbb{RP}^2\times\mathbb{S}^2$, $\mathbb{S}^2\times\mathbb{S}^1\times\mathbb{S}^1$, $\mathbb{RP}^2\times\mathbb{RP}^2$, the $4$-torus $\mathbb{S}^1\times\mathbb{S}^1\times\mathbb{S}^1\times\mathbb{S}^1$, certain $4$-dimensional mapping tori, and their connected sums, with orientation reversals allowed where appropriate. Thus, our constructions provide additional examples of weak semi-simple crystallizations and contribute to a broader combinatorial framework for the determination of regular genus in dimension four.

\section{Preliminaries}\label{pre}

\subsection{Crystallization} \label{crystal}

In this article, all the spaces and maps are considered in the PL category, unless stated otherwise \cite{rs72}. Suppose that $K$ is a finite collection of closed balls and write $|K| = \bigcup_{B\in K} B $. Then $K$ is called a {\it simplicial cell complex} if the following conditions hold.

\begin{enumerate}[$(i)$]
\item $|K|=$ $\bigsqcup_{B\in K}$ int$(B)$,

\item if $A,B\in K$, then $A\cap B$ is a union of balls of $K$,

\item  for each $h$-ball $A\in K$, the poset $\{B\in K \, | \, B \subset A\}$, ordered by inclusion, is isomorphic with the lattice of all faces of the standard $h$-simplex.
\end{enumerate}
\noindent A {\it pseudo-triangulation} of a polyhedron $P$ is a pair $(K,f)$, where $K$ is a simplicial cell complex and $f: |K| \to P$ is a PL homeomorphism (see \cite{fgg86} for more details). A maximal dimensional closed ball of $K$ is called a {\it facet}. If all the facets of $K$ are of the same dimension, then $K$ is called a {\it pure} simplicial cell complex. 

The crystallization theory provides a tool for representing piecewise linear (PL) manifolds of any dimension combinatorially, using edge-colored graphs.
Throughout the article, by a graph, we mean a multigraph with no loops. Let $\Gamma = (V(\Gamma), E(\Gamma))$  be an edge-colored multigraph with no loops, where the edges are colored (or labeled) using $\Delta_n := \{0, 1, \dots, n\}$. The elements of the set $\Delta_n$ are referred to as the {\it colors} of $\Gamma$. The coloring of $\Gamma$ is called a \textit{proper edge-coloring} if any two adjacent edges in $\Gamma$ have different labels. In other words, for a proper edge-coloring, there exists a surjective map $\gamma: E(\Gamma) \to \Delta_n$ such that  $\gamma(e_1) \ne \gamma(e_2)$ for any two adjacent edges $e_1$ and $e_2$. We denote a properly edge-colored graph as $(\Gamma,\gamma)$, or simply as $\Gamma$ if the coloring is understood. If a graph $(\Gamma,\gamma)$ is such that the degree of each vertex in the graph is $n+1$, then it is said to be {\it $(n+1)$-regular}.
We refer to \cite{bm08} for standard terminologies on graphs.

An {\it $(n+1)$-regular colored graph} is a pair $(\Gamma,\gamma)$, where $\Gamma$ is $(n+1)$-regular and $\gamma$ is a proper edge-coloring of $\Gamma$. For each $\mathcal{C} \subseteq \Delta_n$ with cardinality $k$, the graph $\Gamma_\mathcal{C} = (V(\Gamma), \gamma^{-1}(\mathcal{C}))$ is a $k$-regular colored graph with edge-coloring $\gamma|_{\gamma^{-1}(\mathcal{C})}$. For a color set $\{j_1,j_2,\dots,j_k\} \subset \Delta_n$, $g(\Gamma_{\{j_1,j_2, \dots, j_k\}})$ or $g_{\{j_1, j_2, \dots, j_k\}}$ denotes the number of connected components of the graph $\Gamma_{\{j_1, j_2, \dots, j_k\}}$. A graph $(\Gamma,\gamma)$ is called {\it contracted} if the subgraph $\Gamma_{\hat{j}} = \Gamma_{\Delta_n\setminus \{j\}}$ is connected, i.e., $g_{\hat{j}}=1$, for $j \in \Delta_n$.
 
For a properly edge-colored graph $(\Gamma,\gamma)$ with the color set $\Delta_n$, a corresponding $n$-dimensional simplicial cell complex ${\mathcal K}(\Gamma)$ is constructed as follows:

\begin{itemize}
\item{} For each vertex $v\in V(\Gamma)$, take an $n$-simplex $\sigma(v)$ with vertices labeled by $\Delta_n$.

\item{} Corresponding to each edge of color $j$ between $v_1,v_2\in V(\Gamma)$, identify the ($n-1$)-faces of $\sigma(v_1)$ and $\sigma(v_2)$ opposite to the $j$-labeled vertices such that the vertices with the same labels coincide.
\end{itemize}

The simplicial cell complex \(\mathcal{K}(\Gamma)\) is \((n+1)\)-colorable, i.e., its 1-skeleton admits a proper vertex-coloring by \(\Delta_n\). Note that all colors of \(\Delta_n\) need not appear in the edge-coloring of \((\Gamma,\gamma)\). The topological space \( |\mathcal{K}(\Gamma)| \) inherits a natural PL structure, and the graph \((\Gamma,\gamma)\) is said to represent \( |\mathcal{K}(\Gamma)| \). In this article, we primarily consider \((\Gamma,\gamma)\) as an \((n+1)\)-regular colored graph. In that case, in \(\mathcal{K}(\Gamma)\) every \((n-1)\)-simplex is contained in exactly two \(n\)-simplices.
If \( |\mathcal{K}(\Gamma)| \) is homeomorphic to an \( n \)-manifold \( M \), then \((\Gamma, \gamma)\) is referred to as a \textit{gem (graph encoded manifold)} of \( M \), or, simply, it is said that \((\Gamma, \gamma)\) represents \( M \). In this context, $\mathcal{K} (\Gamma)$ is described as a \textit{colored triangulation} of $ M $. 
The {\it disjoint star} of $\sigma \in \mathcal{K}(\Gamma)$ is a simplicial cell complex that consists of all the $n$-simplices of $\mathcal{K}(\Gamma)$ that contain $\sigma$, with re-identification of only their $(n-1)$-faces containing $\sigma$, as in $\mathcal{K}(\Gamma)$. The {\it disjoint link} of $\sigma \in \mathcal{K}(\Gamma)$ is the subcomplex of its disjoint star generated by the simplices that do not intersect $\sigma$.

From the construction above, it can be easily seen that for any subset $\mathcal{C} \subset \Delta_n$ with cardinality $k+1$, $\mathcal{K}(\Gamma)$ has as many $ k $-simplices with vertices labeled by $\mathcal{C}$ as there are connected components of $\Gamma_{\Delta_n \setminus \mathcal{C}}$\cite{fgg86}. Specifically, each component of $(n-k)$-regular colored subgraph induced by the colors from $\Delta_n \setminus \mathcal{C}$ corresponds to the disjoint link of a $ k $-simplex with vertices labeled by $\mathcal{C}$. For further information on CW complexes and related concepts, refer to \cite{bj84}. An $(n+1)$-regular colored gem $(\Gamma,\gamma)$ of a closed manifold $M$ is called a {\em crystallization} of $M$ if it is contracted. In other words, the corresponding simplicial cell complex $\mathcal{K}(\Gamma)$ has exactly $(n+1)$ vertices.

If $ K $ is a colored triangulation of an $ n $-manifold $ M $, meaning $ K $ is an $ (n+1) $-colorable simplicial cell complex and $ |K| $ is homeomorphic to $ M $, then by reversing the steps of the above construction, we obtain a gem $ (\Gamma,\gamma) $ of $ M $. Clearly, $\mathcal{K}(\Gamma)=K$. 
Every closed PL $n$-manifold $M$ is known to admit a gem. From a gem, a crystallization of $M$ can be easily obtained through certain combinatorial moves (see \cite{fg82i, fgg86} for more details). Additionally, it is well established in the literature that a gem of a closed PL manifold $M$ is bipartite if and only if $M$ is orientable. 

Let $(\Gamma,\gamma)$ and $(\bar{\Gamma},\bar{\gamma})$ be two  $(n+1)$-regular colored graphs with color sets $\Delta_n$ and $\bar{\Delta}_n$, respectively. Then $I:=(I_V,I_c):\Gamma \to \bar{\Gamma}$ is called an {\em isomorphism} if $I_V: V(\Gamma) \to V(\bar{\Gamma})$ and $I_c:\Delta_n \to \bar{\Delta}_n$ are bijective maps such that $uv$ is an edge of color $i \in \Delta_n$ if and only if $I_V(u)I_V(v)$ is an edge of color $I_c(i) \in \bar{\Delta}_n$. The graphs $(\Gamma, \gamma)$ and $(\bar{\Gamma}, \bar{\gamma})$ are then said to be {\it isomorphic}. We will now briefly explain some crystallization moves \cite{fg82i}.

\subsection{Crystallization moves}\label{move}
Let $(\Gamma,\gamma)$ be an $(n+1)$-regular colored gem of a closed $n$-manifold $M$. Let $v_1$ and $v_2$ be connected by $h$ edges colored by $i_1,i_2,\cdots, i_h$ such that $v_1$ and $v_2$ lie in different components of $\Gamma_{\Delta_n \setminus \{i_1,i_2,\cdots, i_h\}}$, where $1\le h\le n$. Then we say that $v_1$ and $v_2$ form an \textit{$h$-dipole} with respect to the color set $\{i_1,i_2,\cdots,i_h\}$. Consider the graph $\Gamma_1$ with $V(\Gamma_1)=V(\Gamma)\setminus\{v_1,v_2\}$. For $j\in \Delta_n \setminus \{i_1,i_2,\cdots, i_h\}$, if $v_1$ and $v_2$ are connected to $v_1^\prime$ and $v_2^\prime$ by $j$-colored edges in $\Gamma$, respectively, then $v_1^\prime$ and $v_2^\prime$ are incident to a $j$-colored edge in $\Gamma_1$. Edges that are not incident to $v_1$ or $v_2$ in $\Gamma$ remain unchanged in $\Gamma_1$. This procedure is called \textit{canceling $h$-dipole} with respect to the color set $\{i_1,i_2,\cdots,i_h\}$. Thus, we obtain the gem $\Gamma_1$ of $M$ from $\Gamma$ after the cancellation of this $h$-dipole. We can also obtain $\Gamma$ from $\Gamma_1$ by following the reverse procedure which is called \textit{adding $h$-dipole} with respect to the color set $\{i_1,i_2,\cdots,i_h\}$. If $2\le h\le n-1$, then the graph $\Gamma_1$ is contracted if and only if $\Gamma$ is contracted. Note that if $\Gamma$ is a crystallization, then $h$ cannot be $1$ or $n$.

Let $\Lambda_1, \Lambda_2 \subset V(\Gamma)$ be such that the subgraphs $A_1$ and $A_2$ generated by $\Lambda_1$ and $\Lambda_2$, respectively, represent $n$-dimensional balls. Let there be an isomorphism $\Phi:A_1 \to A_2$ with identity color map such that $u$ and $\Phi(u)$ are joined by an edge of color $i$ for each $u\in \Lambda_1$, and $\Lambda_1$, $\Lambda_2$ lie in different components of $\Gamma_{\hat{i}}$. Consider a new  $(n+1)$-colored graph  $\Gamma^\prime$ obtained from $\Gamma$ as follows. Let $V(\Gamma^\prime)=V(\Gamma)\setminus (\Lambda_1 \cup \Lambda_2)$. For two vertices $p$ and $q$ in $V(\Gamma^\prime)$, if $p$ and $q$ are connected to $u\in \Lambda_1$ and $\Phi(u)$, respectively, by an edge of color $j\in \Delta_n \setminus  \{i\}$ in $\Gamma$, then $p$ and $q$ are joined by an edge of color $j$ in $\Gamma^\prime$. On the other hand, if $p$ and $q$ are joined by an edge of color $j\in \Delta_n$ in $\Gamma$, then $p$ and $q$ are joined by an edge of color $j$ in $\Gamma^\prime$. The process to obtain $\Gamma^\prime$ from $\Gamma$ is called a {\it polyhedral glue move} with respect to $(\Phi,\Lambda_1,\Lambda_2,i)$. From \cite{fg82i}, it is known that $\Gamma^\prime$ also represents $M$. If $\Lambda_1$ and $\Lambda_2$ are singleton sets, then this polyhedral glue move is called a {\it simple glue move} or {\it canceling $1$-dipole}, where $\Lambda_1$ and $\Lambda_2$ form a $1$-dipole with respect to the color $i$. For more details on dipole canceling/adding and polyhedral glue moves, one can see \cite{fg82i}. To obtain a crystallization from a gem of a manifold, one can apply polyhedral glue moves. A finite number of polyhedral glue moves converts a gem to a crystallization.

\noindent \textbf{Observation:} Let $(\Gamma,\gamma)$ be a crystallization of a closed $4$-manifold $M$. Let $\Lambda_1=\{v_1,v_2\}$ and $\Lambda_1^\prime=\{v_1^\prime,v_2^\prime\}$ such that $v_1$ and $v_2$ (resp. $v_1^\prime$ and $v_2^\prime$) are connected by a $k$-colored edge in $\Gamma$. Also, $v_1$ and $v_1^\prime$ (resp. $v_2$ and $v_2^\prime$) are connected by edges colored by $i$ and $j$. Let the vertices $v_1,v_2,v_1^\prime,v_2^\prime$ lie in four different components of $\Gamma_{\Delta_4 \setminus \{i,j,k\}}$, and $\Lambda_1,\Lambda_1^\prime$ lie in different components of $\Gamma_{\Delta_4 \setminus \{i,j\}}$. Since $\Lambda_1$ and $\Lambda_1^\prime$ are isomorphic, we have an isomorphism $\Phi$ such that $\Phi(v_1)=v_1^\prime$ and $\Phi(v_2)=v_2^\prime$. Thus, we have that $v_1$ and $v_1^\prime$ (resp. $v_2$ and $v_2^\prime$) form a $2$-dipole with respect to the color set $\{i,j\}$. We choose one of these two $2$-dipoles randomly. Removing the $2$-dipole involving vertices $v_1$ and $v_1^\prime$, we get a crystallization, say $\Gamma_1$, of $M$. Since $v_1,v_2$ (resp. $v_1^\prime,v_2^\prime$) are incident to a $k$-colored edge in $\Gamma$ and the vertices $v_1,v_2,v_1^\prime,v_2^\prime$ lie in four different components of $\Gamma_{\Delta_4 \setminus \{i,j,k\}}$, the vertices $v_2$ and $v_2^\prime$ form a $3$-dipole with respect to the color set $\{i,j,k\}$ in $\Gamma_1$. Removing this $3$-dipole from $\Gamma_1$, we get a crystallization $\Gamma_2$ of $M$. Instead of first removing $2$-dipole and then the induced $3$-dipole, we can directly obtain $\Gamma_2$ from $\Gamma$ by deleting $\Lambda_1,\Lambda_1^\prime$ from $\Gamma$ and following the procedure as in a polyhedral glue move. We will call this \textit{move} with respect to $(\Phi,\Lambda_1,\Lambda_1^\prime,\{i,j\})$.

\subsection{Regular Genus of closed PL $n$-manifolds}\label{sec:genus}
For a closed connected surface, its regular genus is simply its genus. However, for
closed connected PL $n$-manifolds ($n \geq 3$), the regular genus is defined as follows.
From \cite{fg82, g81}, it is known that if $(\Gamma,\gamma)$ is a bipartite (resp. non-bipartite) $(n+1)$-regular colored graph that represents a closed connected orientable (resp. non-orientable) PL $n$-manifold $M$, then for each cyclic permutation $\varepsilon=(\varepsilon_0,\dots,\varepsilon_n)$ of $\Delta_n$, there exists a regular embedding of $\Gamma$ into an orientable (resp. non-orientable) surface $S$. A {\it regular embedding} is an embedding where each region is bounded by a bi-colored cycle with colors $\varepsilon_i,\varepsilon_{i+1}$ for some $i$ (addition is modulo $n + 1$). Moreover, the Euler characteristic $\chi_\varepsilon(\Gamma)$ of the orientable (resp. non-orientable) surface  $S$ satisfies
$$\chi_\varepsilon(\Gamma)=\sum_{i \in \mathbb{Z}_{n+1}}g_{\varepsilon_i\varepsilon_{i+1}} + (1-n)\frac{|V(\Gamma)|}{2},$$ 
and thus the genus (resp. half of the genus) $\rho_ \varepsilon(\Gamma)$ of $S$ satisfies
$$\rho_ \varepsilon(\Gamma)=1-\frac{\chi_\varepsilon(\Gamma)}{2}.$$
The regular genus $\rho(\Gamma)$ of $(\Gamma,\gamma)$ is defined as
$$\rho(\Gamma)= \min \{\rho_{\varepsilon}(\Gamma) \ | \  \varepsilon \ \mbox{ is a cyclic permutation of } \ \Delta_n\}.$$
The {\it regular genus} of $M$ is defined as 
$$\mathcal G(M) = \min \{\rho(\Gamma) \ | \  (\Gamma,\gamma) \mbox{ represents } M\}.$$
	
A manifold of dimension $n$ with regular genus $0$ is characterized as $\mathbb{S}^n$\cite{fg82}. Some recent works on the regular genus can be found in the following articles \cite{bb21, bc17}.
The following result gives a lower bound for the regular genus of a closed connected PL $4$-manifold.

\begin{proposition}[\cite{bc17}]\label{lbrg}
  Let $M$ be a closed connected PL $4$-manifold with $rk(\pi_1(M))=m$. Then $\mathcal G(M)\ge 2\chi(M)+5m-4$.
\end{proposition}

Davis and Januszkiewicz introduced the concept of the small cover over a simple polytope in \cite{dj91}. As discussed in \cite[Remark, p. 421]{dj91}, throughout this article we consider all small covers and maps within the PL category. Let us state some terminologies and results concerning small covers.

\subsection{Small Cover}\label{smallcover}
A {\it simple $n$-polytope} is a convex polytope such that exactly $n$ codimension-one faces meet at each vertex \cite{bp02}. For example, in platonic solids, a tetrahedron, cube, and dodecahedron are simple $3$-polytopes, while octahedron and icosahedron are not simple. Let $\rho$ be the standard action of $\Z$ on $\mathbb{R}^n$. A $\Z$ action $\eta$ on an $n$-dimensional manifold $M^n$ is called a {\it locally standard action} if for each $x\in M^n$, there exists an automorphism $\theta_x$ of $\Z$, a $\Z$-stable open neighborhood $U_x$ of $x$, and a $\Z$-stable open set $V_x$ in $\mathbb{R}^n$ such that $U_x$ and $V_x$ are $\theta_x$-equivariantly homeomorphic. That is, there is a homeomorphism $f_x:U_x\to V_x$ such that $$f_x(\eta(g,u))=\rho(\theta_x(g),f_x(u)).$$ Further, if the orbit space of this action $\eta$ is a simple convex $n$-polytope $P^n$, then we say that $M^n$ is a {\it small cover} over $P^n$. Therefore, we have a projection map $\pi: M^n \to P^n$ such that $\pi(x)$ is the orbit class of $x$, for all $x\in M^n$.

Given a simple $n$-polytope $P^n$, let $\mathcal F(P^n)$ denote the set of $(n-1)$-faces of $P^n$. A function $$\lambda:\mathcal{F}(P^n) \to \Z$$ is called a {\it $\mathbb Z_2$-characteristic function} if, for each vertex $v=\bigcap_{i=1}^{n} F_i,$ the vectors $\lambda(F_i),\ 1\le i \le n$, form a basis of $\Z$, where $F_i\in \mathcal F (P^n)$. The vector $\lambda(F)$ is called the {\it $\mathbb{Z}_2$-characteristic vector} of $F,$ where $F\in \mathcal F (P^n).$ Let $G_F$ be the $l$-dimensional subspace generated by $\lambda(F_i),\ 1\le i\le l$, where $F=\bigcap_{i=1}^{l} F_i,\ F_i\in \mathcal F (P^n)$, a face of codimension-$l$. Define an equivalence relation on $\Z\times P^n$ as $$(g_1,p) \sim (g_2,p) \iff  \begin{cases}
g_1=g_2 & \text{if} \ p\in \text{int}(P^n)\\
g_1+g_2 \in G_{F_p} & \text{if} \ p\in \partial(P^n)
\end{cases},$$ where $F_p$ is the unique face containing $p$ in its relative interior. Let us denote the manifold $(\Z \times P^n)/ \sim $ by $M^n(\lambda)$. The $\Z$-action $\eta$ on $M^n(\lambda)$ defined as $\eta(g,(g_1,p))=(g+g_1,p)$ is a locally standard action and its orbit space is $P^n$. Therefore, $M^n(\lambda)$ is a small cover over $P^n.$ 

Let $ P^m $ and $ P^n $ be $ m $- and $ n $-polytopes, respectively, with $\mathbb{Z}_2$-characteristic functions $\lambda_m: \mathcal{F}(P^m) \to \mathbb{Z}_2^m$ and $\lambda_n: \mathcal{F}(P^n) \to \mathbb{Z}_2^n$. The set of $(m+n-1)$-faces of $ P^m \times P^n $ is given by $\mathcal{F}(P^m \times P^n) = \{F \times P^n, P^m \times F' \mid F \in \mathcal{F}(P^m), F' \in \mathcal{F}(P^n)\}$. Define a $\mathbb{Z}_2$-characteristic function $\lambda: \mathcal{F}(P^m \times P^n) \to \mathbb{Z}_2^{m+n}$ by $\lambda(F \times P^n) = (\lambda_m(F), \mathbf{0})$ and $\lambda(P^m \times F') = (\mathbf{0}, \lambda_n(F'))$. Then, the small cover $ M^{m+n}(\lambda) $ over $ P^m \times P^n $ is the product $ M^m(\lambda_m) \times M^n(\lambda_n) $.

Let $M^n_1$ and $M^n_2$ be two small covers over $P^n$. The small covers $M^n_1$ and $M^n_2$ are called {\it D-J equivalent} if there exists a $\theta$-equivariant homeomorphism $f:M^n_1\to M^n_2$ covering the identity on $P^n$, where $\theta$ is an automorphism of $\Z$. In short, the following diagram commutes.

 \begin{figure}[h!]
\tikzstyle{ver}=[]
\tikzstyle{edge} = [draw,thick,-]
    \centering
\begin{tikzpicture}[scale=0.6]
\foreach \x/\y/\z in
{-2.5/1.5/M_1^n,2.5/1.5/M_2^n,-2.5/-1.5/P^n,2.5/-1.5/P^n,0/2/f,0/-1/Id,-3/0/\pi_1,3/0/\pi_2}
{\node[ver] () at (\x,\y){$\z$};}
\draw [->](-1.8,1.5)--(1.8,1.5);
\draw [->](-2.5,1.1)--(-2.5,-1.1);
\draw [->](-1.8,-1.5)--(1.8,-1.5);
\draw [->](2.5,1.1)--(2.5,-1.1);
\end{tikzpicture}
\end{figure}

\noindent It is evident that two small covers $M^n(\lambda_1)$ and $M^n(\lambda_2)$ are D-J equivalent if and only if there exists an automorphism $\theta$ of $\Z$ such that $\lambda_2=\theta \circ \lambda_1$. If $M^n$ is a small cover over $P^n$, then there exists a $\mathbb{Z}_2$-characteristic function $\lambda:\mathcal F(P^n)\to \Z$ such that $M^n(\lambda)$ and $M^n$ are equivariantly homeomorphic covering the identity on $P^n$. For more on D-J equivalence and related concepts one can refer to \cite{ccl07,c08}.

\section{Genus-minimal crystallizations of $\mathbb{S}^2\times \mathbb{S}^1\times \mathbb{S}^1$ and $\mathbb{S}^1\times \mathbb{S}^1\times \mathbb{S}^1\times \mathbb{S}^1$}\label{Section:3}

 \begin{figure}
\tikzstyle{ver}=[]
\tikzstyle{verti}=[circle, draw, fill=black!100, inner sep=0pt, minimum width=2.5pt]
\tikzstyle{edge} = [draw,thick,-]
    \centering
\begin{tikzpicture}[scale=0.5]

\foreach \x/\y/\z in
{2/-2/w2,2/0/z2,2/2.5/y2,2/4.5/x2,-1/-2/w1,-1/0/z1,-1/2.5/y1,-1/4.5/x1}
{\node[verti] (\z) at (\x,\y){};}

\foreach \x/\y/\z in
{2.5/-2.5/v_7,2.5/0.5/v_5,2.5/3/v_3,2.5/5/v_1,-1.5/-2.5/v_6,-1.5/0.5/v_4,-1.5/3/v_2,-1.5/5/v_0}
{\node[ver] () at (\x,\y){\tiny{$\z$}};}

\draw[edge] plot [smooth,tension=1.5] coordinates{(y1) (0.5,1.9) (y2)};

\draw[edge] plot [smooth,tension=1.5] coordinates{(z1) (0.5,0.6) (z2)};

\draw [edge] plot [smooth,tension=1.5] coordinates{(x1) (-2.4,1) (w1)};
\draw [edge] plot [smooth,tension=1.5] coordinates{(x2) (3.4,1) (w2)};

\draw[edge,dotted] plot [smooth,tension=1.5] coordinates{(x1) (0.5,3.9) (x2)};
\draw[edge,dotted] plot [smooth,tension=1.5] coordinates{(y1) (0.5,3.1) (y2)};
\path[edge,dotted] (z1)--(w1);
\path[edge,dotted] (z2)--(w2);
\path[edge,dashed] (x1)--(y1);
\path[edge,dashed] (x2)--(y2);

\draw[line width=2pt, line cap=rectangle, dash pattern=on 1pt off 1] (y1) -- (z1);

\draw[line width=2pt, line cap=rectangle, dash pattern=on 1pt off 1] (y2) -- (z2);  

\draw [line width=2pt, line cap=rectangle, dash pattern=on 1pt off 1] plot [smooth,tension=1.5] coordinates{(x1) (0.5,5.1) (x2)};

\draw [edge,dashed] plot [smooth,tension=1.5] coordinates{(z1) (0.5,-0.6) (z2)};

\draw [edge,dashed] plot [smooth,tension=1.5] coordinates{(w1) (0.5,-1.4) (w2)};

\draw [line width=2pt, line cap=rectangle, dash pattern=on 1pt off 1] plot [smooth,tension=1.5] coordinates{(w1) (0.5,-2.6) (w2)};

\end{tikzpicture}
\caption{Standard $8$-vertex crystallization of $\mathbb{S}^2\times \mathbb{S}^1$.}\label{fig:C1}
\end{figure}
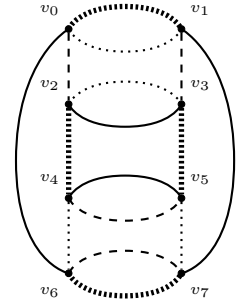

Let $(\Gamma,\gamma)$ be a $2p$-vertex crystallization of the $3$-manifold $M$. Let us name the vertices of $\Gamma$ by $v_0,v_1,\cdots,v_{2p-1}$. Since $\Delta_3$ is the color set, we denote $(\Gamma,\gamma)$ by $\Gamma(0,1,2,3)$. For $a,b,c\in \Delta_4$, by $\Gamma(a,b,c,\_)$, we mean a $3$-regular colored graph isomorphic to $\Gamma(0,1,2,\_)$, where the vertex map is the identity map and color map $\mathcal C:\{0,1,2\}\to \{a,b,c\}$ is defined as $\mathcal C(0)=a$, $\mathcal C(1)=b$, $\mathcal C(2)=c$. Since $\Gamma(0,1,2,3)$ is a crystallization, $\Gamma(a,b,c,\_)$ is connected and is a gem of $\mathbb{S}^2$. Similarly, $\Gamma(a,b,\_,d)$, $\Gamma(a,\_,c,d)$ and $\Gamma(\_,b,c,d)$ are defined. Let us denote $\Gamma(4,0,1,\_), \Gamma(4,0,\_,3), \Gamma(4,\_,2,3), \Gamma(\_,1,2,3)$ and $\Gamma(0,1,2\_)$ by $A, B, C, D$ and $E$, respectively. We will denote the vertex $v_i$ of $S$ by $v_i^S$, for $0\le i\le 2p-1$ and $S\in \{A,B,C,D,E\}$. It is well-known that we can construct a gem of $M\times I$ using $\Gamma(0,1,2,3)$ (Figure \ref{fig:Ga}). In Figure \ref{fig:G}, by saying that $A$ is joined with $B$ by a $2$-colored edge, we mean that $v_i^A$ is connected to $v_i^B$ by a $2$-colored edge, for $0\le i\le 2p-1$. Taking two copies of this gem of $M\times I$ and joining the boundary vertices appropriately, we can construct a gem of $M\times \mathbb{S}^1$ (Figure \ref{fig:Gb}). Since $E$ represents $\mathbb{S}^2$, it represents $4$-dimensional ball $\mathbb{B}^4$. Thus, we can apply the polyhedral glue move $(\Phi,\Lambda,\Lambda^\prime,3)$ on this gem of $M\times \mathbb{S}^1$, where $\Lambda$ and $\Lambda^\prime$ are the sets of vertices of $E$ and $E^\prime$, respectively, and $\Phi(v_i^E)=v_i^{E^\prime}$ for $0\le i\le 2p-1$. After applying this polyhedral glue move, we get another gem of $M\times \mathbb{S}^1$ (Figure \ref{fig:Gc}). We denote this gem by $G$. We fix an edge type corresponding to every color of $\Delta_4$ as in Figure \ref{fig:G}.

\begin{figure}[ht]
\tikzstyle{ver}=[]
\tikzstyle{edge} = [draw,thick,-]
\centering
    
\begin{subfigure}{0.9\textwidth}

\begin{tikzpicture}[scale=1]
\node[ver] () at (-6,0){$\Gamma(0,1,2,\_)$}; 
\node[ver] () at (-3,0){$\Gamma(\_,1,2,3)$};
\node[ver] () at (0,0){$\Gamma(4,\_,2,3)$};
\node[ver] () at (3,0){$\Gamma(4,0,\_,3)$};
\node[ver] () at (6,0){$\Gamma(4,0,1,\_)$};
\node[ver] () at (-6,0.5){$E$}; 
\node[ver] () at (-3,0.5){$D$};
\node[ver] () at (0,0.5){$C$};
\node[ver] () at (3,0.5){$B$};
\node[ver] () at (6,0.5){$A$};

\path[edge,dotted] (1,0)--(2,0);
\path[edge] (4,0)--(5,0);
\draw [line width=2pt, line cap=round, dash pattern=on 0pt off 1.7\pgflinewidth]  (-5,0) -- (-4,0);

\draw[line width=2pt, line cap=rectangle, dash pattern=on 1pt off 1] (-2,0) -- (-1,0);
\end{tikzpicture}

\caption{A gem of $M\times I$.}\label{fig:Ga}

\end{subfigure}
\centering
\begin{subfigure}{0.9\textwidth}
\begin{tikzpicture}[scale=1]

\begin{scope}[shift={(0,1.5)}]
    \node[ver] (E) at (-6,0){$\Gamma(0,1,2,\_)$}; 
\node[ver] (D) at (-3,0){$\Gamma(\_,1,2,3)$};
\node[ver] (C) at (0,0){$\Gamma(4,\_,2,3)$};
\node[ver] (B) at (3,0){$\Gamma(4,0,\_,3)$};
\node[ver] (A) at (6,0){$\Gamma(4,0,1,\_)$};
\node[ver] () at (-6,0.5){$E$}; 
\node[ver] () at (-3,0.5){$D$};
\node[ver] () at (0,0.5){$C$};
\node[ver] () at (3,0.5){$B$};
\node[ver] () at (6,0.5){$A$};

\path[edge,dotted] (1,0)--(2,0);
\path[edge] (4,0)--(5,0);
\draw [line width=2pt, line cap=round, dash pattern=on 0pt off 1.7\pgflinewidth]  (-5,0) -- (-4,0);

\draw[line width=2pt, line cap=rectangle, dash pattern=on 1pt off 1] (-2,0) -- (-1,0);
\end{scope}

\begin{scope}[shift={(0,0)}]
  \node[ver] (E') at (-6,0){$\Gamma(0,1,2,\_)$}; 
\node[ver] (D') at (-3,0){$\Gamma(\_,1,2,3)$};
\node[ver] (C') at (0,0){$\Gamma(4,\_,2,3)$};
\node[ver] (B') at (3,0){$\Gamma(4,0,\_,3)$};
\node[ver] (A') at (6,0){$\Gamma(4,0,1,\_)$};
\node[ver] () at (-6,-0.5){$E^\prime$}; 
\node[ver] () at (-3,-0.5){$D^\prime$};
\node[ver] () at (0,-0.5){$C^\prime$};
\node[ver] () at (3,-0.5){$B^\prime$};
\node[ver] () at (6,-0.5){$A^\prime$};

\path[edge,dotted] (1,0)--(2,0);
\path[edge] (4,0)--(5,0);
\draw [line width=2pt, line cap=round, dash pattern=on 0pt off 1.7\pgflinewidth]  (-5,0) -- (-4,0);

\draw[line width=2pt, line cap=rectangle, dash pattern=on 1pt off 1] (-2,0) -- (-1,0);  
\end{scope}
\path[edge,dashed] (A)--(A');
\path[edge,dashed] (E)--(E');
\end{tikzpicture}
\caption{A gem of $M\times \mathbb{S}^1$.}\label{fig:Gb}
\end{subfigure}
\centering
\begin{subfigure}{0.75\textwidth}
\begin{tikzpicture}[scale=1]

\begin{scope}[shift={(0,1.5)}]
   
\node[ver] (D) at (-3,0){$\Gamma(\_,1,2,3)$};
\node[ver] (C) at (0,0){$\Gamma(4,\_,2,3)$};
\node[ver] (B) at (3,0){$\Gamma(4,0,\_,3)$};
\node[ver] (A) at (6,0){$\Gamma(4,0,1,\_)$};
 
\node[ver] () at (-3,0.5){$D$};
\node[ver] () at (0,0.5){$C$};
\node[ver] () at (3,0.5){$B$};
\node[ver] () at (6,0.5){$A$};

\path[edge,dotted] (1,0)--(2,0);
\path[edge] (4,0)--(5,0);

\draw[line width=2pt, line cap=rectangle, dash pattern=on 1pt off 1] (-2,0) -- (-1,0);
\end{scope}

\begin{scope}[shift={(0,0)}]
 
\node[ver] (D') at (-3,0){$\Gamma(\_,1,2,3)$};
\node[ver] (C') at (0,0){$\Gamma(4,\_,2,3)$};
\node[ver] (B') at (3,0){$\Gamma(4,0,\_,3)$};
\node[ver] (A') at (6,0){$\Gamma(4,0,1,\_)$};
\node[ver] () at (-3,-0.5){$D^\prime$};
\node[ver] () at (0,-0.5){$C^\prime$};
\node[ver] () at (3,-0.5){$B^\prime$};
\node[ver] () at (6,-0.5){$A^\prime$};

\path[edge,dotted] (1,0)--(2,0);
\path[edge] (4,0)--(5,0);

\draw[line width=2pt, line cap=rectangle, dash pattern=on 1pt off 1] (-2,0) -- (-1,0);  
\end{scope}

\draw [line width=2pt, line cap=round, dash pattern=on 0pt off 1.7\pgflinewidth]  (D) -- (D');
\path[edge,dashed] (A)--(A');

\end{tikzpicture}
\caption{The gem $G$ of $M\times \mathbb{S}^1$.}\label{fig:Gc}
\end{subfigure}
\centering
\begin{subfigure}{0.6\textwidth}
\begin{tikzpicture}
    
\foreach \x/\y/\z in {1/0.3/0,3/0.3/1,5/0.3/2,7/0.3/3,9/0.3/4}
{\node[ver] () at (\x,\y){$\z$};}
\path[edge,dashed] (6.5,0) -- (7.5,0);

\draw [line width=2pt, line cap=round, dash pattern=on 0pt off 1.7\pgflinewidth]  (8.5,0) -- (9.5,0);
\path[edge] (4.5,0) -- (5.5,0);
\draw[line width=2pt, line cap=rectangle, dash pattern=on 1pt off 1] (0.5,0) -- (1.5,0);

\path[edge,dotted] (2.5,0) -- (3.5,0);

\end{tikzpicture}
\end{subfigure}

\caption{Construction of the gem $G$ of $M\times \mathbb{S}^1$ using the crystallization of the $3$-manifold $M$.}\label{fig:G}
\end{figure}
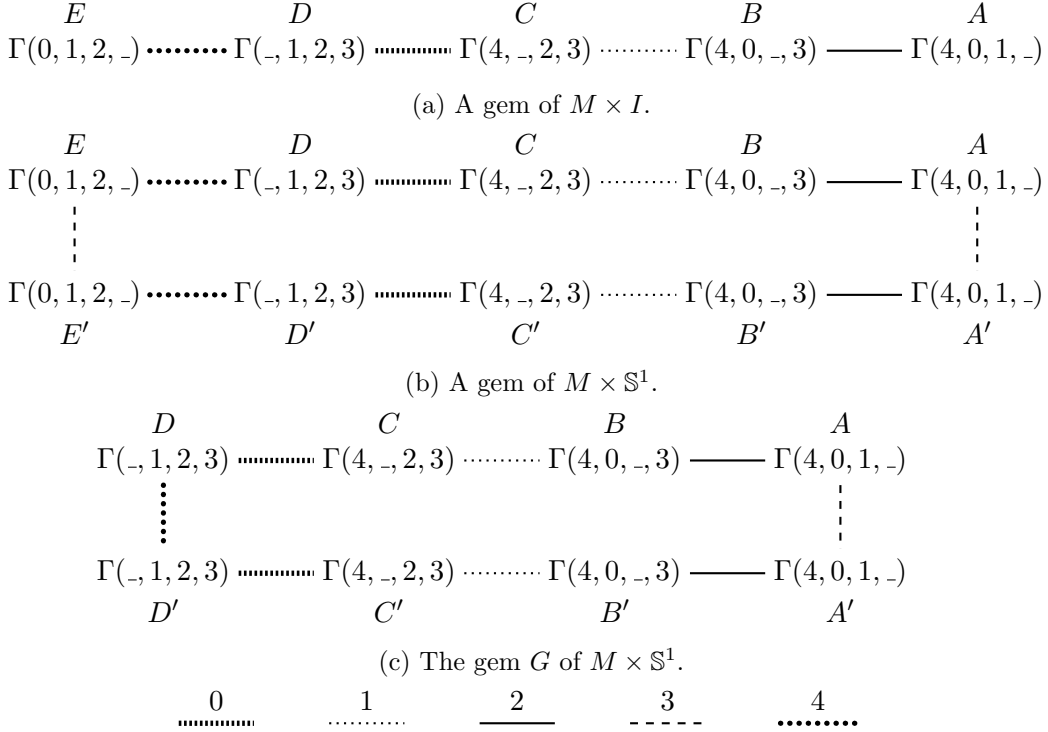

\begin{figure}[h!]
\begin{subfigure}{1\textwidth}
\tikzstyle{ver}=[]
\tikzstyle{verti}=[circle, draw, fill=black!100, inner sep=0pt, minimum width=2.5pt]
\tikzstyle{edge} = [draw,thick,-]
    \centering
\begin{tikzpicture}[scale=0.5]

\begin{scope}[shift={(-15,0)}]

\foreach \x/\y/\z in
{2/-2/w2,2/0/z2,2/2.5/y2,2/4.5/x2,-1/-2/w1,-1/0/z1,-1/2.5/y1,-1/4.5/x1}
{\node[verti] (\z) at (\x,\y){};}

\foreach \x/\y/\z in
{2.5/-2.5/v_7^{D^\prime},2.5/0.5/v_5^{D^\prime},2.5/3/v_3^{D^\prime},2.5/5/v_1^{D^\prime},-1.5/-2.5/v_6^{D^\prime},-1.5/0.5/v_4^{D^\prime},-1.5/3/v_2^{D^\prime},-1.5/5/v_0^{D^\prime}}
{\node[ver] () at (\x,\y){\tiny{$\z$}};}

\draw[edge] plot [smooth,tension=1.5] coordinates{(y1) (0.5,1.9) (y2)};

\draw[edge] plot [smooth,tension=1.5] coordinates{(z1) (0.5,0.6) (z2)};

\draw [edge] plot [smooth,tension=1.5] coordinates{(x1) (-2.4,1) (w1)};
\draw [edge] plot [smooth,tension=1.5] coordinates{(x2) (3.4,1) (w2)};

\draw[edge,dotted] plot [smooth,tension=1.5] coordinates{(x1) (0.5,3.9) (x2)};
\draw[edge,dotted] plot [smooth,tension=1.5] coordinates{(y1) (0.5,3.1) (y2)};
\path[edge,dotted] (z1)--(w1);
\path[edge,dotted] (z2)--(w2);
\path[edge,dashed] (x1)--(y1);
\path[edge,dashed] (x2)--(y2);

\draw [edge,dashed] plot [smooth,tension=1.5] coordinates{(z1) (0.5,-0.6) (z2)};

\draw [edge,dashed] plot [smooth,tension=1.5] coordinates{(w1) (0.5,-1.4) (w2)};

\end{scope}

\begin{scope}[shift={(-8,0)}]
\foreach \x/\y/\z in
{2/-2/w2,2/0/z2,2/2.5/y2,2/4.5/x2,-1/-2/w1,-1/0/z1,-1/2.5/y1,-1/4.5/x1}
{\node[verti] (\z) at (\x,\y){};}

\foreach \x/\y/\z in
{2.5/-2.5/v_7^{C^\prime},2.5/0.5/v_5^{C^\prime},2.5/3/v_3^{C^\prime},2.5/5/v_1^{C^\prime},-1.5/-2.5/v_6^{C^\prime},-1.5/0.5/v_4^{C^\prime},-1.5/3/v_2^{C^\prime},-1.5/5/v_0^{C^\prime}}
{\node[ver] () at (\x,\y){\tiny{$\z$}};}

\draw[edge] plot [smooth,tension=1.5] coordinates{(y1) (0.5,1.9) (y2)};

\draw[edge] plot [smooth,tension=1.5] coordinates{(z1) (0.5,0.6) (z2)};

\draw [edge] plot [smooth,tension=1.5] coordinates{(x1) (-2.4,1) (w1)};
\draw [edge] plot [smooth,tension=1.5] coordinates{(x2) (3.4,1) (w2)};

\path[edge,dashed] (x1)--(y1);
\path[edge,dashed] (x2)--(y2);
\draw [edge,dashed] plot [smooth,tension=1.5] coordinates{(z1) (0.5,-0.6) (z2)};

\draw [edge,dashed] plot [smooth,tension=1.5] coordinates{(w1) (0.5,-1.4) (w2)};

\draw [line width=2pt, line cap=round, dash pattern=on 0pt off 1.7\pgflinewidth] (y1) -- (z1);

\draw [line width=2pt, line cap=round, dash pattern=on 0pt off 1.7\pgflinewidth] (y2) -- (z2);  

\draw [line width=2pt, line cap=round, dash pattern=on 0pt off 1.7\pgflinewidth] plot [smooth,tension=1.5] coordinates{(x1) (0.5,5.1) (x2)};

\draw [line width=2pt, line cap=round, dash pattern=on 0pt off 1.7\pgflinewidth] plot [smooth,tension=1.5] coordinates{(w1) (0.5,-2.6) (w2)};    
\end{scope}

\begin{scope}[shift={(-1,0)}]
\foreach \x/\y/\z in
{2/-2/w2,2/0/z2,2/2.5/y2,2/4.5/x2,-1/-2/w1,-1/0/z1,-1/2.5/y1,-1/4.5/x1}
{\node[verti] (\z) at (\x,\y){};}

\foreach \x/\y/\z in
{2.5/-2.5/v_7^{B^\prime},2.5/0.5/v_5^{B^\prime},2.5/3/v_3^{B^\prime},2.5/5/v_1^{B^\prime},-1.5/-2.5/v_6^{B^\prime},-1.5/0.5/v_4^{B^\prime},-1.5/3/v_2^{B^\prime},-1.5/5/v_0^{B^\prime}}
{\node[ver] () at (\x,\y){\tiny{$\z$}};}

\draw[line width=2pt, line cap=rectangle, dash pattern=on 1pt off 1] plot [smooth,tension=1.5] coordinates{(x1) (0.5,3.9) (x2)};
\draw[line width=2pt, line cap=rectangle, dash pattern=on 1pt off 1] plot [smooth,tension=1.5] coordinates{(y1) (0.5,3.1) (y2)};
\draw[line width=2pt, line cap=rectangle, dash pattern=on 1pt off 1] (z1)--(w1);
\draw[line width=2pt, line cap=rectangle, dash pattern=on 1pt off 1] (z2)--(w2);
\path[edge,dashed] (x1)--(y1);
\path[edge,dashed] (x2)--(y2);
\draw [edge,dashed] plot [smooth,tension=1.5] coordinates{(z1) (0.5,-0.6) (z2)};

\draw [edge,dashed] plot [smooth,tension=1.5] coordinates{(w1) (0.5,-1.4) (w2)};

\draw [line width=2pt, line cap=round, dash pattern=on 0pt off 1.7\pgflinewidth] (y1) -- (z1);

\draw [line width=2pt, line cap=round, dash pattern=on 0pt off 1.7\pgflinewidth] (y2) -- (z2);  

\draw [line width=2pt, line cap=round, dash pattern=on 0pt off 1.7\pgflinewidth] plot [smooth,tension=1.5] coordinates{(x1) (0.5,5.1) (x2)};

\draw [line width=2pt, line cap=round, dash pattern=on 0pt off 1.7\pgflinewidth] plot [smooth,tension=1.5] coordinates{(w1) (0.5,-2.6) (w2)};

\end{scope}

\begin{scope}[shift={(6,0)}]
\foreach \x/\y/\z in
{2/-2/w2,2/0/z2,2/2.5/y2,2/4.5/x2,-1/-2/w1,-1/0/z1,-1/2.5/y1,-1/4.5/x1}
{\node[verti] (\z) at (\x,\y){};}

\foreach \x/\y/\z in
{2.5/-2.5/v_7^{A^\prime},2.5/0.5/v_5^{A^\prime},2.5/3/v_3^{A^\prime},2.5/5/v_1^{A^\prime},-1.5/-2.5/v_6^{A^\prime},-1.5/0.5/v_4^{A^\prime},-1.5/3/v_2^{A^\prime},-1.5/5/v_0^{A^\prime}}
{\node[ver] () at (\x,\y){\tiny{$\z$}};}

\draw[edge,dotted] plot [smooth,tension=1.5] coordinates{(y1) (0.5,1.9) (y2)};

\draw[edge,dotted] plot [smooth,tension=1.5] coordinates{(z1) (0.5,0.6) (z2)};

\draw [edge,dotted] plot [smooth,tension=1.5] coordinates{(x1) (-2.4,1) (w1)};
\draw [edge,dotted] plot [smooth,tension=1.5] coordinates{(x2) (3.4,1) (w2)};

\draw [line width=2pt, line cap=round, dash pattern=on 0pt off 1.7\pgflinewidth] (y1) -- (z1);

\draw [line width=2pt, line cap=round, dash pattern=on 0pt off 1.7\pgflinewidth] (y2) -- (z2);  

\draw [line width=2pt, line cap=round, dash pattern=on 0pt off 1.7\pgflinewidth] plot [smooth,tension=1.5] coordinates{(x1) (0.5,5.1) (x2)};

\draw [line width=2pt, line cap=round, dash pattern=on 0pt off 1.7\pgflinewidth] plot [smooth,tension=1.5] coordinates{(w1) (0.5,-2.6) (w2)};   

\draw[line width=2pt, line cap=rectangle, dash pattern=on 1pt off 1] plot [smooth,tension=1.5] coordinates{(x1) (0.5,3.9) (x2)};
\draw[line width=2pt, line cap=rectangle, dash pattern=on 1pt off 1] plot [smooth,tension=1.5] coordinates{(y1) (0.5,3.1) (y2)};
\draw[line width=2pt, line cap=rectangle, dash pattern=on 1pt off 1] (z1)--(w1);
\draw[line width=2pt, line cap=rectangle, dash pattern=on 1pt off 1] (z2)--(w2);
   
\end{scope}

\end{tikzpicture}
\caption{}\label{fig:G1a}
\end{subfigure}

\begin{subfigure}{1\textwidth}
\tikzstyle{ver}=[]

\tikzstyle{verti}=[circle, draw, fill=black!100, inner sep=0pt, minimum width=2.5pt]
\tikzstyle{edge} = [draw,thick,-]
    \centering
\begin{tikzpicture}[scale=0.5]

\begin{scope}[shift={(-15,0)}]

\foreach \x/\y/\z in
{2/-2/w2,2/0/z2,2/2.5/y2,2/4.5/x2,-1/-2/w1,-1/0/z1,-1/2.5/y1,-1/4.5/x1}
{\node[verti] (\z) at (\x,\y){};}

\foreach \x/\y/\z in
{2.5/-2.5/v_7^{D^\prime},2.5/0.5/v_5^{B^\prime},2.5/3/v_3^{D^\prime},2.5/5/v_1^{D^\prime},-1.5/-2.5/v_6^{D^\prime},-1.5/0.5/v_4^{B^\prime},-1.5/3/v_2^{D^\prime},-1.5/5/v_0^{D^\prime}}
{\node[ver] () at (\x,\y){\tiny{$\z$}};}

\draw[edge] plot [smooth,tension=1.5] coordinates{(y1) (0.5,1.9) (y2)};

\draw [edge] plot [smooth,tension=1.5] coordinates{(x1) (-2.4,1) (w1)};
\draw [edge] plot [smooth,tension=1.5] coordinates{(x2) (3.4,1) (w2)};

\draw[edge,dotted] plot [smooth,tension=1.5] coordinates{(x1) (0.5,3.9) (x2)};
\draw[edge,dotted] plot [smooth,tension=1.5] coordinates{(y1) (0.5,3.1) (y2)};
\path[edge,dotted] (z1)--(w1);
\path[edge,dotted] (z2)--(w2);
\path[edge,dashed] (x1)--(y1);
\path[edge,dashed] (x2)--(y2);

\draw [edge,dashed] plot [smooth,tension=1.5] coordinates{(z1) (0.5,-0.6) (z2)};

\draw [edge,dashed] plot [smooth,tension=1.5] coordinates{(w1) (0.5,-1.4) (w2)};

\end{scope}

\begin{scope}[shift={(-8,0)}]
\foreach \x/\y/\z in
{2/-2/w2,2/0/z2,2/2.5/y2,2/4.5/x2,-1/-2/w1,-1/0/z1,-1/2.5/y1,-1/4.5/x1}
{\node[verti] (\z) at (\x,\y){};}

\foreach \x/\y/\z in
{2.5/-2.5/v_7^{A^\prime},2.5/0.5/v_5^{D},2.5/3/v_3^{C^\prime},2.5/5/v_1^{C^\prime},-1.5/-2.5/v_6^{A^\prime},-1.5/0.5/v_4^{D},-1.5/3/v_2^{C^\prime},-1.5/5/v_0^{C^\prime}}
{\node[ver] () at (\x,\y){\tiny{$\z$}};}

\draw[edge] plot [smooth,tension=1.5] coordinates{(y1) (0.5,1.9) (y2)};

\draw[edge] plot [smooth,tension=1.5] coordinates{(z1) (0.5,0.6) (z2)};

\draw [edge] plot [smooth,tension=1.5] coordinates{(x1) (-2.4,1) (w1)};
\draw [edge] plot [smooth,tension=1.5] coordinates{(x2) (3.4,1) (w2)};

\path[edge,dashed] (x1)--(y1);
\path[edge,dashed] (x2)--(y2);
\draw [edge,dashed] plot [smooth,tension=1.5] coordinates{(z1) (0.5,-0.6) (z2)};

\draw [line width=2pt, line cap=round, dash pattern=on 0pt off 1.7\pgflinewidth] (y1) -- (z1);

\draw [line width=2pt, line cap=round, dash pattern=on 0pt off 1.7\pgflinewidth] (y2) -- (z2);  

\draw [line width=2pt, line cap=round, dash pattern=on 0pt off 1.7\pgflinewidth] plot [smooth,tension=1.5] coordinates{(x1) (0.5,5.1) (x2)};

\draw [line width=2pt, line cap=round, dash pattern=on 0pt off 1.7\pgflinewidth] plot [smooth,tension=1.5] coordinates{(w1) (0.5,-2.6) (w2)};   
\end{scope}

\begin{scope}[shift={(-1,0)}]
\foreach \x/\y/\z in
{2/-2/w2,2/0/z2,2/2.5/y2,2/4.5/x2,-1/-2/w1,-1/0/z1,-1/2.5/y1,-1/4.5/x1}
{\node[verti] (\z) at (\x,\y){};}

\foreach \x/\y/\z in
{2.5/-2.5/v_7^{D^\prime},2.5/0.5/v_5^{B^\prime},2.5/3/v_3^{B^\prime},2.5/5/v_1^{A},-1.5/-2.5/v_6^{D^\prime},-1.5/0.5/v_4^{B^\prime},-1.5/3/v_2^{B^\prime},-1.5/5/v_0^{A}}
{\node[ver] () at (\x,\y){\tiny{$\z$}};}

\draw[line width=2pt, line cap=rectangle, dash pattern=on 1pt off 1] plot [smooth,tension=1.5] coordinates{(x1) (0.5,3.9) (x2)};
\draw[line width=2pt, line cap=rectangle, dash pattern=on 1pt off 1] plot [smooth,tension=1.5] coordinates{(y1) (0.5,3.1) (y2)};
\draw[line width=2pt, line cap=rectangle, dash pattern=on 1pt off 1] (z1)--(w1);
\draw[line width=2pt, line cap=rectangle, dash pattern=on 1pt off 1] (z2)--(w2);
\path[edge,dashed] (x1)--(y1);
\path[edge,dashed] (x2)--(y2);
\draw [edge,dashed] plot [smooth,tension=1.5] coordinates{(z1) (0.5,-0.6) (z2)};

\draw [edge,dashed] plot [smooth,tension=1.5] coordinates{(w1) (0.5,-1.4) (w2)};

\draw [line width=2pt, line cap=round, dash pattern=on 0pt off 1.7\pgflinewidth] (y1) -- (z1);

\draw [line width=2pt, line cap=round, dash pattern=on 0pt off 1.7\pgflinewidth] (y2) -- (z2);  

\draw [line width=2pt, line cap=round, dash pattern=on 0pt off 1.7\pgflinewidth] plot [smooth,tension=1.5] coordinates{(x1) (0.5,5.1) (x2)};

\end{scope}

\begin{scope}[shift={(6,0)}]
\foreach \x/\y/\z in
{2/-2/w2,2/0/z2,2/2.5/y2,2/4.5/x2,-1/-2/w1,-1/0/z1,-1/2.5/y1,-1/4.5/x1}
{\node[verti] (\z) at (\x,\y){};}

\foreach \x/\y/\z in
{2.5/-2.5/v_7^{A^\prime},2.5/0.5/v_5^{A^\prime},2.5/3/v_3^{A^\prime},2.5/5/v_1^{C^\prime},-1.5/-2.5/v_6^{A^\prime},-1.5/0.5/v_4^{A^\prime},-1.5/3/v_2^{A^\prime},-1.5/5/v_0^{C^\prime}}
{\node[ver] () at (\x,\y){\tiny{$\z$}};}

\draw[edge,dotted] plot [smooth,tension=1.5] coordinates{(y1) (0.5,1.9) (y2)};

\draw[edge,dotted] plot [smooth,tension=1.5] coordinates{(z1) (0.5,0.6) (z2)};

\draw [edge,dotted] plot [smooth,tension=1.5] coordinates{(x1) (-2.4,1) (w1)};
\draw [edge,dotted] plot [smooth,tension=1.5] coordinates{(x2) (3.4,1) (w2)};

\draw[line width=2pt, line cap=round, dash pattern=on 0pt off 1.7\pgflinewidth] (y1) -- (z1);

\draw[line width=2pt, line cap=round, dash pattern=on 0pt off 1.7\pgflinewidth] (y2) -- (z2);  

\draw [line width=2pt, line cap=round, dash pattern=on 0pt off 1.7\pgflinewidth] plot [smooth,tension=1.5] coordinates{(x1) (0.5,5.1) (x2)};

\draw [line width=2pt, line cap=round, dash pattern=on 0pt off 1.7\pgflinewidth] plot [smooth,tension=1.5] coordinates{(w1) (0.5,-2.6) (w2)};

\draw[line width=2pt, line cap=rectangle, dash pattern=on 1pt off 1] plot [smooth,tension=1.5] coordinates{(y1) (0.5,3.1) (y2)};
\draw[line width=2pt, line cap=rectangle, dash pattern=on 1pt off 1] (z1)--(w1);
\draw[line width=2pt, line cap=rectangle, dash pattern=on 1pt off 1] (z2)--(w2);
   
\end{scope}

\end{tikzpicture}
\caption{}\label{fig:G1b}
\end{subfigure}

\begin{subfigure}{1\textwidth}
\tikzstyle{ver}=[]
\tikzstyle{verti}=[circle, draw, fill=black!100, inner sep=0pt, minimum width=2.5pt]
\tikzstyle{edge} = [draw,thick,-]
    \centering
\begin{tikzpicture}[scale=0.5]

\begin{scope}[shift={(-15,0)}]

\foreach \x/\y/\z in
{2/0/w2,2/2.5/y2,2/4.5/x2,-1/0/w1,-1/2.5/y1,-1/4.5/x1}
{\node[verti] (\z) at (\x,\y){};}

\foreach \x/\y/\z in
{2.5/-0.5/v_5^{A^\prime},2.5/3/v_3^{D^\prime},2.5/5/v_1^{D^\prime},-1.5/-0.5/v_4^{A^\prime},-1.5/3/v_2^{D^\prime},-1.5/5/v_0^{D^\prime}}
{\node[ver] () at (\x,\y){\tiny{$\z$}};}

\draw[edge] plot [smooth,tension=1.5] coordinates{(y1) (0.5,1.9) (y2)};

\draw [edge] plot [smooth,tension=1.5] coordinates{(x1) (-2.4,2) (w1)};
\draw [edge] plot [smooth,tension=1.5] coordinates{(x2) (3.4,2) (w2)};

\draw[edge,dotted] plot [smooth,tension=1.5] coordinates{(x1) (0.5,3.9) (x2)};
\draw[edge,dotted] plot [smooth,tension=1.5] coordinates{(y1) (0.5,3.1) (y2)};

\path[edge,dashed] (x1)--(y1);
\path[edge,dashed] (x2)--(y2);

\draw [edge,dotted] plot [smooth,tension=1.5] coordinates{(w1) (0.5,-0.5) (w2)};

\end{scope}

\begin{scope}[shift={(-8,0)}]
\foreach \x/\y/\z in
{2/0/z2,2/2.5/y2,2/4.5/x2,-1/0/z1,-1/2.5/y1,-1/4.5/x1}
{\node[verti] (\z) at (\x,\y){};}

\foreach \x/\y/\z in
{2.5/0.5/v_5^{D},2.5/3/v_3^{C^\prime},2.5/5/v_7^{A},-1.5/0.5/v_4^{D},-1.5/3/v_2^{C^\prime},-1.5/5/v_6^{A}}
{\node[ver] () at (\x,\y){\tiny{$\z$}};}

\draw[edge] plot [smooth,tension=1.5] coordinates{(y1) (0.5,1.9) (y2)};

\draw[edge] plot [smooth,tension=1.5] coordinates{(z1) (0.5,0.6) (z2)};

\path[edge,dashed] (x1)--(y1);
\path[edge,dashed] (x2)--(y2);
\draw [edge,dashed] plot [smooth,tension=1.5] coordinates{(z1) (0.5,-0.6) (z2)};

\draw[line width=2pt, line cap=round, dash pattern=on 0pt off 1.7\pgflinewidth] (y1) -- (z1);

\draw[line width=2pt, line cap=round, dash pattern=on 0pt off 1.7\pgflinewidth] (y2) -- (z2);  

\draw [line width=2pt, line cap=round, dash pattern=on 0pt off 1.7\pgflinewidth] plot [smooth,tension=1.5] coordinates{(x1) (0.5,5.1) (x2)};

\end{scope}

\begin{scope}[shift={(-1,0)}]
\foreach \x/\y/\z in
{2/0/z2,2/2.5/y2,2/4.5/x2,-1/0/z1,-1/2.5/y1,-1/4.5/x1}
{\node[verti] (\z) at (\x,\y){};}

\foreach \x/\y/\z in
{2.5/0.5/v_7^{D},2.5/3/v_3^{B^\prime},2.5/5/v_1^{A},-1.5/0.5/v_6^{D},-1.5/3/v_2^{B^\prime},-1.5/5/v_0^{A}}
{\node[ver] () at (\x,\y){\tiny{$\z$}};}

\draw[line width=2pt, line cap=rectangle, dash pattern=on 1pt off 1] plot [smooth,tension=1.5] coordinates{(x1) (0.5,3.9) (x2)};
\draw[line width=2pt, line cap=rectangle, dash pattern=on 1pt off 1] plot [smooth,tension=1.5] coordinates{(y1) (0.5,3.1) (y2)};

\path[edge,dashed] (x1)--(y1);
\path[edge,dashed] (x2)--(y2);
\draw [edge,dashed] plot [smooth,tension=1.5] coordinates{(z1) (0.5,-0.6) (z2)};

\draw[line width=2pt, line cap=round, dash pattern=on 0pt off 1.7\pgflinewidth] (y1) -- (z1);

\draw[line width=2pt, line cap=round, dash pattern=on 0pt off 1.7\pgflinewidth] (y2) -- (z2);  

\draw [line width=2pt, line cap=round, dash pattern=on 0pt off 1.7\pgflinewidth] plot [smooth,tension=1.5] coordinates{(x1) (0.5,5.1) (x2)};

\end{scope}

\begin{scope}[shift={(6,2)}]
\foreach \x/\y/\z in
{2/-2/w2,2/0/z2,2/2.5/y2,-1/-2/w1,-1/0/z1,-1/2.5/y1}
{\node[verti] (\z) at (\x,\y){};}

\foreach \x/\y/\z in
{2.5/-2.5/v_1^{D^\prime},2.5/0.5/v_5^{A^\prime},2.5/3/v_3^{A^\prime},-1.5/-2.5/v_0^{D^\prime},-1.5/0.5/v_4^{A^\prime},-1.5/3/v_2^{A^\prime}}
{\node[ver] () at (\x,\y){\tiny{$\z$}};}

\draw[edge,dotted] plot [smooth,tension=1.5] coordinates{(y1) (0.5,1.9) (y2)};

\draw[edge,dotted] plot [smooth,tension=1.5] coordinates{(z1) (0.5,0.6) (z2)};

\draw[line width=2pt, line cap=round, dash pattern=on 0pt off 1.7\pgflinewidth] (y1) -- (z1);

\draw[line width=2pt, line cap=round, dash pattern=on 0pt off 1.7\pgflinewidth] (y2) -- (z2);

\draw [edge,dotted] plot [smooth,tension=1.5] coordinates{(w1) (0.5,-2.6) (w2)};

\draw[line width=2pt, line cap=rectangle, dash pattern=on 1pt off 1] plot [smooth,tension=1.5] coordinates{(y1) (0.5,3.1) (y2)};
\draw[line width=2pt, line cap=rectangle, dash pattern=on 1pt off 1] (z1)--(w1);
\draw[line width=2pt, line cap=rectangle, dash pattern=on 1pt off 1] (z2)--(w2);
   
\end{scope}

\end{tikzpicture}
\caption{}\label{fig:G1c}
\end{subfigure}

\begin{subfigure}{1\textwidth}
\tikzstyle{ver}=[]
\tikzstyle{verti}=[circle, draw, fill=black!100, inner sep=0pt, minimum width=2.5pt]
\tikzstyle{edge} = [draw,thick,-]
    \centering
\begin{tikzpicture}[scale=0.5]

 \begin{scope}[shift={(-1,0)}]
\foreach \x/\y/\z in
{2/-2/w2,2/0/z2,2/2.5/y2,2/4.5/x2,-1/-2/w1,-1/0/z1,-1/2.5/y1,-1/4.5/x1}
{\node[verti] (\z) at (\x,\y){};}

\foreach \x/\y/\z in
{2.5/-2.6/v_3^{B^\prime},2.5/0.5/v_3^{C^\prime},2.5/2.9/v_3^{D^\prime},2.5/5/v_3^{A^\prime},-1.5/-2.6/v_2^{B^\prime},-1.5/0.5/v_2^{C^\prime},-1.5/2.9/v_2^{D^\prime},-1.5/5/v_2^{A^\prime}}
{\node[ver] () at (\x,\y){\tiny{$\z$}};}

\draw[edge] plot [smooth,tension=1.5] coordinates{(y1) (0.5,1.9) (y2)};

\draw[edge] plot [smooth,tension=1.5] coordinates{(z1) (0.5,0.6) (z2)};

\draw [edge] plot [smooth,tension=1.5] coordinates{(x1) (-2.4,1) (w1)};
\draw [edge] plot [smooth,tension=1.5] coordinates{(x2) (3.4,1) (w2)};

\draw[line width=2pt, line cap=rectangle, dash pattern=on 1pt off 1] (y1) -- (z1);

\draw[line width=2pt, line cap=rectangle, dash pattern=on 1pt off 1] (y2) -- (z2);

\draw [line width=2pt, line cap=rectangle, dash pattern=on 1pt off 1] plot [smooth,tension=1.5] coordinates{(x1) (0.5,5.1) (x2)};

\draw [line width=2pt, line cap=rectangle, dash pattern=on 1pt off 1] plot [smooth,tension=1.5] coordinates{(w1) (0.5,-2.6) (w2)};   

\draw[edge,dotted] plot [smooth,tension=1.5] coordinates{(x1) (0.5,3.9) (x2)};
\draw[edge,dotted] plot [smooth,tension=1.5] coordinates{(y1) (0.5,3.1) (y2)};
\draw[edge,dotted] (z1)--(w1);
\draw[edge,dotted] (z2)--(w2);
   
\end{scope}

    \begin{scope}[shift={(6,0)}]
\foreach \x/\y/\z in
{2/-2/w2',2/0/z2',2/2.5/y2',2/4.5/x2',-1/-2/w1',-1/0/z1',-1/2.5/y1',-1/4.5/x1'}
{\node[verti] (\z) at (\x,\y){};}

\foreach \x/\y/\z in
{2.4/-2.4/v_7^{A},2.5/0.5/v_5^{A},2.5/3/v_3^{A},2.5/5/v_1^{A},-1.4/-2.4/v_6^{A},-1.5/0.6/v_4^{A},-1.5/3/v_2^{A},-1.5/5.1/v_0^{A}}
{\node[ver] () at (\x,\y){\tiny{$\z$}};}

\draw[edge,dotted] plot [smooth,tension=1.5] coordinates{(y1') (0.5,1.9) (y2')};

\draw[edge,dotted] plot [smooth,tension=1.5] coordinates{(z1') (0.5,0.6) (z2')};

\draw [edge,dotted] plot [smooth,tension=1.5] coordinates{(x1') (-2.4,1) (w1')};
\draw [edge,dotted] plot [smooth,tension=1.5] coordinates{(x2') (3.4,1) (w2')};

\draw[line width=2pt, line cap=round, dash pattern=on 0pt off 1.7\pgflinewidth] (y1') -- (z1');

\draw[line width=2pt, line cap=round, dash pattern=on 0pt off 1.7\pgflinewidth] (y2') -- (z2');  

\draw [line width=2pt, line cap=round, dash pattern=on 0pt off 1.7\pgflinewidth] plot [smooth,tension=1.5] coordinates{(x1') (0.5,5.1) (x2')};

\draw [line width=2pt, line cap=round, dash pattern=on 0pt off 1.7\pgflinewidth] plot [smooth,tension=1.5] coordinates{(w1') (0.5,-2.6) (w2')};   

\draw[line width=2pt, line cap=rectangle, dash pattern=on 1pt off 1] plot [smooth,tension=1.5] coordinates{(x1') (0.5,3.9) (x2')};
\draw[line width=2pt, line cap=rectangle, dash pattern=on 1pt off 1] plot [smooth,tension=1.5] coordinates{(y1') (0.5,3.1) (y2')};
\draw[line width=2pt, line cap=rectangle, dash pattern=on 1pt off 1] (z1')--(w1');
\draw[line width=2pt, line cap=rectangle, dash pattern=on 1pt off 1] (z2')--(w2');
   
\end{scope}

\foreach \x/\y in
{x1/y1',x2/y2',y1/z1',y2/z2',z1/w1',z2/w2'}
{\path[edge,dashed] (\x)--(\y);} 

\draw[edge,dashed] plot [smooth,tension=0.5] coordinates{(w1) (-4,-1) (-3,5.5) (1,6) (x1')};

\draw[edge,dashed] plot [smooth,tension=0.5] coordinates{(w2) (7,-3) (10,-2) (10,3) (x2')};

\node[ver] (D) at (-9,-1){$D=\Gamma(\_,1,2,3)$};

\node[ver] (C) at (-9,1){$C=\Gamma(4,\_,2,3)$};

\node[ver] (B) at (-9,3){$B=\Gamma(4,0,\_,3)$};

\node[ver] (A) at (-9,5){$A=\Gamma(4,0,1,\_)$};

\path[edge] (A)--(B);
\path[edge,dotted] (B)--(C);
\draw[line width=2pt, line cap=rectangle, dash pattern=on 1pt off 1] (C)--(D);
\draw [line width=2pt, line cap=round, dash pattern=on 0pt off 1.7\pgflinewidth] (D)--(-4.5,-1);

\end{tikzpicture}
\caption{The crystallization $G_1^\prime$ of $\mathbb{S}^2\times \mathbb{S}^1 \times \mathbb{S}^1$ with $40$ vertices.}\label{fig:G1d}
\end{subfigure}

\caption{Crystallization moves on the gem $G_1$.}\label{fig:G1}
\end{figure}
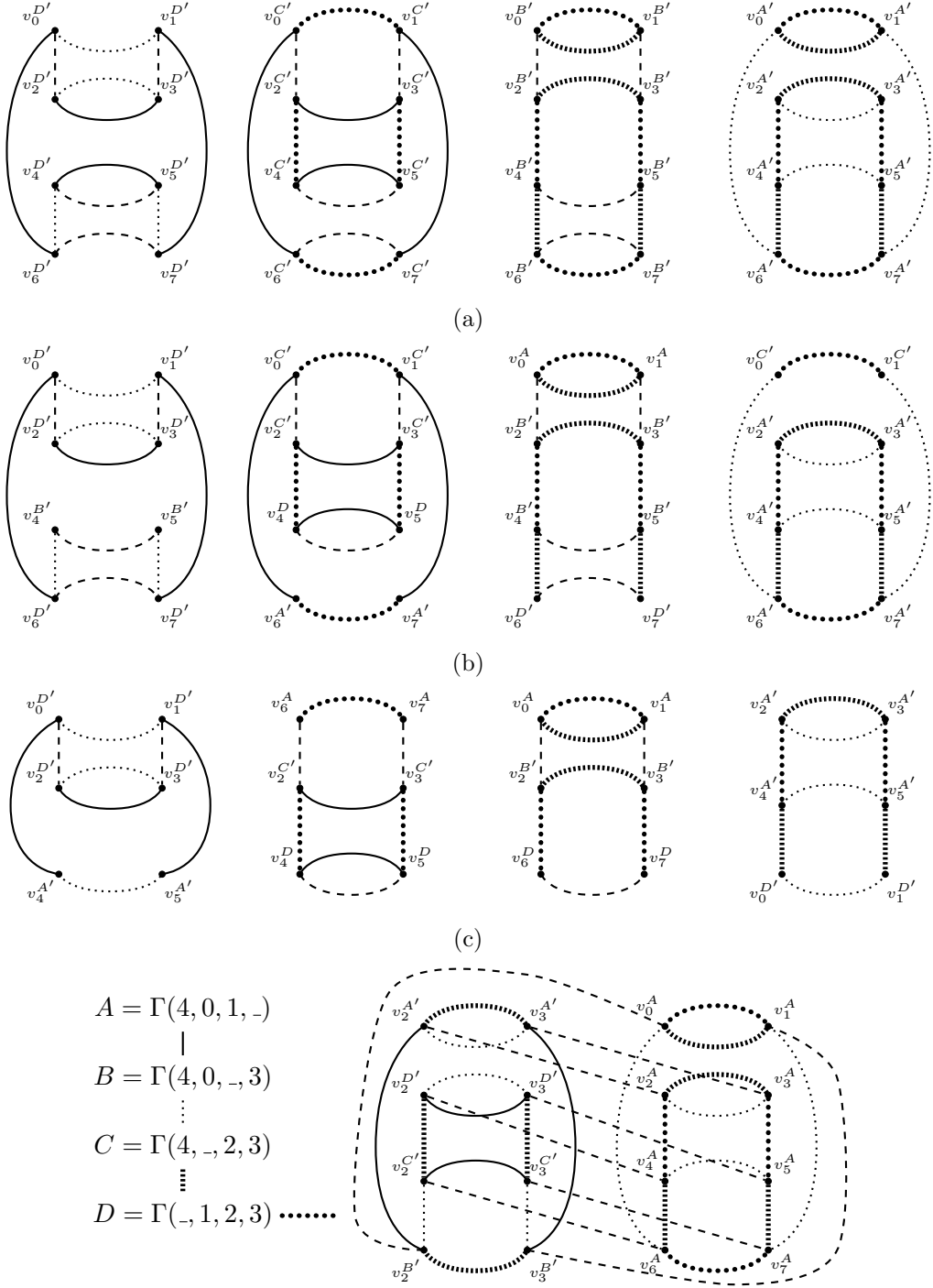

\begin{theorem}\label{theorem:G1}
A crystallization of $\mathbb{S}^2 \times \mathbb{S}^1 \times \mathbb{S}^1$ with $40$ vertices realizes the minimum possible regular genus, which is $6$.
\end{theorem}
\begin{proof}

For constructing the crystallization of $\mathbb{S}^2 \times \mathbb{S}^1 \times \mathbb{S}^1$, let us take $\Gamma(0,1,2,3)$ to be the standard $8$-vertex crystallization of $\mathbb{S}^2 \times \mathbb{S}^1$ (Figure \ref{fig:C1}). Then, $G$ in Figure \ref{fig:Gc} is a gem of $\mathbb{S}^2 \times \mathbb{S}^1 \times \mathbb{S}^1$ with $64$ vertices. We will denote this gem by $G_1$. We will apply some crystallization moves involving vertices of $A^\prime, B^\prime, C^\prime$ and $D^\prime$ to obtain a $40$-vertex crystallization of $\mathbb{S}^2 \times \mathbb{S}^1 \times \mathbb{S}^1$ (Figure \ref{fig:G1}). For the sake of conciseness, we present only $A^\prime, B^\prime, C^\prime$ and $D^\prime$ in Figures \ref{fig:G1a}, \ref{fig:G1b} and \ref{fig:G1c}. For a complete understanding, see these figures together with Figure \ref{fig:Gc}. We apply three polyhedral glue moves $(\Phi_i,\Lambda_i,\Lambda_i^\prime,i-1)$ on $G_1$, for $1\le i\le 3$ (Figure \ref{fig:G1a}). The sets $\Lambda_1, \Lambda_2$ and $\Lambda_3$ are $\{v_4^{D^\prime},v_5^{D^\prime}\}, \{v_6^{C^\prime},v_7^{C^\prime}\}$ and $\{v_0^{B^\prime},v_1^{B^\prime}\}$, respectively. The isomorphism $\Phi_i$ maps $\Lambda_i$ to the set of vertices that are adjacent to the vertices of $\Lambda_i$ via edges of color $i-1$ in $G_1$, for $1 \le i \le 3$. So, the sets $\Lambda_1^\prime, \Lambda_2^\prime$ and $\Lambda_3^\prime$ are $\{v_4^{C^\prime},v_5^{C^\prime}\}, \{v_6^{B^\prime},v_7^{B^\prime}\}$ and $\{v_0^{A^\prime},v_1^{A^\prime}\}$, respectively. After applying these three polyhedral glue moves on $G_1$, we get a crystallization of $\mathbb{S}^2 \times \mathbb{S}^1 \times \mathbb{S}^1$ with $52$ vertices, say $G_1^1$ (Figure \ref{fig:G1b}). 

Let $\Lambda_4=\{v_4^{B^\prime},v_5^{B^\prime}\}, \Lambda_4^\prime=\{v_6^{D^\prime},v_7^{D^\prime}\}, \Lambda_5=\{v_0^{C^\prime},v_1^{C^\prime}\}$ and $\Lambda_4^\prime=\{v_6^{A^\prime},v_7^{A^\prime}\}$. Note that $\Lambda_4, \Lambda_4^\prime$ and $\Lambda_5, \Lambda_5^\prime$ satisfy conditions of the observation in Subsection \ref{move}. Thus, applying moves with respect to $(\Phi_4,\Lambda_4, \Lambda_4^\prime,\{0,1\})$ and $(\Phi_5,\Lambda_5, \Lambda_5^\prime,\{1,2\})$ on $G_1^1$, we get a crystallization of $\mathbb{S}^2 \times \mathbb{S}^1 \times \mathbb{S}^1$ with $44$ vertices, say $G_1^2$ (Figure \ref{fig:G1c}). Again, $\Lambda_6=\{v_0^{D^\prime},v_1^{D^\prime}\}, \Lambda_6^\prime=\{v_4^{A^\prime},v_5^{A^\prime}\}$ satisfy conditions of the observation in Subsection \ref{move}. Thus, applying the move with respect to $(\Phi_6,\Lambda_6, \Lambda_6^\prime,\{0,2\})$ on $G_1^2$, we get the crystallization $G_1^\prime$ of $\mathbb{S}^2 \times \mathbb{S}^1 \times \mathbb{S}^1$ with $40$ vertices (Figure \ref{fig:G1d}). Note that the subgraph of $G_1^\prime$ generated by the remaining vertices of $A^\prime, B^\prime,C^\prime$ and $D^\prime$ is isomorphic to $\Gamma(0,1,2,\_)$.

Now, to compute the regular genus of $G_1^\prime$, we will compute the number of bi-colored cycles. From Figure \ref{fig:G1d}, we have that $g_{\{i,j\}}=10$, when $\{i,j\}\in \{\{0,2\},\{0,3\},\{1,3\},\{1,4\},\{2,4\}\}$, and $g_{\{i,j\}}=8$, when $\{i,j\}\in \{\{0,1\},\{0,4\},\{1,2\},\{2,3\},\{3,4\}\}$. Therefore, the regular genus of $G_1^\prime$ is 
$$\rho(G_1^\prime)=\rho_{\varepsilon}(G_1^\prime)=1-\frac{1}{2}\left(\frac{-3}{2}(40)+g_{\{0,2\}}+g_{\{2,4\}}+g_{\{1,4\}}+g_{\{1,3\}}+g_{\{0,3\}}\right)$$ $$=1-\frac{1}{2}(-60+10+10+10+10+10)=6,$$
where $\varepsilon=(0,2,4,1,3)$. Since $2\chi(M)+5m-4=6$ for $M=\mathbb{S}^2 \times \mathbb{S}^1 \times \mathbb{S}^1$, we get that $\mathcal G(\mathbb{S}^2 \times \mathbb{S}^1 \times \mathbb{S}^1)=\rho(G_1^\prime)=6$ due to Proposition \ref{lbrg}.
\end{proof}

The isomorphism signature of the crystallization $G_1^\prime$ of $\mathbb{S}^2\times \mathbb{S}^1\times \mathbb{S}^1$, obtained using Regina \cite{burton2013regina}, is the following. Ensure that the isomorphism signature contains no space before using it in Regina.

\noindent  OLvLvMLwvLLLPPMPQQQwLzMQQQzQQQQLQQcbfggkhrpqxqyrszxCtwAzBuBvvwEFG\\EJDHGIIDDDEFFGGKHHILKLJMNMNMMNNaaaaaaaaaaaaaaaaaaaaaaaaaaaaaaaaaaaaaa\\aaaaaaaaaaaaaaaaaaaaaaaaaaaaaaaaaaaaaaaaaaaaaaaaaaaaaaaaaaaaaaaaaaaaaaaaaaaaaaaaaa\\aa

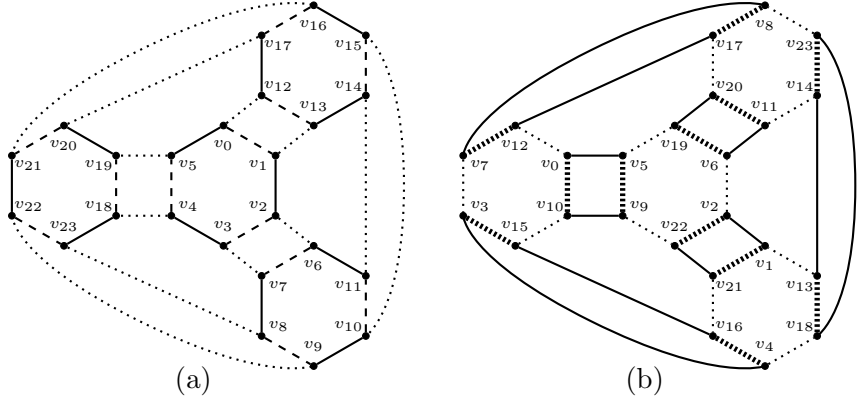
\begin{figure}[h!]
\tikzstyle{ver}=[]
\tikzstyle{verti}=[circle, draw, fill=black!100, inner sep=0pt, minimum width=2.5pt]
\tikzstyle{edge} = [draw,thick,-]
    \centering
\begin{tikzpicture}[scale=0.4]
\begin{scope}
\begin{scope}[shift={(0,0)},rotate=-30]
\foreach \x/\y in {120/0,180/5,240/4,300/3,0/2,60/1}
{\node[verti] (\y) at (\x:2){};}

\foreach \x/\y in {120/0,180/5,240/4,300/3,0/2,60/1}
{\node[ver] () at (\x:1.35){\tiny{$v_{\y}$}};}

\foreach \x/\y in 
{0/1,2/3,4/5}
{\path[edge,dashed] (\x)--(\y);}

\foreach \x/\y in 
{2/1,4/3,0/5}
{\path[edge] (\x)--(\y);}

\end{scope}

\begin{scope}[shift={(-5.3,0)},rotate=-30]
\foreach \x/\y in {120/20,180/21,240/22,300/23,0/18,60/19}
{\node[verti] (\y) at (\x:2){};}

\foreach \x/\y in {120/20,180/21,240/22,300/23,0/18,60/19}
{\node[ver] () at (\x:1.35){\tiny{$v_{\y}$}};}

\foreach \x/\y in 
{18/19,20/21,22/23}
{\path[edge,dashed] (\x)--(\y);}

\foreach \x/\y in 
{20/19,21/22,23/18}
{\path[edge] (\x)--(\y);}

\end{scope}

\begin{scope}[shift={(3,4)},rotate=-30]
\foreach \x/\y in {120/16,180/17,240/12,300/13,0/14,60/15}
{\node[verti] (\y) at (\x:2){};}

\foreach \x/\y in {120/16,180/17,240/12,300/13,0/14,60/15}
{\node[ver] () at (\x:1.35){\tiny{$v_{\y}$}};}

\foreach \x/\y in 
{12/13,14/15,16/17}
{\path[edge,dashed] (\x)--(\y);}

\foreach \x/\y in 
{13/14,15/16,17/12}
{\path[edge] (\x)--(\y);}

\end{scope}

\begin{scope}[shift={(3,-4)},rotate=-30]
\foreach \x/\y in {120/6,180/7,240/8,300/9,0/10,60/11}
{\node[verti] (\y) at (\x:2){};}

\foreach \x/\y in {120/6,180/7,240/8,300/9,0/10,60/11}
{\node[ver] () at (\x:1.35){\tiny{$v_{\y}$}};}

\foreach \x/\y in 
{6/7,8/9,10/11}
{\path[edge,dashed] (\x)--(\y);}

\foreach \x/\y in 
{7/8,9/10,11/6}
{\path[edge] (\x)--(\y);}

\end{scope}

\foreach \x/\y in 
{0/12,1/13,2/6,3/7,4/18,5/19,17/20,11/14,8/23}
{\path[edge,dotted] (\x)--(\y);}

\draw [edge,dotted] plot [smooth,tension=1.5] coordinates{(21) (-2.8,4.5) (16)};

\draw [edge,dotted] plot [smooth,tension=1.5] coordinates{(15) (6,0) (10)};

\draw [edge,dotted] plot [smooth,tension=1.5] coordinates{(22) (-2.8,-4.5) (9)};

\node[ver] () at (-1,-6.5){(a)};

\end{scope}

\begin{scope}[shift={(15,0)}]
\begin{scope}[shift={(0,0)},rotate=-30]
\foreach \x/\y in {120/0,180/5,240/4,300/3,0/2,60/1}
{\node[verti] (\y) at (\x:2){};}

\foreach \x/\y in {120/19,180/5,240/9,300/22,0/2,60/6}
{\node[ver] () at (\x:1.35){\tiny{$v_{\y}$}};}

\foreach \x/\y in 
{0/1,2/3,4/5}
{\draw [line width=2pt, line cap=rectangle, dash pattern=on 1pt off 1] (\x)--(\y);}

\foreach \x/\y in 
{2/1,4/3,0/5}
{\path[edge,dotted] (\x)--(\y);}

\end{scope}

\begin{scope}[shift={(-5.3,0)},rotate=-30]
\foreach \x/\y in {120/20,180/21,240/22,300/23,0/18,60/19}
{\node[verti] (\y) at (\x:2){};}

\foreach \x/\y in {120/12,180/7,240/3,300/15,0/10,60/0}
{\node[ver] () at (\x:1.35){\tiny{$v_{\y}$}};}

\foreach \x/\y in 
{18/19,20/21,22/23}
{\draw [line width=2pt, line cap=rectangle, dash pattern=on 1pt off 1] (\x)--(\y);}

\foreach \x/\y in 
{20/19,21/22,23/18}
{\path[edge,dotted] (\x)--(\y);}

\end{scope}

\begin{scope}[shift={(3,4)},rotate=-30]
\foreach \x/\y in {120/16,180/17,240/12,300/13,0/14,60/15}
{\node[verti] (\y) at (\x:2){};}

\foreach \x/\y in {120/8,180/17,240/20,300/11,0/14,60/23}
{\node[ver] () at (\x:1.35){\tiny{$v_{\y}$}};}

\foreach \x/\y in 
{12/13,14/15,16/17}
{\draw [line width=2pt, line cap=rectangle, dash pattern=on 1pt off 1] (\x)--(\y);}

\foreach \x/\y in 
{13/14,15/16,17/12}
{\path[edge,dotted] (\x)--(\y);}

\end{scope}

\begin{scope}[shift={(3,-4)},rotate=-30]
\foreach \x/\y in {120/6,180/7,240/8,300/9,0/10,60/11}
{\node[verti] (\y) at (\x:2){};}

\foreach \x/\y in {120/1,180/21,240/16,300/4,0/18,60/13}
{\node[ver] () at (\x:1.35){\tiny{$v_{\y}$}};}

\foreach \x/\y in 
{6/7,8/9,10/11}
{\draw [line width=2pt, line cap=rectangle, dash pattern=on 1pt off 1] (\x)--(\y);}

\foreach \x/\y in 
{7/8,9/10,11/6}
{\path[edge,dotted] (\x)--(\y);}

\end{scope}

\foreach \x/\y in 
{0/12,1/13,2/6,3/7,4/18,5/19,17/20,11/14,8/23}
{\path [edge] (\x)--(\y);}

\draw [edge] plot [smooth,tension=1.5] coordinates{(21) (-2.8,4.5) (16)};

\draw [edge] plot [smooth,tension=1.5] coordinates{(15) (6,0) (10)};

\draw [edge] plot [smooth,tension=1.5] coordinates{(22) (-2.8,-4.5) (9)};

\node[ver] () at (-1,-6.5){(b)};

\end{scope}

\end{tikzpicture}
\caption{$24$-vertex crystallization, $\Gamma=(a)\cup (b)$, of $\mathbb{S}^1\times \mathbb{S}^1\times \mathbb{S}^1$.}\label{fig:C2}
\end{figure}

\begin{figure}[h!]
\tikzstyle{ver}=[]
\tikzstyle{verti}=[circle, draw, fill=black!100, inner sep=0pt, minimum width=2.5pt]
\tikzstyle{edge} = [draw,thick,-]
    \centering
\begin{tikzpicture}[scale=0.4]
\begin{scope}[shift={(15,-13.5)}]
\begin{scope}[shift={(0,0)},rotate=-30]
\foreach \x/\y in {120/0,180/5,240/4,300/3,0/2,60/1}
{\node[verti] (\y) at (\x:2){};}

\foreach \x/\y in {120/0,180/5,240/4,300/3,0/2,60/1}
{\node[ver] () at (\x:1.35){\tiny{$v_{\y}^{D^\prime}$}};}

\foreach \x/\y in 
{0/1,2/3,4/5}
{\path[edge,dashed] (\x)--(\y);}

\foreach \x/\y in 
{2/1,4/3,0/5}
{\path[edge] (\x)--(\y);}

\end{scope}

\begin{scope}[shift={(-5.3,0)},rotate=-30]
\foreach \x/\y in {120/20,180/21,240/22,300/23,0/18,60/19}
{\node[verti] (\y) at (\x:2){};}

\foreach \x/\y in {120/20,180/21,240/22,300/23,0/18,60/19}
{\node[ver] () at (\x:1.35){\tiny{$v_{\y}^{D^\prime}$}};}

\foreach \x/\y in 
{18/19,20/21,22/23}
{\path[edge,dashed] (\x)--(\y);}

\foreach \x/\y in 
{20/19,21/22,23/18}
{\path[edge] (\x)--(\y);}

\end{scope}

\begin{scope}[shift={(3,4)},rotate=-30]
\foreach \x/\y in {120/16,180/17,240/12,300/13,0/14,60/15}
{\node[verti] (\y) at (\x:2){};}

\foreach \x/\y in {120/16,180/17,240/12,300/13,0/14,60/15}
{\node[ver] () at (\x:1.35){\tiny{$v_{\y}^{D^\prime}$}};}

\foreach \x/\y in 
{12/13,14/15,16/17}
{\path[edge,dashed] (\x)--(\y);}

\foreach \x/\y in 
{13/14,15/16,17/12}
{\path[edge] (\x)--(\y);}

\end{scope}

\begin{scope}[shift={(3,-4)},rotate=-30]
\foreach \x/\y in {120/6,180/7,240/8,300/9,0/10,60/11}
{\node[verti] (\y) at (\x:2){};}

\foreach \x/\y in {120/6,180/7,240/8,300/9,0/10,60/11}
{\node[ver] () at (\x:1.35){\tiny{$v_{\y}^{D^\prime}$}};}

\foreach \x/\y in 
{6/7,8/9,10/11}
{\path[edge,dashed] (\x)--(\y);}

\foreach \x/\y in 
{7/8,9/10,11/6}
{\path[edge] (\x)--(\y);}

\end{scope}

\foreach \x/\y in 
{0/12,1/13,2/6,3/7,4/18,5/19,17/20,11/14,8/23}
{\path[edge,dotted] (\x)--(\y);}

\draw [edge,dotted] plot [smooth,tension=1.5] coordinates{(21) (-2.8,4.5) (16)};

\draw [edge,dotted] plot [smooth,tension=1.5] coordinates{(15) (6,0) (10)};

\draw [edge,dotted] plot [smooth,tension=1.5] coordinates{(22) (-2.8,-4.5) (9)};

\end{scope}

\begin{scope}[shift={(0,-13.5)}]
\begin{scope}[shift={(0,0)},rotate=-30]
\foreach \x/\y in {120/0,180/5,240/4,300/3,0/2,60/1}
{\node[verti] (\y) at (\x:2){};}

\foreach \x/\y in {120/3,180/2,240/1,300/0,0/5,60/4}
{\node[ver] () at (\x:1.35){\tiny{$v_{\y}^{C^\prime}$}};}

\foreach \x/\y in 
{0/1,2/3,4/5}
{\path[edge] (\x)--(\y);}

\foreach \x/\y in 
{2/1,4/3,0/5}
{\path[edge,dashed] (\x)--(\y);}

\end{scope}

\begin{scope}[shift={(-5.3,0)},rotate=-30]
\foreach \x/\y in {120/20,180/21,240/22,300/23,0/18,60/19}
{\node[verti] (\y) at (\x:2){};}

\foreach \x/\y in {120/23,180/18,240/19,300/20,0/21,60/22}
{\node[ver] () at (\x:1.35){\tiny{$v_{\y}^{C^\prime}$}};}

\foreach \x/\y in 
{18/19,20/21,22/23}
{\path[edge] (\x)--(\y);}

\foreach \x/\y in 
{20/19,21/22,23/18}
{\path[edge,dashed] (\x)--(\y);}

\end{scope}

\begin{scope}[shift={(3,4)},rotate=-30]
\foreach \x/\y in {120/16,180/17,240/12,300/13,0/14,60/15}
{\node[verti] (\y) at (\x:2){};}

\foreach \x/\y in {120/13,180/14,240/15,300/16,0/17,60/12}
{\node[ver] () at (\x:1.35){\tiny{$v_{\y}^{C^\prime}$}};}

\foreach \x/\y in 
{12/13,14/15,16/17}
{\path[edge] (\x)--(\y);}

\foreach \x/\y in 
{13/14,15/16,17/12}
{\path[edge,dashed] (\x)--(\y);}

\end{scope}

\begin{scope}[shift={(3,-4)},rotate=-30]
\foreach \x/\y in {120/6,180/7,240/8,300/9,0/10,60/11}
{\node[verti] (\y) at (\x:2){};}

\foreach \x/\y in {120/9,180/10,240/11,300/6,0/7,60/8}
{\node[ver] () at (\x:1.35){\tiny{$v_{\y}^{C^\prime}$}};}

\foreach \x/\y in 
{6/7,8/9,10/11}
{\path[edge] (\x)--(\y);}

\foreach \x/\y in 
{7/8,9/10,11/6}
{\path[edge,dashed] (\x)--(\y);}

\end{scope}

\foreach \x/\y in 
{0/12,1/13,2/6,3/7,4/18,5/19,17/20,11/14,8/23}
{\draw [line width=2pt, line cap=round, dash pattern=on 0pt off 1.7\pgflinewidth] (\x)--(\y);}

\draw[line width=2pt, line cap=round, dash pattern=on 0pt off 1.7\pgflinewidth] plot [smooth,tension=1.5] coordinates{(21) (-2.8,4.5) (16)};

\draw[line width=2pt, line cap=round, dash pattern=on 0pt off 1.7\pgflinewidth] plot [smooth,tension=1.5] coordinates{(15) (6,0) (10)};

\draw[line width=2pt, line cap=round, dash pattern=on 0pt off 1.7\pgflinewidth] plot[smooth,tension=1.5] coordinates{(22) (-2.8,-4.5) (9)};

\end{scope}

\begin{scope}[shift={(0,0)}]
\begin{scope}[shift={(0,0)},rotate=-30]
\foreach \x/\y in {120/0,180/5,240/4,300/3,0/2,60/1}
{\node[verti] (\y) at (\x:2){};}

\foreach \x/\y in {120/3,180/15,240/14,300/23,0/22,60/2}
{\node[ver] () at (\x:1.35){\tiny{$v_{\y}^{B^\prime}$}};}

\foreach \x/\y in 
{0/1,2/3,4/5}
{\path[edge,dashed] (\x)--(\y);}

\foreach \x/\y in 
{2/1,4/3,0/5}
{\draw [line width=2pt, line cap=round, dash pattern=on 0pt off 1.7\pgflinewidth] (\x)--(\y);}

\end{scope}

\begin{scope}[shift={(-5.3,0)},rotate=-30]
\foreach \x/\y in {120/20,180/21,240/22,300/23,0/18,60/19}
{\node[verti] (\y) at (\x:2){};}

\foreach \x/\y in {120/0,180/1,240/21,300/20,0/11,60/10}
{\node[ver] () at (\x:1.35){\tiny{$v_{\y}^{B^\prime}$}};}

\foreach \x/\y in 
{18/19,20/21,22/23}
{\path[edge,dashed] (\x)--(\y);}

\foreach \x/\y in 
{20/19,21/22,23/18}
{\draw[line width=2pt, line cap=round, dash pattern=on 0pt off 1.7\pgflinewidth] (\x)--(\y);}

\end{scope}

\begin{scope}[shift={(3,4)},rotate=-30]
\foreach \x/\y in {120/16,180/17,240/12,300/13,0/14,60/15}
{\node[verti] (\y) at (\x:2){};}

\foreach \x/\y in {120/13,180/12,240/7,300/6,0/19,60/18}
{\node[ver] () at (\x:1.35){\tiny{$v_{\y}^{B^\prime}$}};}

\foreach \x/\y in 
{12/13,14/15,16/17}
{\path[edge,dashed] (\x)--(\y);}

\foreach \x/\y in 
{13/14,15/16,17/12}
{\draw [line width=2pt, line cap=round, dash pattern=on 0pt off 1.7\pgflinewidth] (\x)--(\y);}

\end{scope}

\begin{scope}[shift={(3,-4)},rotate=-30]
\foreach \x/\y in {120/6,180/7,240/8,300/9,0/10,60/11}
{\node[verti] (\y) at (\x:2){};}

\foreach \x/\y in {120/9,180/8,240/17,300/16,0/4,60/5}
{\node[ver] () at (\x:1.35){\tiny{$v_{\y}^{B^\prime}$}};}

\foreach \x/\y in 
{6/7,8/9,10/11}
{\path[edge,dashed] (\x)--(\y);}

\foreach \x/\y in 
{7/8,9/10,11/6}
{\draw [line width=2pt, line cap=round, dash pattern=on 0pt off 1.7\pgflinewidth] (\x)--(\y);}

\end{scope}

\foreach \x/\y in 
{0/12,1/13,2/6,3/7,4/18,5/19,17/20,11/14,8/23}
{\draw [line width=2pt, line cap=rectangle, dash pattern=on 1pt off 1] (\x)--(\y);}

\draw[line width=2pt, line cap=rectangle, dash pattern=on 1pt off 1] plot [smooth,tension=1.5] coordinates{(21) (-2.8,4.5) (16)};

\draw[line width=2pt, line cap=rectangle, dash pattern=on 1pt off 1] plot [smooth,tension=1.5] coordinates{(15) (6,0) (10)};

\draw[line width=2pt, line cap=rectangle, dash pattern=on 1pt off 1] plot[smooth,tension=1.5] coordinates{(22) (-2.8,-4.5) (9)};

\end{scope}

\begin{scope}[shift={(15,0)}]
\begin{scope}[shift={(0,0)},rotate=-30]
\foreach \x/\y in {120/0,180/5,240/4,300/3,0/2,60/1}
{\node[verti] (\y) at (\x:2){};}

\foreach \x/\y in {120/19,180/5,240/9,300/22,0/2,60/6}
{\node[ver] () at (\x:1.35){\tiny{$v_{\y}^{A^\prime}$}};}

\foreach \x/\y in 
{0/1,2/3,4/5}
{\draw [line width=2pt, line cap=round, dash pattern=on 0pt off 1.7\pgflinewidth] (\x)--(\y);}

\foreach \x/\y in 
{2/1,4/3,0/5}
{\draw [line width=2pt, line cap=rectangle, dash pattern=on 1pt off 1] (\x)--(\y);}

\end{scope}

\begin{scope}[shift={(-5.3,0)},rotate=-30]
\foreach \x/\y in {120/20,180/21,240/22,300/23,0/18,60/19}
{\node[verti] (\y) at (\x:2){};}

\foreach \x/\y in {120/12,180/7,240/3,300/15,0/10,60/0}
{\node[ver] () at (\x:1.35){\tiny{$v_{\y}^{A^\prime}$}};}

\foreach \x/\y in 
{18/19,20/21,22/23}
{\draw[line width=2pt, line cap=round, dash pattern=on 0pt off 1.7\pgflinewidth] (\x)--(\y);}

\foreach \x/\y in 
{20/19,21/22,23/18}
{\draw [line width=2pt, line cap=rectangle, dash pattern=on 1pt off 1](\x)--(\y);}

\end{scope}

\begin{scope}[shift={(3,4)},rotate=-30]
\foreach \x/\y in {120/16,180/17,240/12,300/13,0/14,60/15}
{\node[verti] (\y) at (\x:2){};}

\foreach \x/\y in {120/8,180/17,240/20,300/11,0/14,60/23}
{\node[ver] () at (\x:1.35){\tiny{$v_{\y}^{A^\prime}$}};}

\foreach \x/\y in 
{12/13,14/15,16/17}
{\draw [line width=2pt, line cap=round, dash pattern=on 0pt off 1.7\pgflinewidth] (\x)--(\y);}

\foreach \x/\y in 
{13/14,15/16,17/12}
{\draw [line width=2pt, line cap=rectangle, dash pattern=on 1pt off 1] (\x)--(\y);}

\end{scope}

\begin{scope}[shift={(3,-4)},rotate=-30]
\foreach \x/\y in {120/6,180/7,240/8,300/9,0/10,60/11}
{\node[verti] (\y) at (\x:2){};}

\foreach \x/\y in {120/1,180/21,240/16,300/4,0/18,60/13}
{\node[ver] () at (\x:1.35){\tiny{$v_{\y}^{A^\prime}$}};}

\foreach \x/\y in 
{7/8,9/10,11/6}
{\draw [line width=2pt, line cap=rectangle, dash pattern=on 1pt off 1] (\x)--(\y);}

\foreach \x/\y in 
{6/7,8/9,10/11}
{\draw [line width=2pt, line cap=round, dash pattern=on 0pt off 1.7\pgflinewidth] (\x)--(\y);}

\end{scope}

\foreach \x/\y in 
{0/12,1/13,2/6,3/7,4/18,5/19,17/20,11/14,8/23}
{\path [edge,dotted] (\x)--(\y);}

\draw[edge,dotted] plot [smooth,tension=1.5] coordinates{(21) (-2.8,4.5) (16)};

\draw[edge,dotted] plot [smooth,tension=1.5] coordinates{(15) (6,0) (10)};

\draw[edge,dotted] plot[smooth,tension=1.5] coordinates{(22) (-2.8,-4.5) (9)};
\end{scope}

\end{tikzpicture}
\caption{The subgraphs $A^\prime, B^\prime, C^\prime, D^\prime$ of the gem $G_2$.}\label{fig:G2a}
\end{figure}

\begin{remark}{\rm
The classification of PL $n$-manifolds by regular genus is a well-established problem in combinatorial topology. For orientable PL 4-manifolds, the classification is complete up to regular genus 5 (cf. \cite{ca99, fgg86, s99}), but it remains open for regular genus 6 and beyond. In \cite{B19}, two orientable and two non-orientable prime closed PL 4-manifolds with regular genus 6 were constructed, and it was conjectured that $\mathbb{S}^2 \times \mathbb{S}^1 \times \mathbb{S}^1$ also has regular genus 6. This problem remained open for several years. In this article, we settle this conjecture by proving that the regular genus of $\mathbb{S}^2 \times \mathbb{S}^1 \times \mathbb{S}^1$ is indeed 6. 
}
\end{remark}

In the following Theorem, we prove that the regular genus of the $4$-dimensional torus is $16$.

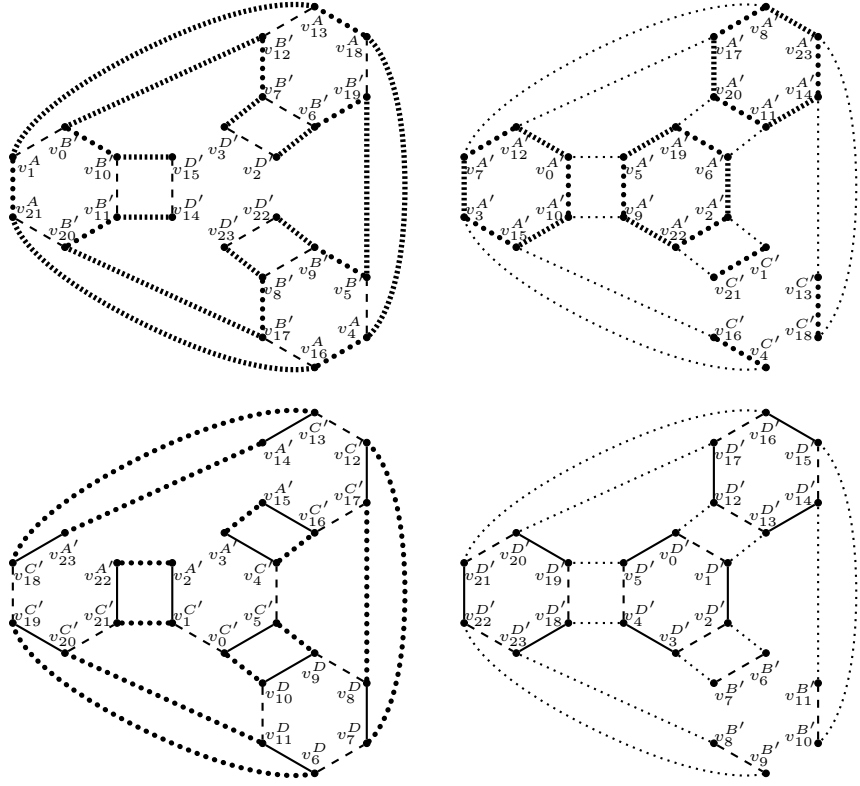
\begin{figure}[h!]
\tikzstyle{ver}=[]
\tikzstyle{verti}=[circle, draw, fill=black!100, inner sep=0pt, minimum width=2.5pt]
\tikzstyle{edge} = [draw,thick,-]
    \centering
\begin{tikzpicture}[scale=0.4]
\begin{scope}[shift={(15,-13.5)}]
\begin{scope}[shift={(0,0)},rotate=-30]
\foreach \x/\y in {120/0,180/5,240/4,300/3,0/2,60/1}
{\node[verti] (\y) at (\x:2){};}

\foreach \x/\y in {120/0,180/5,240/4,300/3,0/2,60/1}
{\node[ver] () at (\x:1.35){\tiny{$v_{\y}^{D^\prime}$}};}

\foreach \x/\y in 
{0/1,2/3,4/5}
{\path[edge,dashed] (\x)--(\y);}

\foreach \x/\y in 
{2/1,4/3,0/5}
{\path[edge] (\x)--(\y);}

\end{scope}

\begin{scope}[shift={(-5.3,0)},rotate=-30]
\foreach \x/\y in {120/20,180/21,240/22,300/23,0/18,60/19}
{\node[verti] (\y) at (\x:2){};}

\foreach \x/\y in {120/20,180/21,240/22,300/23,0/18,60/19}
{\node[ver] () at (\x:1.35){\tiny{$v_{\y}^{D^\prime}$}};}

\foreach \x/\y in 
{18/19,20/21,22/23}
{\path[edge,dashed] (\x)--(\y);}

\foreach \x/\y in 
{20/19,21/22,23/18}
{\path[edge] (\x)--(\y);}

\end{scope}

\begin{scope}[shift={(3,4)},rotate=-30]
\foreach \x/\y in {120/16,180/17,240/12,300/13,0/14,60/15}
{\node[verti] (\y) at (\x:2){};}

\foreach \x/\y in {120/16,180/17,240/12,300/13,0/14,60/15}
{\node[ver] () at (\x:1.35){\tiny{$v_{\y}^{D^\prime}$}};}

\foreach \x/\y in 
{12/13,14/15,16/17}
{\path[edge,dashed] (\x)--(\y);}

\foreach \x/\y in 
{13/14,15/16,17/12}
{\path[edge] (\x)--(\y);}

\end{scope}

\begin{scope}[shift={(3,-4)},rotate=-30]
\foreach \x/\y in {120/6,180/7,240/8,300/9,0/10,60/11}
{\node[verti] (\y) at (\x:2){};}

\foreach \x/\y in {120/6,180/7,240/8,300/9,0/10,60/11}
{\node[ver] () at (\x:1.35){\tiny{$v_{\y}^{B^\prime}$}};}

\foreach \x/\y in 
{6/7,8/9,10/11}
{\path[edge,dashed] (\x)--(\y);}

\end{scope}

\foreach \x/\y in 
{0/12,1/13,2/6,3/7,4/18,5/19,17/20,11/14,8/23}
{\path[edge,dotted] (\x)--(\y);}

\draw [edge,dotted] plot [smooth,tension=1.5] coordinates{(21) (-2.8,4.5) (16)};

\draw [edge,dotted] plot [smooth,tension=1.5] coordinates{(15) (6,0) (10)};

\draw [edge,dotted] plot [smooth,tension=1.5] coordinates{(22) (-2.8,-4.5) (9)};

\end{scope}

\begin{scope}[shift={(0,-13.5)}]
\begin{scope}[shift={(0,0)},rotate=-30]
\foreach \x/\y in {120/0,180/5,240/4,300/3,0/2,60/1}
{\node[verti] (\y) at (\x:2){};}

\foreach \x/\y in {240/1,300/0,0/5,60/4}
{\node[ver] () at (\x:1.35){\tiny{$v_{\y}^{C^\prime}$}};}

\foreach \x/\y in {120/3,180/2}
{\node[ver] () at (\x:1.35){\tiny{$v_{\y}^{A^\prime}$}};}

\foreach \x/\y in 
{0/1,2/3,4/5}
{\path[edge] (\x)--(\y);}

\foreach \x/\y in 
{2/1,4/3}
{\path[edge,dashed] (\x)--(\y);}

\end{scope}

\begin{scope}[shift={(-5.3,0)},rotate=-30]
\foreach \x/\y in {120/20,180/21,240/22,300/23,0/18,60/19}
{\node[verti] (\y) at (\x:2){};}

\foreach \x/\y in {180/18,240/19,300/20,0/21}
{\node[ver] () at (\x:1.35){\tiny{$v_{\y}^{C^\prime}$}};}

\foreach \x/\y in {120/23,60/22}
{\node[ver] () at (\x:1.35){\tiny{$v_{\y}^{A^\prime}$}};}

\foreach \x/\y in 
{18/19,20/21,22/23}
{\path[edge] (\x)--(\y);}

\foreach \x/\y in 
{21/22,23/18}
{\path[edge,dashed] (\x)--(\y);}

\end{scope}

\begin{scope}[shift={(3,4)},rotate=-30]
\foreach \x/\y in {120/16,180/17,240/12,300/13,0/14,60/15}
{\node[verti] (\y) at (\x:2){};}

\foreach \x/\y in {120/13,300/16,0/17,60/12}
{\node[ver] () at (\x:1.35){\tiny{$v_{\y}^{C^\prime}$}};}

\foreach \x/\y in {180/14,240/15}
{\node[ver] () at (\x:1.35){\tiny{$v_{\y}^{A^\prime}$}};}

\foreach \x/\y in 
{12/13,14/15,16/17}
{\path[edge] (\x)--(\y);}

\foreach \x/\y in 
{13/14,15/16}
{\path[edge,dashed] (\x)--(\y);}

\end{scope}

\begin{scope}[shift={(3,-4)},rotate=-30]
\foreach \x/\y in {120/6,180/7,240/8,300/9,0/10,60/11}
{\node[verti] (\y) at (\x:2){};}

\foreach \x/\y in {120/9,180/10,240/11,300/6,0/7,60/8}
{\node[ver] () at (\x:1.35){\tiny{$v_{\y}^{D}$}};}

\foreach \x/\y in 
{6/7,8/9,10/11}
{\path[edge] (\x)--(\y);}

\foreach \x/\y in 
{7/8,9/10,11/6}
{\path[edge,dashed] (\x)--(\y);}

\end{scope}

\foreach \x/\y in 
{0/12,1/13,2/6,3/7,4/18,5/19,17/20,11/14,8/23}
{\draw [line width=2pt, line cap=round, dash pattern=on 0pt off 1.7\pgflinewidth] (\x)--(\y);}

\draw[line width=2pt, line cap=round, dash pattern=on 0pt off 1.7\pgflinewidth] plot [smooth,tension=1.5] coordinates{(21) (-2.8,4.5) (16)};

\draw[line width=2pt, line cap=round, dash pattern=on 0pt off 1.7\pgflinewidth] plot [smooth,tension=1.5] coordinates{(15) (6,0) (10)};

\draw[line width=2pt, line cap=round, dash pattern=on 0pt off 1.7\pgflinewidth] plot[smooth,tension=1.5] coordinates{(22) (-2.8,-4.5) (9)};

\end{scope}

\begin{scope}[shift={(0,0)}]
\begin{scope}[shift={(0,0)},rotate=-30]
\foreach \x/\y in {120/0,180/5,240/4,300/3,0/2,60/1}
{\node[verti] (\y) at (\x:2){};}

\foreach \x/\y in {120/3,180/15,240/14,300/23,0/22,60/2}
{\node[ver] () at (\x:1.35){\tiny{$v_{\y}^{D^\prime}$}};}

\foreach \x/\y in 
{0/1,2/3,4/5}
{\path[edge,dashed] (\x)--(\y);}

\end{scope}

\begin{scope}[shift={(-5.3,0)},rotate=-30]
\foreach \x/\y in {120/20,180/21,240/22,300/23,0/18,60/19}
{\node[verti] (\y) at (\x:2){};}

\foreach \x/\y in {120/0,300/20,0/11,60/10}
{\node[ver] () at (\x:1.35){\tiny{$v_{\y}^{B^\prime}$}};}

\foreach \x/\y in {180/1,240/21}
{\node[ver] () at (\x:1.35){\tiny{$v_{\y}^{A}$}};}

\foreach \x/\y in 
{18/19,20/21,22/23}
{\path[edge,dashed] (\x)--(\y);}

\foreach \x/\y in 
{20/19,21/22,23/18}
{\draw[line width=2pt, line cap=round, dash pattern=on 0pt off 1.7\pgflinewidth] (\x)--(\y);}

\end{scope}

\begin{scope}[shift={(3,4)},rotate=-30]
\foreach \x/\y in {120/16,180/17,240/12,300/13,0/14,60/15}
{\node[verti] (\y) at (\x:2){};}

\foreach \x/\y in {180/12,240/7,300/6,0/19}
{\node[ver] () at (\x:1.35){\tiny{$v_{\y}^{B^\prime}$}};}

\foreach \x/\y in {120/13,60/18}
{\node[ver] () at (\x:1.35){\tiny{$v_{\y}^{A}$}};}

\foreach \x/\y in 
{12/13,14/15,16/17}
{\path[edge,dashed] (\x)--(\y);}

\foreach \x/\y in 
{13/14,15/16,17/12}
{\draw [line width=2pt, line cap=round, dash pattern=on 0pt off 1.7\pgflinewidth] (\x)--(\y);}

\end{scope}

\begin{scope}[shift={(3,-4)},rotate=-30]
\foreach \x/\y in {120/6,180/7,240/8,300/9,0/10,60/11}
{\node[verti] (\y) at (\x:2){};}

\foreach \x/\y in {120/9,180/8,240/17,60/5}
{\node[ver] () at (\x:1.35){\tiny{$v_{\y}^{B^\prime}$}};}

\foreach \x/\y in {300/16,0/4}
{\node[ver] () at (\x:1.35){\tiny{$v_{\y}^{A}$}};}

\foreach \x/\y in 
{6/7,8/9,10/11}
{\path[edge,dashed] (\x)--(\y);}

\foreach \x/\y in 
{7/8,9/10,11/6}
{\draw [line width=2pt, line cap=round, dash pattern=on 0pt off 1.7\pgflinewidth] (\x)--(\y);}

\end{scope}

\foreach \x/\y in 
{0/12,1/13,2/6,3/7,4/18,5/19,17/20,11/14,8/23}
{\draw [line width=2pt, line cap=rectangle, dash pattern=on 1pt off 1] (\x)--(\y);}

\draw[line width=2pt, line cap=rectangle, dash pattern=on 1pt off 1] plot [smooth,tension=1.5] coordinates{(21) (-2.8,4.5) (16)};

\draw[line width=2pt, line cap=rectangle, dash pattern=on 1pt off 1] plot [smooth,tension=1.5] coordinates{(15) (6,0) (10)};

\draw[line width=2pt, line cap=rectangle, dash pattern=on 1pt off 1] plot[smooth,tension=1.5] coordinates{(22) (-2.8,-4.5) (9)};
\end{scope}

\begin{scope}[shift={(15,0)}]
\begin{scope}[shift={(0,0)},rotate=-30]
\foreach \x/\y in {120/0,180/5,240/4,300/3,0/2,60/1}
{\node[verti] (\y) at (\x:2){};}

\foreach \x/\y in {120/19,180/5,240/9,300/22,0/2,60/6}
{\node[ver] () at (\x:1.35){\tiny{$v_{\y}^{A^\prime}$}};}

\foreach \x/\y in 
{0/1,2/3,4/5}
{\draw [line width=2pt, line cap=round, dash pattern=on 0pt off 1.7\pgflinewidth] (\x)--(\y);}

\foreach \x/\y in 
{2/1,4/3,0/5}
{\draw [line width=2pt, line cap=rectangle, dash pattern=on 1pt off 1] (\x)--(\y);}

\end{scope}

\begin{scope}[shift={(-5.3,0)},rotate=-30]
\foreach \x/\y in {120/20,180/21,240/22,300/23,0/18,60/19}
{\node[verti] (\y) at (\x:2){};}

\foreach \x/\y in {120/12,180/7,240/3,300/15,0/10,60/0}
{\node[ver] () at (\x:1.35){\tiny{$v_{\y}^{A^\prime}$}};}

\foreach \x/\y in 
{18/19,20/21,22/23}
{\draw[line width=2pt, line cap=round, dash pattern=on 0pt off 1.7\pgflinewidth] (\x)--(\y);}

\foreach \x/\y in 
{20/19,21/22,23/18}
{\draw [line width=2pt, line cap=rectangle, dash pattern=on 1pt off 1](\x)--(\y);}

\end{scope}

\begin{scope}[shift={(3,4)},rotate=-30]
\foreach \x/\y in {120/16,180/17,240/12,300/13,0/14,60/15}
{\node[verti] (\y) at (\x:2){};}

\foreach \x/\y in {120/8,180/17,240/20,300/11,0/14,60/23}
{\node[ver] () at (\x:1.35){\tiny{$v_{\y}^{A^\prime}$}};}

\foreach \x/\y in 
{12/13,14/15,16/17}
{\draw [line width=2pt, line cap=round, dash pattern=on 0pt off 1.7\pgflinewidth] (\x)--(\y);}

\foreach \x/\y in 
{13/14,15/16,17/12}
{\draw [line width=2pt, line cap=rectangle, dash pattern=on 1pt off 1] (\x)--(\y);}

\end{scope}

\begin{scope}[shift={(3,-4)},rotate=-30]
\foreach \x/\y in {120/6,180/7,240/8,300/9,0/10,60/11}
{\node[verti] (\y) at (\x:2){};}

\foreach \x/\y in {120/1,180/21,240/16,300/4,0/18,60/13}
{\node[ver] () at (\x:1.35){\tiny{$v_{\y}^{C^\prime}$}};}

\foreach \x/\y in 
{6/7,8/9,10/11}
{\draw [line width=2pt, line cap=round, dash pattern=on 0pt off 1.7\pgflinewidth] (\x)--(\y);}

\end{scope}

\foreach \x/\y in 
{0/12,1/13,2/6,3/7,4/18,5/19,17/20,11/14,8/23}
{\path [edge,dotted] (\x)--(\y);}

\draw[edge,dotted] plot [smooth,tension=1.5] coordinates{(21) (-2.8,4.5) (16)};

\draw[edge,dotted] plot [smooth,tension=1.5] coordinates{(15) (6,0) (10)};

\draw[edge,dotted] plot[smooth,tension=1.5] coordinates{(22) (-2.8,-4.5) (9)};
\end{scope}

\end{tikzpicture}
\caption{Subgraphs of the crystallization $G_2^1$.}\label{fig:G2b}
\end{figure}

\begin{theorem}\label{theorem:G2}
A crystallization of $\mathbb{S}^1 \times \mathbb{S}^1 \times \mathbb{S}^1 \times \mathbb{S}^1$ with $120$ vertices realizes the minimum possible regular genus, which is $16$.
\end{theorem}
\begin{proof}
    For constructing this crystallization, we take $\Gamma(0,1,2,3)$ to be the $24$-vertex crystallization of $\mathbb{S}^1 \times \mathbb{S}^1 \times \mathbb{S}^1$ in Figure \ref{fig:C2}. So, $G$ is a gem of $\mathbb{S}^1 \times \mathbb{S}^1 \times \mathbb{S}^1 \times \mathbb{S}^1$ with $192$ vertices (Figure \ref{fig:Gc}). We denote this gem by $G_2$. Since we will apply crystallization moves involving vertices of $A^\prime, B^\prime, C^\prime, D^\prime$, we present only $A^\prime, B^\prime, C^\prime, D^\prime$ in Figures \ref{fig:G2a}, \ref{fig:G2b} and \ref{fig:G2c}. Let us apply three polyhedral glue moves with respect to $(\Phi_i,\Lambda_i,\Lambda^\prime_i,i-1)$ on $G_2$ to obtain a crystallization, say $G_2^1$, of $\mathbb{S}^1 \times \mathbb{S}^1 \times \mathbb{S}^1 \times \mathbb{S}^1$ with $156$ vertices, for $1\le i\le 3$. The sets $\Lambda_1,\Lambda_2$ and $\Lambda_3$ are $\{v_6^{D^\prime},v_7^{D^\prime},v_8^{D^\prime},v_9^{D^\prime},v_{10}^{D^\prime},v_{11}^{D^\prime}\}, \{v_3^{C^\prime},v_2^{C^\prime},v_{22}^{C^\prime},v_{23}^{C^\prime},v_{14}^{C^\prime},v_{15}^{C^\prime}\}$ and $  \{v_1^{B^\prime},v_{13}^{B^\prime},v_{18}^{B^\prime},v_4^{B^\prime},v_{16}^{B^\prime},v_{21}^{B^\prime}\}$, respectively. The isomorphism $\Phi_i$ maps $\Lambda_i$ to the set of vertices that are adjacent to the vertices of $\Lambda_i$ via edges of color $i-1$ in $G_2$, for $1 \le i \le 3$. Thus, the sets $\Lambda_1^\prime,\Lambda_2^\prime$ and $\Lambda_3^\prime$ are $\{v_6^{C^\prime},v_7^{C^\prime},v_8^{C^\prime},v_9^{C^\prime},v_{10}^{C^\prime},v_{11}^{C^\prime}\}, \{v_3^{B^\prime},v_2^{B^\prime},v_{22}^{B^\prime},v_{23}^{B^\prime},v_{14}^{B^\prime},v_{15}^{B^\prime}\}$ and $\{v_1^{A^\prime},v_{13}^{A^\prime},v_{18}^{A^\prime},v_4^{A^\prime},v_{16}^{A^\prime}, v_{21}^{A^\prime}\}$, respectively. Figure \ref{fig:G2b} together with Figure \ref{fig:Gc} represent the crystallization $G_2^1$ with $156$ vertices. 
    
    Now, because of the observation in Subsection \ref{move}, we can apply three moves with respect to $(\Phi_i,\Lambda_i,\Lambda^\prime_i,\{0,1\})$, for $4\le i\le 6$ and three moves with respect to $(\Phi_i,\Lambda_i,\Lambda^\prime_i,\{1,2\})$, for $7\le i\le 9$ on $G_2^1$. The sets $\Lambda_4,\Lambda_5,\Lambda_6$ are $\{v_2^{D^\prime},v_3^{D^\prime}\}, \{v_{14}^{D^\prime},v_{15}^{D^\prime}\}, \{v_{22}^{D^\prime},v_{23}^{D^\prime}\}$, respectively, and the sets $\Lambda_4^\prime,\Lambda_5^\prime,\Lambda_6^\prime$ are $\{v_6^{B^\prime},v_7^{B^\prime}\}, \{v_{11}^{B^\prime},v_{10}^{B^\prime}\}, \{v_{9}^{B^\prime},v_{8}^{B^\prime}\}$, respectively. The sets $\Lambda_7,\Lambda_8,\Lambda_9$ are $\{v_1^{C^\prime},v_{21}^{C^\prime}\}, \{v_{13}^{C^\prime},v_{18}^{C^\prime}\}, \{v_{4}^{C^\prime},v_{16}^{C^\prime}\}$, respectively, and the sets $\Lambda_7^\prime,\Lambda_8^\prime,\Lambda_9^\prime$ are $\{v_2^{A^\prime},v_{22}^{A^\prime}\}, \{v_{14}^{A^\prime},v_{23}^{A^\prime}\},\\\{v_{3}^{A^\prime},v_{15}^{A^\prime}\}$, respectively. After applying these moves on the crystallization $G_2^1$, we get another crystallization of $\mathbb{S}^1\times \mathbb{S}^1\times \mathbb{S}^1\times \mathbb{S}^1$ with $132$ vertices, say $G_2^2$ (Figure \ref{fig:G2c}). Let the sets $\Lambda_{10}, \Lambda_{11}, \Lambda_{12}$ be $\{v_1^{D^\prime},v_{13}^{D^\prime}\}, \{v_{4}^{D^\prime},v_{18}^{D^\prime}\}, \{v_{16}^{D^\prime},v_{21}^{D^\prime}\}$, respectively, and the sets $\Lambda_{10}^\prime,\Lambda_{11}^\prime,\Lambda_{12}^\prime$ be $\{v_6^{A^\prime},v_{11}^{A^\prime}\}, \{v_{7}^{A^\prime},v_{8}^{A^\prime}\}, \{v_{10}^{A^\prime},v_{9}^{A^\prime}\}$, respectively. Since $\Lambda_i,\Lambda_i^\prime$ satisfy the conditions of the observation in Subsection \ref{move}, we apply the move with respect to $(\Phi_i,\Lambda_i, \Lambda_i^\prime,\{0,2\})$ on $G_2^2$, for $10\le i\le 12$. Therefore, we get the crystallization $G_2^\prime$ with $120$ vertices of $\mathbb{S}^1\times \mathbb{S}^1\times \mathbb{S}^1\times \mathbb{S}^1$ (Figure \ref{fig:G2d}). The subgraph of $G_2^\prime$ generated by the remaining vertices of $A^\prime, B^\prime, C^\prime, D^\prime$ is isomorphic to $\Gamma(0,1,2,\_)$. 

    In $\Gamma(0,1,2,3)$, $g_{\{i,j\}}=4$ and all the cycles colored by $\{i,j\}$ are of length $6$, when $\{i,j\}\in \{\{2,3\},\{1,2\},\{0,1\},\{0,3\}\}$. All the cycles colored by $\{i,j\}$ are of length $4$ and thus, $g_{\{i,j\}}=6$ in $\Gamma(0,1,2,3)$, when $\{i,j\}\in \{\{0,2\},\{1,3\}\}$. Using this, we get from Figure \ref{fig:G2d} that in $G_2^\prime$, all the cycles colored by $\{i,j\}$ are of length $4$ and thus, $g_{\{i,j\}}=30$, when $\{i,j\}\in \{\{0,2\},\{2,4\},\{1,4\},\{1,3\},\{0,3\}\}$. The regular genus of $G_2^\prime$ with respect to the permutation $\varepsilon=(0,2,4,1,3)$ is 
    $$\rho_{\varepsilon}(G_2^\prime)=1-\frac{1}{2}\left(\frac{-3}{2}(120)+g_{\{0,2\}}+g_{\{2,4\}}+g_{\{1,4\}}+g_{\{1,3\}}+g_{\{0,3\}}\right)$$ $$=1-\frac{1}{2}(-180+30+30+30+30+30)=16.$$
    Therefore, we have that $\mathcal G(\mathbb{S}^1\times \mathbb{S}^1\times \mathbb{S}^1\times \mathbb{S}^1)\le \rho(G_2^\prime)\le \rho_{\varepsilon}(G_2^\prime)=16.$ But from Proposition \ref{lbrg}, we have that $\mathcal G(\mathbb{S}^1\times \mathbb{S}^1\times \mathbb{S}^1\times \mathbb{S}^1)\ge 16$. Hence, we get that $\mathcal G(\mathbb{S}^1\times \mathbb{S}^1\times \mathbb{S}^1\times \mathbb{S}^1)= \rho_\varepsilon(G_2^\prime)=16$.
\end{proof}

The isomorphism signature of the crystallization $G_2^\prime$ of $\mathbb{S}^1\times \mathbb{S}^1\times \mathbb{S}^1 \times \mathbb{S}^1$ obtained using Regina \cite{burton2013regina} is the following.

\noindent 
-c4bvvvvLLwMvvwwzwvzzwwAzLMzzQzvLLzvAzLwwMMPALMLPzMQQQQQzMQQLQQ\\QMQQQQPAvzzwMQvwQQPQAQQQQAMQQQQQQQQQQAPQQlaiamajaqawaxaBavaH\\aIaLazaRaKaSaMaTaDaYaVaOaFa2a0a5a6a9adb8aeb+afbkbhbabobmb4agbib5aibpbqbvb8a\\wb9axbsbdbAbcbfbybjbpbjbbbcbubcb-abb-albebabxbnblbgbnb4aybDbFb6a4abbjbabCbCbE\\bgbEbqbdbmbqbubnbgbrbfbrbzbbbzb6aubIbNbObPbKbSbQbMbPbQbVbXbUbUbWbWbIb\\IbMbJbJbRbRbMbVbJbYbNbYbMbPbXbRbLbKbKbNbVbObLbPbTbQbLb1b1bTbSbQbS\\bTbWbUbZbZbObIb0bTb0bLbXbZb0bKbJbYbVb1bUb3b3b2b2b2b3b2b3baaaaaaaaaaaaaaa\\aaaaaaaaaaaaaaaaaaaaaaaaaaaaaaaaaaaaaaaaaaaaaaaaaaaaaaaaaaaaaaaaaaaaaaaaaaaaaaaaa\\aaaaaaaaaaaaaaaaaaaaaaaaaaaaaaaaaaaaaaaaaaaaaaaaaaaaaaaaaaaaaaaaaaaaaaaaaaaaaaaaa\\aaaaaaaaaaaaaaaaaaaaaaaaaaaaaaaaaaaaaaaaaaaaaaaaaaaaaaaaaaaaaaaaaaaaaaaaaaaaaaaaa\\aaaaaaaaaaaaaaaaaaaaaaaaaaaaaaaaaaaaaaaaaaaaaaaaaaaaaaaaaaaaaaaaaaaaaaaaaaaaaaaaa\\aaaaaaaaaaaaaaaaaaaaaaa

\begin{figure}[h!]
\tikzstyle{ver}=[]
\tikzstyle{verti}=[circle, draw, fill=black!100, inner sep=0pt, minimum width=2.5pt]
\tikzstyle{edge} = [draw,thick,-]
    \centering
\begin{tikzpicture}[scale=0.39]
\begin{scope}[shift={(15,-13.5)}]
\begin{scope}[shift={(0,0)},rotate=-30]
\foreach \x/\y in {120/0,180/5,240/4,300/3,0/2,60/1}
{\node[verti] (\y) at (\x:2){};}

\foreach \x/\y in {120/0,180/5,240/4,60/1}
{\node[ver] () at (\x:1.3){\tiny{$v_{\y}^{D^\prime}$}};}

\foreach \x/\y in {300/7,0/6}
{\node[ver] () at (\x:1.3){\tiny{$v_{\y}^{A^\prime}$}};}

\foreach \x/\y in 
{0/1,4/5}
{\path[edge,dashed] (\x)--(\y);}

\foreach \x/\y in 
{2/1,4/3,0/5}
{\path[edge] (\x)--(\y);}

\end{scope}

\begin{scope}[shift={(-5.3,0)},rotate=-30]
\foreach \x/\y in {120/20,180/21,240/22,300/23,0/18,60/19}
{\node[verti] (\y) at (\x:2){};}

\foreach \x/\y in {120/20,180/21,0/18,60/19}
{\node[ver] () at (\x:1.3){\tiny{$v_{\y}^{D^\prime}$}};}

\foreach \x/\y in {240/9,300/8}
{\node[ver] () at (\x:1.3){\tiny{$v_{\y}^{A^\prime}$}};}

\foreach \x/\y in 
{18/19,20/21}
{\path[edge,dashed] (\x)--(\y);}

\foreach \x/\y in 
{20/19,21/22,23/18}
{\path[edge] (\x)--(\y);}

\end{scope}

\begin{scope}[shift={(3,4)},rotate=-30]
\foreach \x/\y in {120/16,180/17,240/12,300/13,0/14,60/15}
{\node[verti] (\y) at (\x:2){};}

\foreach \x/\y in {120/16,180/17,240/12,300/13}
{\node[ver] () at (\x:1.3){\tiny{$v_{\y}^{D^\prime}$}};}

\foreach \x/\y in {0/11,60/10}
{\node[ver] () at (\x:1.3){\tiny{$v_{\y}^{A^\prime}$}};}

\foreach \x/\y in 
{12/13,16/17}
{\path[edge,dashed] (\x)--(\y);}

\foreach \x/\y in 
{13/14,15/16,17/12}
{\path[edge] (\x)--(\y);}

\end{scope}

\foreach \x/\y in 
{0/12,1/13,4/18,5/19,17/20}
{\path[edge,dotted] (\x)--(\y);}

\draw [edge,dotted] plot [smooth,tension=1.5] coordinates{(21) (-2.8,4.5) (16)};

\foreach \x/\y in 
{3/23,2/14}
{\path [edge,dotted] (\x)--(\y);}

\draw[edge,dotted] plot [smooth,tension=1.5] coordinates{(22) (2,-3) (15)};

\end{scope}

\begin{scope}[shift={(0,-13.5)}]
\begin{scope}[shift={(0,0)},rotate=-30]
\foreach \x/\y in {240/4,300/3,0/2,60/1}
{\node[verti] (\y) at (\x:2){};}

\foreach \x/\y in {300/0,0/5}
{\node[ver] () at (\x:1.3){\tiny{$v_{\y}^{C^\prime}$}};}

\foreach \x/\y in {240/2,60/3}
{\node[ver] () at (\x:1.3){\tiny{$v_{\y}^{A}$}};}

\path[edge] (2)--(3);

\foreach \x/\y in 
{2/1,4/3}
{\path[edge,dashed] (\x)--(\y);}

\end{scope}

\begin{scope}[shift={(-5.3,0)},rotate=-30]
\foreach \x/\y in {180/21,240/22,300/23,0/18}
{\node[verti] (\y) at (\x:2){};}

\foreach \x/\y in {240/19,300/20}
{\node[ver] () at (\x:1.3){\tiny{$v_{\y}^{C^\prime}$}};}

\foreach \x/\y in {180/23,0/22}
{\node[ver] () at (\x:1.3){\tiny{$v_{\y}^{A}$}};}

\path[edge] (22)--(23);

\foreach \x/\y in 
{21/22,23/18}
{\path[edge,dashed] (\x)--(\y);}

\end{scope}

\begin{scope}[shift={(3,4)},rotate=-30]
\foreach \x/\y in {120/16,300/13,0/14,60/15}
{\node[verti] (\y) at (\x:2){};}

\foreach \x/\y in {0/17,60/12}
{\node[ver] () at (\x:1.3){\tiny{$v_{\y}^{C^\prime}$}};}

\foreach \x/\y in {120/14,300/15}
{\node[ver] () at (\x:1.3){\tiny{$v_{\y}^{A}$}};}

\path[edge] (14)--(15);

\foreach \x/\y in 
{13/14,15/16}
{\path[edge,dashed] (\x)--(\y);}

\end{scope}

\begin{scope}[shift={(3,-4)},rotate=-30]
\foreach \x/\y in {120/6,180/7,240/8,300/9,0/10,60/11}
{\node[verti] (\y) at (\x:2){};}

\foreach \x/\y in {120/9,180/10,240/11,300/6,0/7,60/8}
{\node[ver] () at (\x:1.3){\tiny{$v_{\y}^{D}$}};}

\foreach \x/\y in 
{6/7,8/9,10/11}
{\path[edge] (\x)--(\y);}

\foreach \x/\y in 
{7/8,9/10,11/6}
{\path[edge,dashed] (\x)--(\y);}

\end{scope}

\foreach \x/\y in 
{1/13,2/6,3/7,4/18,11/14,8/23}
{\draw [line width=2pt, line cap=round, dash pattern=on 0pt off 1.7\pgflinewidth] (\x)--(\y);}

\draw[line width=2pt, line cap=round, dash pattern=on 0pt off 1.7\pgflinewidth] plot [smooth,tension=1.5] coordinates{(21) (-2.8,4.5) (16)};

\draw[line width=2pt, line cap=round, dash pattern=on 0pt off 1.7\pgflinewidth] plot [smooth,tension=1.5] coordinates{(15) (6,0) (10)};

\draw[line width=2pt, line cap=round, dash pattern=on 0pt off 1.7\pgflinewidth] plot[smooth,tension=1.5] coordinates{(22) (-2.8,-4.5) (9)};

\end{scope}

\begin{scope}[shift={(0,0)}]

\begin{scope}[shift={(-5.3,0)},rotate=-30]
\foreach \x/\y in {120/20,180/21,240/22,300/23,0/18,60/19}
{\node[verti] (\y) at (\x:2){};}

\foreach \x/\y in {120/0,300/20}
{\node[ver] () at (\x:1.3){\tiny{$v_{\y}^{B^\prime}$}};}

\foreach \x/\y in {0/14,60/15}
{\node[ver] () at (\x:1.3){\tiny{$v_{\y}^{D}$}};}

\foreach \x/\y in {180/1,240/21}
{\node[ver] () at (\x:1.3){\tiny{$v_{\y}^{A}$}};}

\foreach \x/\y in 
{18/19,20/21,22/23}
{\path[edge,dashed] (\x)--(\y);}

\foreach \x/\y in 
{20/19,21/22,23/18}
{\draw[line width=2pt, line cap=round, dash pattern=on 0pt off 1.7\pgflinewidth] (\x)--(\y);}

\end{scope}

\begin{scope}[shift={(3,4)},rotate=-30]
\foreach \x/\y in {120/16,180/17,240/12,300/13,0/14,60/15}
{\node[verti] (\y) at (\x:2){};}

\foreach \x/\y in {180/12,0/19}
{\node[ver] () at (\x:1.3){\tiny{$v_{\y}^{B^\prime}$}};}

\foreach \x/\y in {240/3,300/2}
{\node[ver] () at (\x:1.3){\tiny{$v_{\y}^{D}$}};}

\foreach \x/\y in {120/13,60/18}
{\node[ver] () at (\x:1.3){\tiny{$v_{\y}^{A}$}};}

\foreach \x/\y in 
{12/13,14/15,16/17}
{\path[edge,dashed] (\x)--(\y);}

\foreach \x/\y in 
{13/14,15/16,17/12}
{\draw [line width=2pt, line cap=round, dash pattern=on 0pt off 1.7\pgflinewidth] (\x)--(\y);}

\end{scope}

\begin{scope}[shift={(3,-4)},rotate=-30]
\foreach \x/\y in {120/6,180/7,240/8,300/9,0/10,60/11}
{\node[verti] (\y) at (\x:2){};}

\foreach \x/\y in {240/17,60/5}
{\node[ver] () at (\x:1.3){\tiny{$v_{\y}^{B^\prime}$}};}

\foreach \x/\y in {120/22,180/23}
{\node[ver] () at (\x:1.3){\tiny{$v_{\y}^{D}$}};}

\foreach \x/\y in {300/16,0/4}
{\node[ver] () at (\x:1.3){\tiny{$v_{\y}^{A}$}};}

\foreach \x/\y in 
{6/7,8/9,10/11}
{\path[edge,dashed] (\x)--(\y);}

\foreach \x/\y in 
{7/8,9/10,11/6}
{\draw [line width=2pt, line cap=round, dash pattern=on 0pt off 1.7\pgflinewidth] (\x)--(\y);}

\end{scope}

\foreach \x/\y in 
{17/20,11/14,8/23}
{\draw [line width=2pt, line cap=rectangle, dash pattern=on 1pt off 1] (\x)--(\y);}

\draw[line width=2pt, line cap=rectangle, dash pattern=on 1pt off 1] plot [smooth,tension=1.5] coordinates{(21) (-2.8,4.5) (16)};

\draw[line width=2pt, line cap=rectangle, dash pattern=on 1pt off 1] plot [smooth,tension=1.5] coordinates{(15) (6,0) (10)};

\draw[line width=2pt, line cap=rectangle, dash pattern=on 1pt off 1] plot[smooth,tension=1.5] coordinates{(22) (-2.8,-4.5) (9)};
\end{scope}

\begin{scope}[shift={(15,0)}]
\begin{scope}[shift={(0,0)},rotate=-30]
\foreach \x/\y in {120/0,180/5,240/4,300/3,0/2,60/1}
{\node[verti] (\y) at (\x:2){};}

\foreach \x/\y in {120/19,180/5,240/9,60/6}
{\node[ver] () at (\x:1.3){\tiny{$v_{\y}^{A^\prime}$}};}

\foreach \x/\y in {300/21,0/1}
{\node[ver] () at (\x:1.3){\tiny{$v_{\y}^{D^\prime}$}};}

\foreach \x/\y in 
{0/1,4/5}
{\draw [line width=2pt, line cap=round, dash pattern=on 0pt off 1.7\pgflinewidth] (\x)--(\y);}

\foreach \x/\y in 
{2/1,4/3,0/5}
{\draw [line width=2pt, line cap=rectangle, dash pattern=on 1pt off 1] (\x)--(\y);}

\end{scope}

\begin{scope}[shift={(-5.3,0)},rotate=-30]
\foreach \x/\y in {120/20,180/21,240/22,300/23,0/18,60/19}
{\node[verti] (\y) at (\x:2){};}

\foreach \x/\y in {120/12,180/7,0/10,60/0}
{\node[ver] () at (\x:1.3){\tiny{$v_{\y}^{A^\prime}$}};}

\foreach \x/\y in {240/4,300/16}
{\node[ver] () at (\x:1.3){\tiny{$v_{\y}^{D^\prime}$}};}

\foreach \x/\y in 
{18/19,20/21}
{\draw[line width=2pt, line cap=round, dash pattern=on 0pt off 1.7\pgflinewidth] (\x)--(\y);}

\foreach \x/\y in 
{20/19,21/22,23/18}
{\draw [line width=2pt, line cap=rectangle, dash pattern=on 1pt off 1](\x)--(\y);}

\end{scope}

\begin{scope}[shift={(3,4)},rotate=-30]
\foreach \x/\y in {120/16,180/17,240/12,300/13,0/14,60/15}
{\node[verti] (\y) at (\x:2){};}

\foreach \x/\y in {120/8,180/17,240/20,300/11}
{\node[ver] () at (\x:1.3){\tiny{$v_{\y}^{A^\prime}$}};}

\foreach \x/\y in {0/13,60/18}
{\node[ver] () at (\x:1.3){\tiny{$v_{\y}^{D^\prime}$}};}

\foreach \x/\y in 
{12/13,16/17}
{\draw [line width=2pt, line cap=round, dash pattern=on 0pt off 1.7\pgflinewidth] (\x)--(\y);}

\foreach \x/\y in 
{13/14,15/16,17/12}
{\draw [line width=2pt, line cap=rectangle, dash pattern=on 1pt off 1] (\x)--(\y);}

\end{scope}

\foreach \x/\y in 
{0/12,1/13,4/18,5/19,17/20}
{\path [edge,dotted] (\x)--(\y);}

\draw[edge,dotted] plot [smooth,tension=1.5] coordinates{(21) (-2.8,4.5) (16)};

\foreach \x/\y in 
{3/23,2/14}
{\path [edge,dotted] (\x)--(\y);}

\draw[edge,dotted] plot [smooth,tension=1.5] coordinates{(22) (2,-3) (15)};

\end{scope}

\end{tikzpicture}
\caption{Subgraphs of the crystallization $G_2^2$.}\label{fig:G2c}
\end{figure}

 \begin{figure}[h!]
\tikzstyle{ver}=[]
\tikzstyle{verti}=[circle, draw, fill=black!100, inner sep=0pt, minimum width=2.5pt]
\tikzstyle{edge} = [draw,thick,-]
    \centering
\begin{tikzpicture}[scale=0.4]
\begin{scope}[shift={(15,0)}]
\begin{scope}[shift={(0,0)},rotate=-30]
\foreach \x/\y in {120/0,180/5,240/4,300/3,0/2,60/1}
{\node[verti] (\y) at (\x:2){};}

\foreach \x/\y in {300/5,0/19}
{\node[ver] () at (\x:1.35){\tiny{$v_{\y}^{B^\prime}$}};}

\foreach \x/\y in {240/5,60/19}
{\node[ver] () at (\x:1.35){\tiny{$v_{\y}^{C^\prime}$}};}

\foreach \x/\y in {120/19,180/5}
{\node[ver] () at (\x:1.35){\tiny{$v_{\y}^{D^\prime}$}};}

\foreach \x/\y in 
{0/1,2/3,4/5}
{\draw [line width=2pt, line cap=rectangle, dash pattern=on 1pt off 1] (\x)--(\y);}

\foreach \x/\y in 
{2/1,4/3,0/5}
{\path[edge,dotted] (\x)--(\y);}

\end{scope}

\begin{scope}[shift={(-5.3,0)},rotate=-30]
\foreach \x/\y in {120/20,180/21,240/22,300/23,0/18,60/19}
{\node[verti] (\y) at (\x:2){};}

\foreach \x/\y in {240/12,300/0}
{\node[ver] () at (\x:1.35){\tiny{$v_{\y}^{B^\prime}$}};}

\foreach \x/\y in {180/12,0/0}
{\node[ver] () at (\x:1.35){\tiny{$v_{\y}^{C^\prime}$}};}

\foreach \x/\y in {120/12,60/0}
{\node[ver] () at (\x:1.35){\tiny{$v_{\y}^{D^\prime}$}};}

\foreach \x/\y in 
{18/19,20/21,22/23}
{\draw  [line width=2pt, line cap=rectangle, dash pattern=on 1pt off 1.5]  (\x)--(\y);}

\foreach \x/\y in 
{20/19,21/22,23/18}
{\path[edge,dotted](\x)--(\y);}

\end{scope}

\begin{scope}[shift={(3,4)},rotate=-30]
\foreach \x/\y in {120/16,180/17,240/12,300/13,0/14,60/15}
{\node[verti] (\y) at (\x:2){};}

\foreach \x/\y in {0/20,60/17}
{\node[ver] () at (\x:1.35){\tiny{$v_{\y}^{B^\prime}$}};}

\foreach \x/\y in {120/17,300/20}
{\node[ver] () at (\x:1.35){\tiny{$v_{\y}^{C^\prime}$}};}

\foreach \x/\y in {180/17,240/20}
{\node[ver] () at (\x:1.35){\tiny{$v_{\y}^{D^\prime}$}};}

\foreach \x/\y in 
{12/13,14/15,16/17}
{\draw [line width=2pt, line cap=rectangle, dash pattern=on 1pt off 1] (\x)--(\y);}

\foreach \x/\y in 
{13/14,15/16,17/12}
{\path [edge,dotted] (\x)--(\y);}

\end{scope}

\begin{scope}[shift={(3,-4)},rotate=-30]
\foreach \x/\y in {120/6,180/7,240/8,300/9,0/10,60/11}
{\node[verti] (\y) at (\x:2){};}

\foreach \x/\y in {120/19,180/5,240/0,300/12,0/17,60/20}
{\node[ver] () at (\x:1.35){\tiny{$v_{\y}^{A^\prime}$}};}

\foreach \x/\y in 
{7/8,9/10,11/6}
{\path [edge,dotted] (\x)--(\y);}

\foreach \x/\y in 
{6/7,8/9,10/11}
{\draw [line width=2pt, line cap=rectangle, dash pattern=on 1pt off 1] (\x)--(\y);}

\end{scope}

\foreach \x/\y in 
{0/12,1/13,2/6,3/7,4/18,5/19,17/20,11/14,8/23}
{\path [edge] (\x)--(\y);}

\draw[edge] plot [smooth,tension=1.5] coordinates{(21) (-2.8,4.5) (16)};

\draw[edge] plot [smooth,tension=1.5] coordinates{(15) (6,0) (10)};

\draw[edge] plot[smooth,tension=1.5] coordinates{(22) (-2.8,-4.5) (9)};
\end{scope}

\begin{scope}[shift={(0,0)}]

\path[edge] (-1,7)--(-1,-6);
\path[edge] (-7,7)--(-7,-6);
\path[edge] (5,7)--(5,-6);
\path[edge] (-7,7)--(5,7);
\path[edge] (-7,-6)--(5,-6);

 \node[ver] (1) at (-6,6){\tiny{$v_{19}^{D^\prime}$}}; 
 \node[ver] (2) at (-6,5){\tiny{$v_{19}^{C^\prime}$}};
 \node[ver] (3) at (-6,4){\tiny{$v_{19}^{B^\prime}$}};
 \node[ver] (4) at (-6,3){\tiny{$v_{5}^{B^\prime}$}};
 \node[ver] (5) at (-6,2){\tiny{$v_{5}^{C^\prime}$}};
 \node[ver] (6) at (-6,1){\tiny{$v_{5}^{D^\prime}$}};
 \node[ver] (7) at (-6,0){\tiny{$v_{0}^{D^\prime}$}};
 \node[ver] (8) at (-6,-1){\tiny{$v_{0}^{C^\prime}$}};
 \node[ver] (9) at (-6,-2){\tiny{$v_{0}^{B^\prime}$}};
 \node[ver] (10) at (-6,-3){\tiny{$v_{12}^{B^\prime}$}};
 \node[ver] (11) at (-6,-4){\tiny{$v_{12}^{C^\prime}$}};
 \node[ver] (12) at (-6,-5){\tiny{$v_{12}^{D^\prime}$}};

 \node[ver] (1') at (-2,6){\tiny{$v_{8}^{A}$}}; 
 \node[ver] (2') at (-2,5){\tiny{$v_{23}^{A}$}};
 \node[ver] (3') at (-2,4){\tiny{$v_{18}^{A}$}};
 \node[ver] (4') at (-2,3){\tiny{$v_{4}^{A}$}};
 \node[ver] (5') at (-2,2){\tiny{$v_{3}^{A}$}};
 \node[ver] (6') at (-2,1){\tiny{$v_{7}^{A}$}};
 \node[ver] (7') at (-2,0){\tiny{$v_{6}^{A}$}};
 \node[ver] (8') at (-2,-1){\tiny{$v_{2}^{A}$}};
 \node[ver] (9') at (-2,-2){\tiny{$v_{1}^{A}$}};
 \node[ver] (10') at (-2,-3){\tiny{$v_{13}^{A}$}};
 \node[ver] (11') at (-2,-4){\tiny{$v_{14}^{A}$}};
 \node[ver] (12') at (-2,-5){\tiny{$v_{11}^{A}$}};

 \node[ver] (13) at (0,6){\tiny{$v_{20}^{B^\prime}$}};
 \node[ver] (14) at (0,5){\tiny{$v_{20}^{C^\prime}$}};
 \node[ver] (15) at (0,4){\tiny{$v_{20}^{D^\prime}$}};
 \node[ver] (16) at (0,3){\tiny{$v_{17}^{D^\prime}$}};
 \node[ver] (17) at (0,2){\tiny{$v_{17}^{C^\prime}$}};
 \node[ver] (18) at (0,1){\tiny{$v_{17}^{B^\prime}$}};
 \node[ver] (19) at (0,0){\tiny{$v_{19}^{A^\prime}$}};
 \node[ver] (20) at (0,-1){\tiny{$v_{20}^{A^\prime}$}};
 \node[ver] (21) at (0,-2){\tiny{$v_{17}^{A^\prime}$}};
 \node[ver] (22) at (0,-3){\tiny{$v_{12}^{A^\prime}$}};
 \node[ver] (23) at (0,-4){\tiny{$v_{0}^{A^\prime}$}};
 \node[ver] (24) at (0,-5){\tiny{$v_{5}^{A^\prime}$}};

  \node[ver] (13') at (4,6){\tiny{$v_{21}^{A}$}};
 \node[ver] (14') at (4,5){\tiny{$v_{22}^{A}$}};
 \node[ver] (15') at (4,4){\tiny{$v_{9}^{A}$}};
 \node[ver] (16') at (4,3){\tiny{$v_{10}^{A}$}};
 \node[ver] (17') at (4,2){\tiny{$v_{15}^{A}$}};
 \node[ver] (18') at (4,1){\tiny{$v_{16}^{A}$}};
 \node[ver] (19') at (4,0){\tiny{$v_{19}^{A}$}};
 \node[ver] (20') at (4,-1){\tiny{$v_{20}^{A}$}};
 \node[ver] (21') at (4,-2){\tiny{$v_{17}^{A}$}};
 \node[ver] (22') at (4,-3){\tiny{$v_{12}^{A}$}};
 \node[ver] (23') at (4,-4){\tiny{$v_{0}^{A}$}};
 \node[ver] (24') at (4,-5){\tiny{$v_{5}^{A}$}};

 \foreach \x/\y in
{1/1',2/2',3/3',4/4',5/5',6/6',7/7',8/8',9/9',10/10',11/11',12/12',13/13',14/14',15/15',16/16',17/17',18/18',19/19',20/20',21/21',22/22',23/23',24/24'}
 {\path[edge,dashed] (\x)--(\y);}
 \end{scope}

\begin{scope}[shift={(-5,9)}]
   
\node[ver] (D) at (21,0){$\Gamma(\_,1,2,3)$};
\node[ver] (C) at (14,0){$\Gamma(4,\_,2,3)$};
\node[ver] (B) at (7,0){$\Gamma(4,0,\_,3)$};
\node[ver] (A) at (0,0){$\Gamma(4,0,1,\_)$};
 
\node[ver] () at (21,1){$D$};
\node[ver] () at (14,1){$C$};
\node[ver] () at (7,1){$B$};
\node[ver] () at (0,1){$A$};

\path[edge,dotted] (B)--(C);
\path[edge] (A)--(B);

\draw[line width=2pt, line cap=rectangle, dash pattern=on 1pt off 1] (C) -- (D);
\end{scope}

\draw [line width=2pt, line cap=round, dash pattern=on 0pt off 1.7\pgflinewidth] (D)--(16.1,7.2);

\end{tikzpicture}
\caption{Crystallization $G_2^\prime$ of $\mathbb{S}^1\times \mathbb{S}^1\times \mathbb{S}^1\times \mathbb{S}^1$ with $120$ vertices.}\label{fig:G2d}
\end{figure}
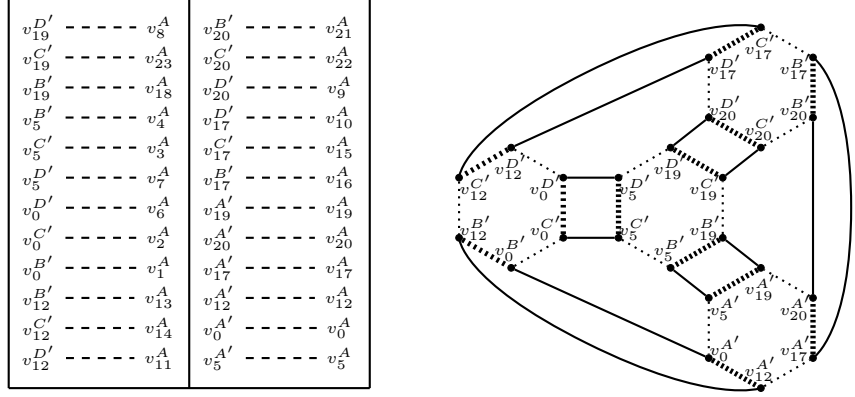

 Consider the $2$-cube. Let us represent it by $(a_1^0,a_2^1,a_3^1,a_4^2)$. Here $a_3^1=a_1^0\times I$ and $a_4^2=a_2^1\times I$, where $I=[0,1]$. By $a_i^j$, we mean that it is the $i^{\text{th}}$ vertex of the cube and it is colored by the color $j$. Considering the $2$-cube $\times I$, we represent the $3$-cube by $(a_1^0,a_2^1,a_3^1,a_4^2,a_5^1,a_6^2,a_7^2,a_8^3)$, where $a_i^j=a_{i-4}^{j-1}\times I$ for $5\le i\le 8$ and $1\le j\le 3$. Therefore, we represent the $(n+1)$-cube by $(a_1^0,a_2^1,\cdots,a_{2^n}^n,a_{2^n+1}^1,a_{2^n+2}^2,\cdots,a_{2^{n+1}}^{n+1})$ for $n\ge 1$. Consider the triangulation of the $(n+1)$-cube that consists of the $(n+1)$-simplices whose vertices form a path of length $n+1$ from $a_1^0$ to $a_{2^{n+1}}^{n+1}$. This triangulation contains as many $(n+1)$-simplices as many paths of length $n+1$ from $a_1^0$ to $a_{2^{n+1}}^{n+1}$. Thus, the number of $(n+1)$-simplices is $(n+1)!$ in this triangulation of the $(n+1)$-cube. Now, consider the disjoint link of the vertex $a_{2^{n+1}}^{n+1}$ in this colored triangulation of the $(n+1)$-cube. The disjoint link of $a_{2^{n+1}}^{n+1}$ is a colored triangulation of the $n$-ball, and the number of $n$-simplices in this colored triangulation is $(n+1)!$. Since this is a colored triangulation, we can obtain an $n$-regular colored graph that is dual to this colored triangulation of the $n$-ball (cf. Subsection \ref{crystal}). Let us denote this $n$-regular colored graph representing the $n$-ball by $\Gamma$. Now, we will introduce $0$-colored edges to the graph $\Gamma$ to make it an $(n+1)$-regular colored graph. Join two vertices $v_i$ and $v_j$ of $\Gamma$ by a $0$-colored edge if these are connected via a path $n,n-1,n-2,\cdots,1,2,3,\cdots,n$ of the colored edges in $\Gamma$. Thus, we get an $(n+1)$-regular colored graph, say $\Gamma^\prime$. For $1\le n\le 4$, we get that the graph $\Gamma^\prime$ represents $\mathbb{S}^1 \times \mathbb{S}^1 \times \cdots \times \mathbb{S}^1$ ($n$ times). For $n=4$, we verified it using Regina. The isomorphism signature of $\Gamma^\prime$ is the same as the isomorphism signature of $G_2^\prime$ (cf. Theorem \ref{theorem:G2}). It follows from the construction of $\Gamma^\prime$ that for $n\ge 4$ and even (resp. odd), the length of a bi-colored cycle colored by any two consecutive colors in the cyclic permutation $(0,2,4,\cdots, n,1,3,5,\cdots,n-1)$ (resp. $(0,2,4,\cdots,n-1,1,n,n-2,n-4,\cdots,3)$) is $4$. We also have that any bi-colored cycle of $\Gamma^\prime$ is of length $4$ or $6$. Hence, the regular genus of $\Gamma^\prime$ is $1+\frac{(n+1)! \  (n-3)}{8}$ for $n\ge 4$. This lead us to the following conjecture.  

\begin{conj}\label{conj}
    The $(n+1)$-regular colored graph $\Gamma^\prime$ with $(n+1)!$ vertices is a genus-minimal crystallization of the $n$-dimensional torus $\mathbb{S}^1 \times \mathbb{S}^1 \times \cdots \times \mathbb{S}^1$ ($n$ times) with regular genus $1+\frac{(n+1)! \  (n-3)}{8}$, for $n\ge 5$.
\end{conj}

\section{Genus-minimal crystallizations of small covers over $\Delta^2 \times \Delta^2$}\label{Section:4}

In this section, we will study all the small covers over the simple polytope $\Delta^2 \times \Delta^2$ up to D-J equivalence. We construct colored triangulations of these small covers such that the regular genera of the dual colored graphs attain the lower bounds for the regular genera of these small covers presented in \cite{bc17}. 

We will denote the simple polytope $\Delta^2 \times \Delta^2$ by $P$, one of the triangles by $A=[a_0,a_1,a_2]$, the other by $B=[b_0,b_1,b_2]$ and the vertex $(a_i,b_j)\in \Delta^2 \times \Delta^2=A \times B$ by $v_{i+j}^j$, where $i,j \in \{0,1,2\}$. Let the set of all the $3$-faces of $P$ be $$\mathcal{F}(P)=\{F_1=(v_0^0,v_1^0,v_2^0,v_1^1,v_2^1,v_3^1), F_2=(v_0^0,v_1^0,v_2^0,v_2^2,v_3^2,v_4^2), F_3=(v_0^0,v_1^1,v_2^2,v_1^0,v_2^1,v_3^2), $$ $$F_4=(v_0^0,v_1^1,v_2^2,v_2^0,v_3^1,v_4^2), F_5=(v_1^1,v_2^1,v_3^1,v_2^2,v_3^2,v_4^2), F_6=(v_1^0,v_2^1,v_3^2,v_2^0, v_3^1,v_4^2)\}.$$ The following lemma gives all the $\mathbb{Z}_2$-characteristic functions on the set $\mathcal{F}(P)$.

\begin{lemma}\label{lemma:seven}
    There are seven small covers over the simple polytope $P=\Delta^2\times \Delta^2$ up to D-J equivalence of which one is $\mathbb{RP}^2\times \mathbb{RP}^2$, and others are $\mathbb{RP}^2$-bundles over $\mathbb{RP}^2$.
\end{lemma}
 \begin{proof}
   Let $\lambda:\mathcal F(P) \to \mathbb{Z}_2^4$ be a $\mathbb{Z}_2$-characteristic function. Since the vertex $v_0^0=\cap_{i=1}^4 F_i$, the set $\{\lambda(F_i)\ |\ 1\le i\le 4\}$ is a basis of $\mathbb{Z}_2^4$. Let $\lambda(F_i)=e_i$, for $1\le i\le 4$. Considering the vertices $v_1^0$ and $v_2^0$, we get that $\lambda(F_6)= (c_1,c_2,1,1),$ 
   where $c_1,c_2\in \{0,1\}$. Similarly, when we consider vertices $v_1^1$ and $v_2^2$, then we get $\lambda(F_5)=(1,1,c_3,c_4)$, where $c_3,c_4\in \{0,1\}$. Now, we claim that if $c_k=1$ for some $k\in \{1,2\}$, then $c_3=c_4=0$. To prove this, let us first assume that $c_1=c_3=1$. Considering the vertex $v_4^2$, we get that $\{e_2,e_4,\lambda(F_5),\lambda(F_6)\}$ is a basis of $\mathbb{Z}_2^4$. But since $c_1=c_3=1$, we have that $\lambda(F_6)\in \text{span}\{e_2,e_4,\lambda(F_5)\}$, which implies that $c_3=0$.
   Now, if $c_1=c_4=1$, considering the vertex $v_3^2$, we again get a contradiction. Therefore, when $c_1=1$, we get that $c_3=c_4=0$. Similarly, considering the vertices $v^1_2$ and $v^1_3$, we get that $c_3=c_4=0$ when $c_2=1$. This proves the above claim. Therefore, there are seven $\mathbb{Z}_2$-characteristic functions on $\mathcal F(P)$ corresponding to a fixed basis of $\mathbb{Z}_2^4$, i.e., there are seven small covers over the simple polytope $P$ up to D-J equivalence (cf. Subsection \ref{smallcover}). Let $\lambda_1$ be the $\mathbb{Z}_2$-characteristic function on $\mathcal F(P)$ such that $\lambda_1(F_5)=(1,1,0,0)$ and $\lambda_1(F_6)=(0,0,1,1)$. Let $\lambda_2, \lambda_3, \lambda_4$ be $\mathbb{Z}_2$-characteristic functions on $\mathcal F(P)$ such that $\lambda_i(F_5)=(1,1,0,0)$ and $\lambda_2(F_6)=(1,1,1,1),\ \lambda_3(F_6)=(1,0,1,1),\ \lambda_4(F_6)=(0,1,1,1)$, for $2\le i\le 4$. Also, let $\lambda_5, \lambda_6, \lambda_7$ be $\mathbb{Z}_2$-characteristic functions on $\mathcal F(P)$ such that  $\lambda_5(F_5)=(1,1,1,1),\ \lambda_6(F_5)=(1,1,1,0),\ \lambda_7(F_5)=(1,1,0,1)$ and $\lambda_i(F_6)=(0,0,1,1)$, for $5\le i\le 7$. Then, it follows from Subsection \ref{smallcover} that $M^4(\lambda_1)$ is $\mathbb{RP}^2\times \mathbb{RP}^2$. For each of the other six $\mathbb{Z}_2$-characteristic functions, the corresponding small cover is an $\mathbb{RP}^2$-bundle over $\mathbb{RP}^2$. Thus, all the small covers over the simple polytope $P$ are non-orientable.
\end{proof}

\begin{figure}[ht]
\tikzstyle{ver}=[]
\tikzstyle{edge} = [draw,thick,-]
\centering
\begin{tikzpicture}[scale=1]

\node[ver] () at (-6,5){$t_w^1=[v_0^0,v_1^0,v_2^0,v_3^1,v_4^2]_w$}; 
\node[ver] () at (-6,4.4){$0$};
\node[ver] () at (-3,4.4){$1$};
\node[ver] () at (0,4.4){$2$};
\node[ver] () at (3,4.4){$3$};
\node[ver] () at (6,4.4){$4$};

\node[ver] () at (-6,3.9){\footnotesize{$(c_1,c_2,1,1)$}};
\node[ver] () at (-3,3.9){\footnotesize{$e_4$}};
\node[ver] () at (3,3.9){\footnotesize{$e_2$}};
\node[ver] () at (6,3.9){\footnotesize{$e_1$}};

\node[ver] () at (-6,3){$t_w^2=[v_0^0,v_1^0,v_2^1,v_3^1,v_4^2]_w$}; 
\node[ver] () at (-6,2.4){$0$};
\node[ver] () at (-3,2.4){$1$};
\node[ver] () at (0,2.4){$2$};
\node[ver] () at (3,2.4){$3$};
\node[ver] () at (6,2.4){$4$};

\node[ver] () at (-6,1.9){\footnotesize{$(c_1,c_2,1,1)$}};
\node[ver] () at (6,1.9){\footnotesize{$e_1$}};

\node[ver] () at (-7,1.1){$t_w^4=[v_0^0,v_1^1,v_2^1,v_3^1,v_4^2]_w$}; 
\node[ver] () at (-7,0.4){$0$};
\node[ver] () at (-5.6,0.4){$1$};
\node[ver] () at (-4.2,0.4){$2$};
\node[ver] () at (-2.8,0.4){$3$};
\node[ver] () at (-1.4,0.4){$4$};

\node[ver] () at (-7,-0.1){\footnotesize{$(1,1,c_3,c_4)$}};
\node[ver] () at (-4.2,-0.1){\footnotesize{$e_4$}};
\node[ver] () at (-1.4,-0.1){\footnotesize{$e_1$}};

\node[ver] () at (2,1.1){$t_w^3=[v_0^0,v_1^0,v_2^1,v_3^2,v_4^2]_w$}; 
\node[ver] () at (7,0.4){$4$};
\node[ver] () at (5.6,0.4){$3$};
\node[ver] () at (4.2,0.4){$2$};
\node[ver] () at (2.8,0.4){$1$};
\node[ver] () at (1.4,0.4){$0$};

\node[ver] () at (7,-0.1){\footnotesize{$e_3$}};
\node[ver] () at (4.2,-0.1){\footnotesize{$e_2$}};
\node[ver] () at (1.4,-0.1){\footnotesize{$(c_1,c_2,1,1)$}};

\node[ver] () at (-6,-1){$t_w^5=[v_0^0,v_1^1,v_2^1,v_3^2,v_4^2]_w$}; 
\node[ver] () at (-6,-1.6){$4$};
\node[ver] () at (-3,-1.6){$3$};
\node[ver] () at (0,-1.6){$2$};
\node[ver] () at (3,-1.6){$1$};
\node[ver] () at (6,-1.6){$0$};

\node[ver] () at (-6,-2.1){\footnotesize{$e_3$}};
\node[ver] () at (6,-2.1){\footnotesize{$(1,1,c_3,c_4)$}};

\node[ver] () at (-6,-3){$t_w^6=[v_0^0,v_1^1,v_2^2,v_3^2,v_4^2]_w$}; 
\node[ver] () at (-6,-3.6){$4$};
\node[ver] () at (-3,-3.6){$3$};
\node[ver] () at (0,-3.6){$2$};
\node[ver] () at (3,-3.6){$1$};
\node[ver] () at (6,-3.6){$0$};

\node[ver] () at (-6,-4.1){\footnotesize{$e_3$}};
\node[ver] () at (-3,-4.1){\footnotesize{$e_4$}};
\node[ver] () at (3,-4.1){\footnotesize{$e_2$}};
\node[ver] () at (6,-4.1){\footnotesize{$(1,1,c_3,c_4)$}};

\path[edge] (0,4.1)--(0,2.7);
\path[edge] (-3,2.1)--(-5.6,0.7);
\path[edge] (3,2.1)--(5.6,0.7);
\path[edge] (-2.8,0.1)--(-3,-1.3);
\path[edge] (2.8,0.1)--(3,-1.3);
\path[edge] (0,-1.9)--(0,-3.3);

\path[edge] (-9,3.65)--(7.2,3.65);
\path[edge] (-9,1.65)--(7.2,1.65);
\path[edge] (-9,-0.35)--(7.2,-0.35);
\path[edge] (-9,-2.35)--(7.2,-2.35);
\path[edge] (-0.6,1.65)--(-0.6,-0.35);

\node[ver] () at (0.8,3.4){\scriptsize{identified}};
\node[ver] () at (-3.4,1.4){\scriptsize{identified}};
\node[ver] () at (5.4,1.4){\scriptsize{identified}};
\node[ver] () at (3.6,-0.6){\scriptsize{identified}};
\node[ver] () at (-2.2,-0.6){\scriptsize{identified}};
\node[ver] () at (0.8,-2.6){\scriptsize{identified}};

\end{tikzpicture}
\caption{$t_w^j$ with all its $3$-faces and their $\mathbb Z_2$-characteristic vectors, for $1\le j\le 6$. }\label{fig:T}
\end{figure}
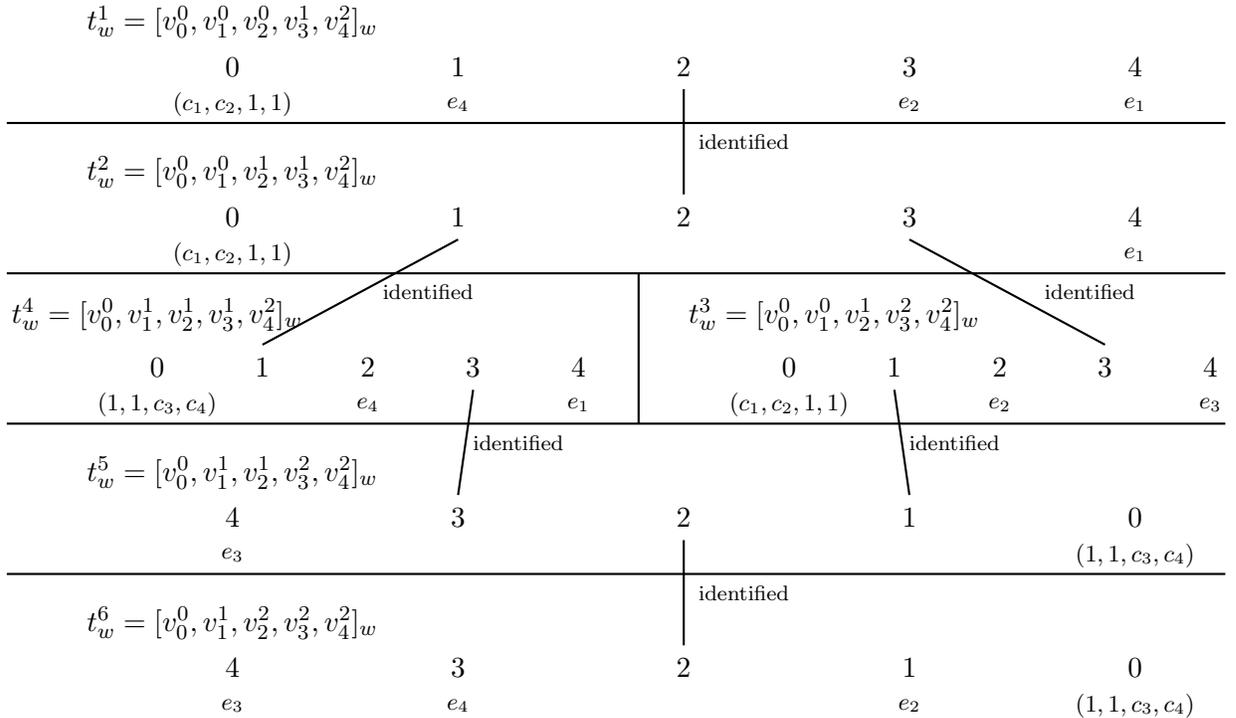

\begin{figure}[h!]  
    \tikzstyle{ver}=[]
\tikzstyle{verti}=[circle, draw, fill=black!100, inner sep=0pt, minimum width=3pt]
\tikzstyle{vert}=[circle, draw, fill=black!100, inner sep=0pt, minimum width=1pt]
\tikzstyle{edge} = [draw,thick,-]
\centering

\begin{tikzpicture}[scale=0.4]

\begin{scope}[shift={(-16,16)}]
\foreach \x/\y/\z in
{1.5/1.5/1,1.5/-1.5/2,-1.5/-1.5/3,-1.5/1.5/4,3/3/5,3/-3/6,-3/-3/7,-3/3/8,4.5/4.5/9,4.5/-4.5/10,-4.5/-4.5/11,-4.5/4.5/12,6/6/13,6/-6/14,-6/-6/15,-6/6/16}
{\node[verti] (a\z) at (\x,\y){};}

\foreach \x/\y/\z in
{1.1/2/0,1.1/-2/34,-1.1/-2/134,-1.1/2/1,2.6/3.5/4,2.6/-3.5/3,-2.6/-3.5/13,-2.6/3.5/14,4.1/5/24,4.1/-5/23,-4.1/-5/123,-4.1/5/124,5.6/6.5/2,5.6/-6.5/234,-5.6/-6.5/1234,-5.6/6.5/12}
{\node[ver] () at (\x,\y){\tiny{$T_{\z}^1$}};}

\foreach \x/\y in 
{5/9,6/10,7/11,8/12}
{\path[edge,dashed] (a\x)--(a\y);}

\foreach \x/\y in 
{1/4,2/3,6/7,8/5,10/11,12/9,14/15,16/13}
{\draw [line width=2pt, line cap=round, dash pattern=on 0pt off 1.7\pgflinewidth]  (a\x) -- (a\y);}

\foreach \x/\y in 
{1/2,4/3,6/5,8/7,10/9,12/11,14/13,16/15}
{\draw[line width=2pt, line cap=rectangle, dash pattern=on 1pt off 1]
(a\x) -- (a\y);}

\foreach \x/\y in 
{5/1,6/2,7/3,8/4,9/13,10/14,11/15,12/16}
{\path[edge,dotted] (a\x)--(a\y);}

\draw[edge,dashed] plot [smooth,tension=0.5] coordinates{(a1) (4.2,2.5) (a13)};

\draw[edge,dashed] plot [smooth,tension=0.5] coordinates{(a2) (4.2,-2.5) (a14)};

\draw[edge,dashed] plot [smooth,tension=0.5] coordinates{(a3) (-4.2,-2.5) (a15)};

\draw[edge,dashed] plot [smooth,tension=0.5] coordinates{(a4) (-4.2,2.5) (a16)};

\end{scope}

\begin{scope}[shift={(-6,16)}]
\foreach \x/\y/\z in
{-1/1/1,1/1/2,-1/0/3,1/0/4,-1/-1/5,1/-1/6}
{\node[vert] (\z) at (\x,\y){};}

\foreach \x/\y/\z in
{-1.5/1.5/1,1.5/1.5/2,-1.5/0/3,1.5/0/4,-1.5/-1.5/5,1.5/-1.5/6}
{\node[ver] () at (\x,\y){\tiny{$T_w^{\z}$}};}

\foreach \x/\y in
{1/2,5/6}
{\path[edge] (\x)--(\y);}

\foreach \x/\y in
{4/2,5/3}
{\path[edge,dotted] (\x)--(\y);}

\foreach \x/\y in
{3/2,5/4}
{\path[edge,dashed] (\x)--(\y);}

\end{scope}

\begin{scope}[shift={(4,16)}]

\begin{scope}[shift={(4,4)}]
\foreach \x/\y/\z in
{1.5/1.5/1,1.5/-1.5/2,-1.5/-1.5/3,-1.5/1.5/4}
{\node[verti] (b1\z) at (\x,\y){};}

\foreach \x/\y/\z in
{1/2/0,1.1/-2/34,-1.1/-2/134,-1.1/2/1}
{\node[ver] () at (\x,\y){\tiny{$T_{\z}^2$}};}

\foreach \x/\y in 
{1/4,2/3}
{\draw [line width=2pt, line cap=round, dash pattern=on 0pt off 1.7\pgflinewidth]  (b1\x) -- (b1\y);}

\foreach \x/\y in 
{1/2,4/3}
{\draw[line width=2pt, line cap=rectangle, dash pattern=on 1pt off 1]
(b1\x) -- (b1\y);}

\end{scope}

\begin{scope}[shift={(4,-2)}]
\foreach \x/\y/\z in
{1.5/1.5/1,1.5/-1.5/2,-1.5/-1.5/3,-1.5/1.5/4}
{\node[verti] (b2\z) at (\x,\y){};}

\foreach \x/\y/\z in
{1/2/4,1.1/-2/3,-1.1/-2/13,-1.1/2/14}
{\node[ver] () at (\x,\y){\tiny{$T_{\z}^2$}};}

\foreach \x/\y in 
{1/4,2/3}
{\draw [line width=2pt, line cap=round, dash pattern=on 0pt off 1.7\pgflinewidth]  (b2\x) -- (b2\y);}

\foreach \x/\y in 
{1/2,4/3}
{\draw[line width=2pt, line cap=rectangle, dash pattern=on 1pt off 1]
(b2\x) -- (b2\y);}
\end{scope}

\begin{scope}[shift={(-4,-2)}]
\foreach \x/\y/\z in
{1.5/1.5/1,1.5/-1.5/2,-1.5/-1.5/3,-1.5/1.5/4}
{\node[verti] (b3\z) at (\x,\y){};}

\foreach \x/\y/\z in
{1.1/2/24,1.1/-2/23,-1.1/-2/123,-1.1/2/124}
{\node[ver] () at (\x,\y){\tiny{$T_{\z}^2$}};}

\foreach \x/\y in 
{1/4,2/3}
{\draw [line width=2pt, line cap=round, dash pattern=on 0pt off 1.7\pgflinewidth]  (b3\x) -- (b3\y);}

\foreach \x/\y in 
{1/2,4/3}
{\draw[line width=2pt, line cap=rectangle, dash pattern=on 1pt off 1]
(b3\x) -- (b3\y);}
\end{scope}

\begin{scope}[shift={(-4,4)}]
\foreach \x/\y/\z in
{1.5/1.5/1,1.5/-1.5/2,-1.5/-1.5/3,-1.5/1.5/4}
{\node[verti] (b4\z) at (\x,\y){};}

\foreach \x/\y/\z in
{1/2/2,1.1/-2/234,-1.1/-2/1234,-1.1/2/12}
{\node[ver] () at (\x,\y){\tiny{$T_{\z}^2$}};}

\foreach \x/\y in 
{1/4,2/3}
{\draw [line width=2pt, line cap=round, dash pattern=on 0pt off 1.7\pgflinewidth]  (b4\x) -- (b4\y);}

\foreach \x/\y in 
{1/2,4/3}
{\draw[line width=2pt, line cap=rectangle, dash pattern=on 1pt off 1]
(b4\x) -- (b4\y);}
\end{scope}

\end{scope}

\begin{scope}[shift={(-6,5)}]
\begin{scope}[shift={(-15,0)}]
\foreach \x/\y/\z in
{1.5/1.5/1,1.5/-1.5/2,-1.5/-1.5/3,-1.5/1.5/4,3/3/5,3/-3/6,-3/-3/7,-3/3/8}
{\node[verti] (c1\z) at (\x,\y){};}

\foreach \x/\y/\z in
{1.1/2/13,1.1/-2/14,-1.1/-2/134,-1.1/2/1,2.6/3.5/123,2.6/-3.5/124,-2.6/-3.5/1234,-2.6/3.5/12}
{\node[ver] () at (\x,\y){\tiny{$T_{\z}^3$}};}

\foreach \x/\y in 
{1/4,2/3,6/7,8/5}
{\draw [line width=2pt, line cap=round, dash pattern=on 0pt off 1.7\pgflinewidth]  (c1\x) -- (c1\y);}

\foreach \x/\y in 
{1/2,4/3,6/5,8/7}
{\draw[line width=2pt, line cap=rectangle, dash pattern=on 1pt off 1]
(c1\x) -- (c1\y);}

\foreach \x/\y in 
{5/1,6/2,7/3,8/4}
{\path[edge] (c1\x)--(c1\y);}

\end{scope}

\begin{scope}[shift={(-5,0)}]
\foreach \x/\y/\z in
{1.5/1.5/1,1.5/-1.5/2,-1.5/-1.5/3,-1.5/1.5/4,3/3/5,3/-3/6,-3/-3/7,-3/3/8}
{\node[verti] (c2\z) at (\x,\y){};}

\foreach \x/\y/\z in
{1.1/2/23,1.1/-2/24,-1.1/-2/234,-1.1/2/2,2.6/3.5/3,2.6/-3.5/4,-2.6/-3.5/34,-2.6/3.5/0}
{\node[ver] () at (\x,\y){\tiny{$T_{\z}^3$}};}

\foreach \x/\y in 
{1/4,2/3,6/7,8/5}
{\draw [line width=2pt, line cap=round, dash pattern=on 0pt off 1.7\pgflinewidth]  (c2\x) -- (c2\y);}

\foreach \x/\y in 
{1/2,4/3,6/5,8/7}
{\draw[line width=2pt, line cap=rectangle, dash pattern=on 1pt off 1]
(c2\x) -- (c2\y);}

\foreach \x/\y in 
{5/1,6/2,7/3,8/4}
{\path[edge] (c2\x)--(c2\y);}

\end{scope}

\begin{scope}[shift={(5,0)}]
\foreach \x/\y/\z in
{1.5/1.5/1,1.5/-1.5/2,-1.5/-1.5/3,-1.5/1.5/4,3/3/5,3/-3/6,-3/-3/7,-3/3/8}
{\node[verti] (d1\z) at (\x,\y){};}

\foreach \x/\y/\z in
{1.1/2/0,1.1/-2/12,-1.1/-2/2,-1.1/2/1,2.6/3.5/4,2.6/-3.5/124,-2.6/-3.5/24,-2.6/3.5/14}
{\node[ver] () at (\x,\y){\tiny{$T_{\z}^4$}};}

\foreach \x/\y in 
{1/4,2/3,6/7,8/5}
{\draw [line width=2pt, line cap=round, dash pattern=on 0pt off 1.7\pgflinewidth]  (d1\x) -- (d1\y);}

\foreach \x/\y in 
{1/2,4/3,6/5,8/7}
{\draw[line width=2pt, line cap=rectangle, dash pattern=on 1pt off 1]
(d1\x) -- (d1\y);}

\foreach \x/\y in 
{5/1,6/2,7/3,8/4}
{\path[edge] (d1\x)--(d1\y);}

\end{scope}

\begin{scope}[shift={(15,0)}]
\foreach \x/\y/\z in
{1.5/1.5/1,1.5/-1.5/2,-1.5/-1.5/3,-1.5/1.5/4,3/3/5,3/-3/6,-3/-3/7,-3/3/8}
{\node[verti] (d2\z) at (\x,\y){};}

\foreach \x/\y/\z in
{1.1/2/3,1.1/-2/123,-1.1/-2/23,-1.1/2/13,2.6/3.5/34,2.6/-3.5/1234,-2.6/-3.5/234,-2.6/3.5/134}
{\node[ver] () at (\x,\y){\tiny{$T_{\z}^4$}};}

\foreach \x/\y in 
{1/4,2/3,6/7,8/5}
{\draw [line width=2pt, line cap=round, dash pattern=on 0pt off 1.7\pgflinewidth]  (d2\x) -- (d2\y);}

\foreach \x/\y in 
{1/2,4/3,6/5,8/7}
{\draw[line width=2pt, line cap=rectangle, dash pattern=on 1pt off 1]
(d2\x) -- (d2\y);}

\foreach \x/\y in 
{5/1,6/2,7/3,8/4}
{\path[edge] (d2\x)--(d2\y);}

\end{scope}

\end{scope}

\begin{scope}[shift={(-16,-7)}]

\begin{scope}[shift={(4,4)}]
\foreach \x/\y/\z in
{1.5/1.5/1,1.5/-1.5/2,-1.5/-1.5/3,-1.5/1.5/4}
{\node[verti] (e1\z) at (\x,\y){};}

\foreach \x/\y/\z in
{1/2/13,1.1/-2/23,-1.1/-2/2,-1.1/2/1}
{\node[ver] () at (\x,\y){\tiny{$T_{\z}^5$}};}

\foreach \x/\y in 
{1/4,2/3}
{\draw [line width=2pt, line cap=round, dash pattern=on 0pt off 1.7\pgflinewidth]  (e1\x) -- (e1\y);}

\foreach \x/\y in 
{1/2,4/3}
{\draw[line width=2pt, line cap=rectangle, dash pattern=on 1pt off 1]
(e1\x) -- (e1\y);}

\end{scope}

\begin{scope}[shift={(4,-2)}]
\foreach \x/\y/\z in
{1.5/1.5/1,1.5/-1.5/2,-1.5/-1.5/3,-1.5/1.5/4}
{\node[verti] (e2\z) at (\x,\y){};}

\foreach \x/\y/\z in
{1/2/0,1.1/-2/12,-1.1/-2/123,-1.1/2/3}
{\node[ver] () at (\x,\y){\tiny{$T_{\z}^5$}};}

\foreach \x/\y in 
{1/4,2/3}
{\draw [line width=2pt, line cap=round, dash pattern=on 0pt off 1.7\pgflinewidth]  (e2\x) -- (e2\y);}

\foreach \x/\y in 
{1/2,4/3}
{\draw[line width=2pt, line cap=rectangle, dash pattern=on 1pt off 1]
(e2\x) -- (e2\y);}
\end{scope}

\begin{scope}[shift={(-4,-2)}]
\foreach \x/\y/\z in
{1.5/1.5/1,1.5/-1.5/2,-1.5/-1.5/3,-1.5/1.5/4}
{\node[verti] (e3\z) at (\x,\y){};}

\foreach \x/\y/\z in
{1.1/2/34,1.1/-2/1234,-1.1/-2/124,-1.1/2/4}
{\node[ver] () at (\x,\y){\tiny{$T_{\z}^5$}};}

\foreach \x/\y in 
{1/4,2/3}
{\draw [line width=2pt, line cap=round, dash pattern=on 0pt off 1.7\pgflinewidth]  (e3\x) -- (e3\y);}

\foreach \x/\y in 
{1/2,4/3}
{\draw[line width=2pt, line cap=rectangle, dash pattern=on 1pt off 1]
(e3\x) -- (e3\y);}
\end{scope}

\begin{scope}[shift={(-4,4)}]
\foreach \x/\y/\z in
{1.5/1.5/1,1.5/-1.5/2,-1.5/-1.5/3,-1.5/1.5/4}
{\node[verti] (e4\z) at (\x,\y){};}

\foreach \x/\y/\z in
{1/2/24,1.1/-2/14,-1.1/-2/134,-1.1/2/234}
{\node[ver] () at (\x,\y){\tiny{$T_{\z}^5$}};}

\foreach \x/\y in 
{1/4,2/3}
{\draw [line width=2pt, line cap=round, dash pattern=on 0pt off 1.7\pgflinewidth]  (e4\x) -- (e4\y);}

\foreach \x/\y in 
{1/2,4/3}
{\draw[line width=2pt, line cap=rectangle, dash pattern=on 1pt off 1]
(e4\x) -- (e4\y);}
\end{scope}

\end{scope}

\begin{scope}[shift={(5,-6)}]
\foreach \x/\y/\z in
{1.5/1.5/1,1.5/-1.5/2,-1.5/-1.5/3,-1.5/1.5/4,3/3/5,3/-3/6,-3/-3/7,-3/3/8,4.5/4.5/9,4.5/-4.5/10,-4.5/-4.5/11,-4.5/4.5/12,6/6/13,6/-6/14,-6/-6/15,-6/6/16}
{\node[verti] (f\z) at (\x,\y){};}

\foreach \x/\y/\z in
{1.1/2/13,1.1/-2/23,-1.1/-2/2,-1.1/2/1,2.4/3.5/123,2.6/-3.5/3,-2.6/-3.5/0,-2.6/3.5/12,3.8/5/1234,4.1/-5/34,-4.1/-5/4,-4.1/5/124,5.6/6.5/134,5.6/-6.5/234,-5.6/-6.5/24,-5.6/6.5/14}
{\node[ver] () at (\x,\y){\tiny{$T_{\z}^6$}};}

\foreach \x/\y in 
{5/9,6/10,7/11,8/12}
{\path[edge,dashed] (f\x)--(f\y);}

\foreach \x/\y in 
{1/4,2/3,6/7,8/5,10/11,12/9,14/15,16/13}
{\draw [line width=2pt, line cap=round, dash pattern=on 0pt off 1.7\pgflinewidth]  (f\x) -- (f\y);}

\foreach \x/\y in 
{1/2,4/3,6/5,8/7,10/9,12/11,14/13,16/15}
{\draw[line width=2pt, line cap=rectangle, dash pattern=on 1pt off 1]
(f\x) -- (f\y);}

\foreach \x/\y in 
{5/1,6/2,7/3,8/4,9/13,10/14,11/15,12/16}
{\path[edge,dotted] (f\x)--(f\y);}

\draw[edge,dashed] plot [smooth,tension=0.5] coordinates{(f1) (4.2,2.5) (f13)};

\draw[edge,dashed] plot [smooth,tension=0.5] coordinates{(f2) (4.2,-2.5) (f14)};

\draw[edge,dashed] plot [smooth,tension=0.5] coordinates{(f3) (-4.2,-2.5) (f15)};

\draw[edge,dashed] plot [smooth,tension=0.5] coordinates{(f4) (-4.2,2.5) (f16)};

\end{scope}

%\begin{scope} [shift = {(-15,-14)}]
%\foreach \x/\y/\z in {1/-0.5/0,5/-0.5/1,9/-0.5/2,13/-0.5/3,17/-0.5/4}
%{\node[ver] () at (\x,\y){$\z$};}
%\path[edge,dashed] (12,0) -- (14,0);

%\draw [line width=2pt, line cap=round, dash pattern=on 0pt off 1.7\pgflinewidth]  (16,0) -- (18,0);
%\path[edge] (8,0) -- (10,0);
%\draw[line width=2pt, line cap=rectangle, dash pattern=on 1pt off 1] (0,0) -- (2,0);

%\path[edge,dotted] (4,0) -- (6,0);
%\end{scope} 

\end{tikzpicture}
\caption{The gem $\Gamma_1$ of $\mathbb{RP}^2 \times \mathbb{RP}^2$ with $96$ vertices.}\label{fig:Ex1}

\end{figure}
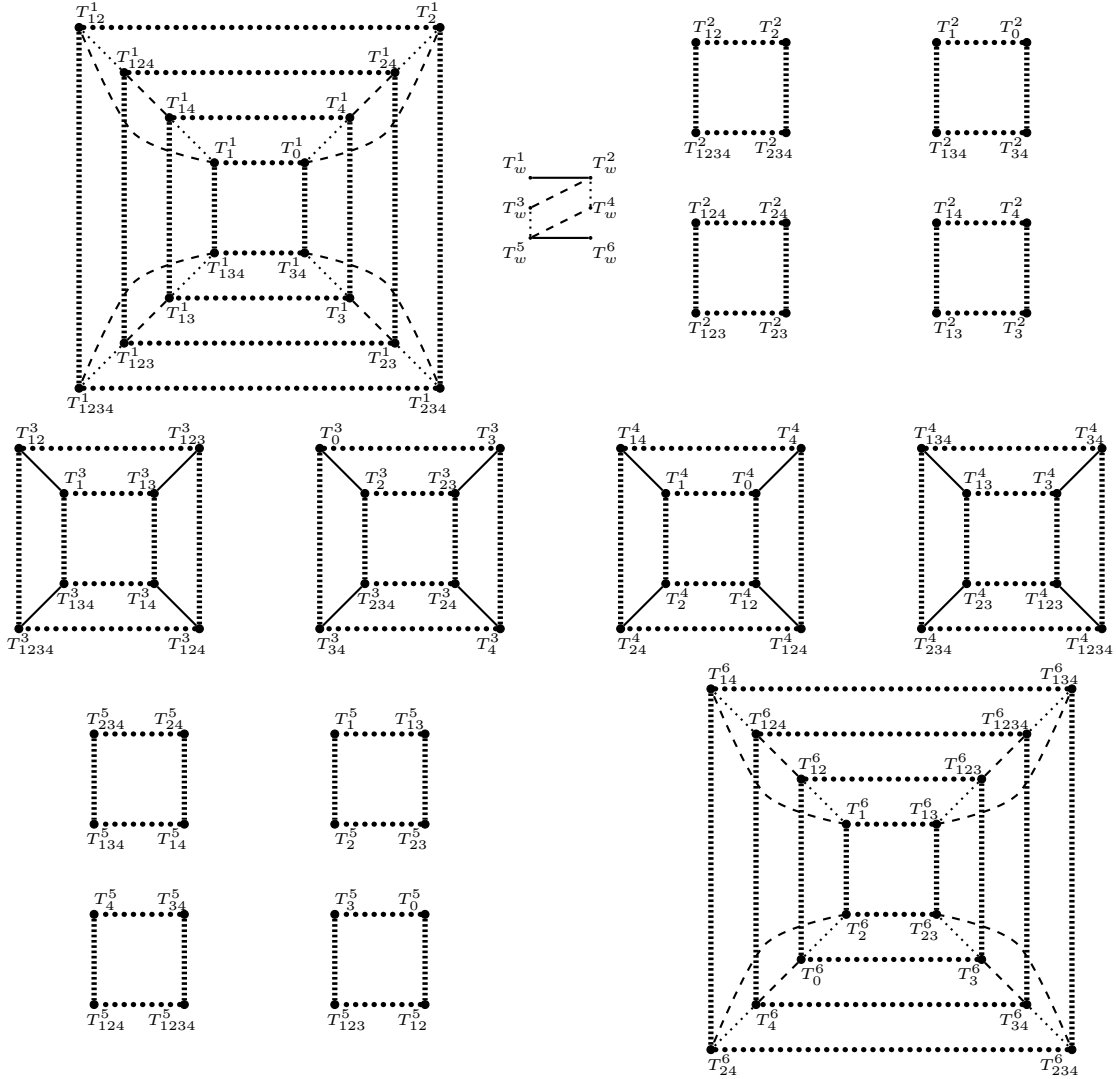

\begin{figure}[h!]
\tikzstyle{ver}=[]
\tikzstyle{vert}=[circle, draw, fill=black!100, inner sep=0pt, minimum width=2pt]
\tikzstyle{verti}=[circle, draw, fill=black!100, inner sep=0pt, minimum width=3pt]
\tikzstyle{edge} = [draw,thick,-]
    \centering
\begin{tikzpicture}[scale=0.45]
\begin{scope}[rotate=23]

\foreach \x/\y in {45/a1,90/a2,135/a3,180/a4,225/a5,270/a6,315/a7,0/a8}
{\node[verti] (\y) at (\x:2){};}

\foreach \x/\y in {45/b1,90/b2,135/b3,180/b4,225/b5,270/b6,315/b7,0/b8}
{\node[verti] (\y) at (\x:3.5){};}

\foreach \x/\y in {45/c1,90/c2,135/c3,180/c4,225/c5,270/c6,315/c7,0/c8}
{\node[verti] (\y) at (\x:5){};}

\foreach \x/\y in {45/d1,90/d2,135/d3,180/d4,225/d5,270/d6,315/d7,0/d8}
{\node[verti] (\y) at (\x:6.5){};}

\foreach \x/\y in {45/e1,90/e2,135/e3,180/e4,225/e5,270/e6,315/e7,0/e8}
{\node[verti] (\y) at (\x:8){};}

\foreach \x/\y in {45/f1,90/f2,135/f3,180/f4,225/f5,270/f6,315/f7,0/f8}
{\node[verti] (\y) at (\x:9.5){};}

\foreach \x/\y in {45/g1,90/g2,135/g3,180/g4,225/g5,270/g6,315/g7,0/g8}
{\node[verti] (\y) at (\x:11){};}

\foreach \x/\y in {45/h1,90/h2,135/h3,180/h4,225/h5,270/h6,315/h7,0/h8}
{\node[verti] (\y) at (\x:12.5){};}

\foreach \x/\y/\z in {45/134/3,90/1/3,135/1/5,180/2/5,225/2/3,270/234/3,315/234/5,0/134/5}
{\node[ver] at (\x:1.35){\tiny{$T_{\y}^{\z}$}};}

\foreach \x/\y/\z in {34/134/2,83/1/2,130/1/4,173/2/4,218/2/2,262/234/2,310/234/4,353/134/4}
{\node[ver] at (\x:3.8){\tiny{$T_{\y}^{\z}$}};}

\foreach \x/\y/\z in {38/34/2,85/0/2,131/0/4,174/12/4,218/12/2,263/1234/2,311/1234/4,355/34/4}
{\node[ver] at (\x:5.3){\tiny{$T_{\y}^{\z}$}};}   

\foreach \x/\y/\z in {40/34/3,86/0/3,131/0/5,175/12/5,219/12/3,264/1234/3,311/1234/5,356/34/5}
{\node[ver] at (\x:6.8){\tiny{$T_{\y}^{\z}$}};}

\foreach \x/\y/\z in {42/4/3,87/3/3,132/3/5,176/123/5,220/123/3,266/124/3,312/124/5,357/4/5}
{\node[ver] at (\x:8.45){\tiny{$T_{\y}^{\z}$}};}

\foreach \x/\y/\z in {42/4/2,87/3/2,132/3/4,176/123/4,221/123/2,267/124/2,312/124/4,357/4/4}
{\node[ver] at (\x:9.9){\tiny{$T_{\y}^{\z}$}};}

\foreach \x/\y/\z in {42/14/2,87/13/2,133/13/4,177/23/4,222/23/2,267/24/2,313/24/4,357/14/4}
{\node[ver] at (\x:11.3){\tiny{$T_{\y}^{\z}$}};}

\foreach \x/\y/\z in {42/14/3,87/13/3,133/13/5,177/23/5,223/23/3,268/24/3,313/24/5,357/14/5}
{\node[ver] at (\x:12.8){\tiny{$T_{\y}^{\z}$}};}

\foreach \x/\y in {a2/a3,a4/a5,a6/a7,a1/a8,b2/b3,b4/b5,b6/b7,b1/b8,c2/c3,c4/c5,c6/c7,c1/c8,d2/d3,d4/d5,d6/d7,d1/d8,e2/e3,e4/e5,e6/e7,e1/e8,f2/f3,f4/f5,f6/f7,f1/f8,g2/g3,g4/g5,g6/g7,g1/g8,h2/h3,h4/h5,h6/h7,h1/h8}
{\path[edge,dotted] (\x) -- (\y);}

\foreach \x/\y in {a1/b1,a2/b2,a3/b3,a4/b4,a5/b5,a6/b6,a7/b7,a8/b8,c1/d1,c2/d2,c3/d3,c4/d4,c5/d5,c6/d6,c7/d7,c8/d8,e1/f1,e2/f2,e3/f3,e4/f4,e5/f5,e6/f6,e7/f7,e8/f8,g1/h1,g2/h2,g3/h3,g4/h4,g5/h5,g6/h6,g7/h7,g8/h8}
{\path[edge,dashed] (\x) -- (\y);}

\foreach \x/\y in 
{a1/a2,a4/a3,a6/a5,a8/a7,b1/b2,b4/b3,b6/b5,b8/b7,c1/c2,c4/c3,c6/c5,c8/c7,d1/d2,d4/d3,d6/d5,d8/d7,e1/e2,e4/e3,e6/e5,e8/e7,f1/f2,f4/f3,f6/f5,f8/f7,g1/g2,g4/g3,g6/g5,g8/g7,h1/h2,h4/h3,h6/h5,h8/h7}
{\draw[line width=2pt, line cap=rectangle, dash pattern=on 1pt off 1]
(\x) -- (\y);}

\foreach \x/\y in {c1/b1,c2/b2,c3/b3,c4/b4,c5/b5,c6/b6,c7/b7,c8/b8,e1/d1,e2/d2,e3/d3,e4/d4,e5/d5,e6/d6,e7/d7,e8/d8,g1/f1,g2/f2,g3/f3,g4/f4,g5/f5,g6/f6,g7/f7,g8/f8}
{\draw [line width=2pt, line cap=round, dash pattern=on 0pt off 1.7\pgflinewidth]  (\x) -- (\y);}

\foreach \x/\y in {57/i1,102/i2,147/i3,192/i4,237/i5,282/i6,327/i7,12/i8}
{\node[ver](\y) at (\x:7){};}

\foreach \x/\y/\z in {a1/h1/i1,a2/h2/i2,a3/h3/i3,a4/h4/i4,a5/h5/i5,a6/h6/i6,a7/h7/i7,a8/h8/i8}
{\draw [line width=2pt, line cap=round, dash pattern=on 0pt off 1.7\pgflinewidth] plot [smooth,tension=0.5] coordinates{(\x) (\z) (\y)};}

\end{scope}  

\end{tikzpicture}

 \caption{Subgraph $S_1$ of $\Gamma_1$.}
    \label{fig:S}
\end{figure}
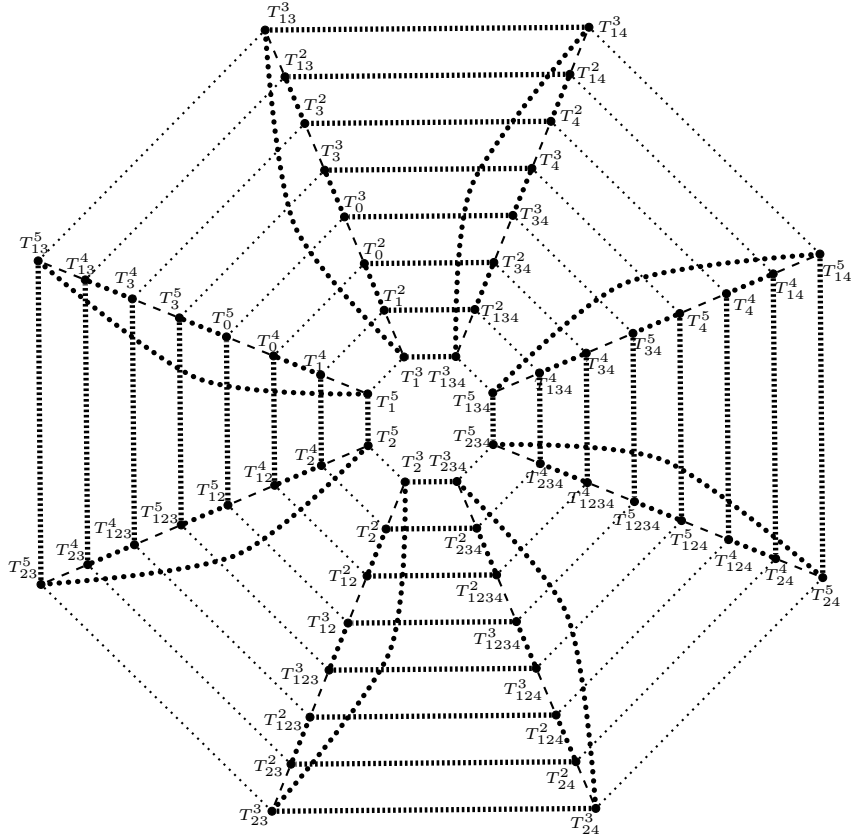

\begin{figure}[h]
\tikzstyle{ver}=[]
\tikzstyle{edge} = [draw,thick,-]
\centering
\begin{tikzpicture}[scale=0.8]

\foreach \x/\y/\z in
{-9/0/13,-8/0/14,-7/0/24,-6/0/23,-4/0/13,-3/0/24,-2/0/14,-1/0/23,1/0/13,2/0/4,3/0/124,4/0/23,6/0/13,7/0/124,8/0/4,9/0/23,-9/1/3,-8/1/4,-7/1/124,-6/1/123,-4/1/3,-3/1/124,-2/1/4,-1/1/123,1/1/3,2/1/14,3/1/24,4/1/123,6/1/3,7/1/24,8/1/14,9/1/123,-9/2/0,-8/2/34,-7/2/1234,-6/2/12,-4/2/0,-3/2/1234,-2/2/34,-1/2/12,1/2/0,2/2/134,3/2/234,4/2/12,6/2/0,7/2/234,8/2/134,9/2/12,-9/3/1,-8/3/134,-7/3/234,-6/3/2,-4/3/1,-3/3/234,-2/3/134,-1/3/2,1/3/1,2/3/34,3/3/1234,4/3/2,6/3/1,7/3/1234,8/3/34,9/3/2}
{\node[ver] at (\x,\y){$\z$};}

\foreach \x/\y/\z in
{-6.5/-5/13,-5.5/-5/14,-4.5/-5/23,-3.5/-5/24,-1.5/-5/13,-0.5/-5/14,0.5/-5/234,1.5/-5/2,3.5/-5/13,4.5/-5/14,5.5/-5/2,6.5/-5/234,-6.5/-4/3,-5.5/-4/4,-4.5/-4/123,-3.5/-4/124,-1.5/-4/3,-0.5/-4/4,0.5/-4/1234,1.5/-4/12,3.5/-4/3,4.5/-4/4,5.5/-4/12,6.5/-4/1234,-6.5/-3/0,-5.5/-3/34,-4.5/-3/12,-3.5/-3/1234,-1.5/-3/0,-0.5/-3/34,0.5/-3/124,1.5/-3/123,3.5/-3/0,4.5/-3/34,5.5/-3/123,6.5/-3/124,-6.5/-2/1,-5.5/-2/134,-4.5/-2/2,-3.5/-2/234,-1.5/-2/1,-0.5/-2/134,0.5/-2/24,1.5/-2/23,3.5/-2/1,4.5/-2/134,5.5/-2/23,6.5/-2/24}
{\node[ver] at (\x,\y){$\z$};}

\draw (-9.4,-0.4) rectangle (-5.5,3.4);
\draw (-4.4,-0.4) rectangle (-0.5,3.4);
\draw (0.6,-0.4) rectangle (4.5,3.4);
\draw (5.6,-0.4) rectangle (9.5,3.4);

\draw (-6.9,-5.4) rectangle (-2.9,-1.6);
\draw (-1.9,-5.4) rectangle (2,-1.6);
\draw (3.1,-5.4) rectangle (7.1,-1.6);

\end{tikzpicture}
\caption{Compact form of $S_i$, for $1\le i\le 7$.}\label{fig:CF}
\end{figure}

Using the construction introduced in \cite{ab25}, we will construct gems of these small covers. Let us fix some notations that will be used throughout this section. For an element $g=(c_1,c_2,c_3,c_4)$ in $\mathbb{Z}_2^4$, let the corresponding word $w_g$ be formed by listing the indices of the non-zero entries in $g$. For example, the word corresponding to the element $(0,1,0,1)$ is $2\,4$. Let $W$ be the set of all the words corresponding to the elements of $\mathbb{Z}_2^4$. The simple polytope $\{g\}\times P$ $\subset \mathbb{Z}_2^4\times P$, a copy of $P$, is denoted by $t_{w_g}$, where $w_g\in W$ is the word corresponding to $g\in \mathbb{Z}_2^4$. For $w\in W$, the simple polytope $t_w$ is color triangulated by the $4$-simplices $$t^1_w=[v_0^0,v_1^0,v_2^0,v_3^1,v_4^2]_w,\ t^2_w=[v_0^0,v_1^0,v_2^1,v_3^1,v_4^2]_w,\ t^3_w=[v_0^0,v_1^0,v_2^1,v_3^2,v_4^2]_w,$$ $$ t^4_w=[v_0^0,v_1^1,v_2^1,v_3^1,v_4^2]_w,\ t^5_w=[v_0^0,v_1^1,v_2^1,v_3^2,v_4^2]_w, \ t^6_w=[v_0^0,v_1^1,v_2^2,v_3^2,v_4^2]_w,$$ where the vertices of each $4$-simplex are colored by the colors $(0,1,2,3,4)$ in order. If a $3$-face of a $4$-simplex does not have the vertex of color $i$, then we call this face the $i$-colored face of the $4$-simplex. Let $\lambda_i: \mathcal F(P)\to \mathbb{Z}_2^4$ be a $\mathbb{Z}_2$-characteristic function, where $1\le i\le 7$ (cf. Lemma \ref{lemma:seven}). Figure \ref{fig:T} shows all the $3$-faces, together with their corresponding $\mathbb{Z}_2$-characteristic vectors, of the $4$-simplex $t_w^j$ of $t_w$, for $1\le j\le 6$ and $w\in W$. With this colored triangulation of each copy of $P$ in $\mathbb{Z}_2^4\times P$, it follows from the construction of small covers that $M^4(\lambda_i)$ admits a colored triangulation, where $1\le i\le 7$ (cf. Subsection \ref{smallcover}). 

Let $(\Gamma_i,\gamma_i)$ denote the $5$-regular colored graph corresponding to this colored triangulation of $M^4(\lambda_i)$, for $1\le i\le 7$. We denote the vertex of $V(\Gamma_i)$ corresponding to the $4$-simplex $t_w^j$ by $T_w^j$, for all $w\in W$ and $1\le j\le 6$ (cf. Subsection \ref{crystal}). Since one copy of $P$ has six $4$-simplices, we have that $\Gamma_i$ has $96$ vertices. Figure \ref{fig:Ex1} represents the $5$-regular colored gem $\Gamma_1$ of $\mathbb{RP}^2\times \mathbb{RP}^2$ obtained via the above construction. From Figure \ref{fig:T}, it follows that $g(\Gamma_{i_{\hat{0}}})=g(\Gamma_{i_{\hat{4}}})=1$, $g(\Gamma_{i_{\hat{1}}})=g(\Gamma_{i_{\hat{3}}})=2$ and $g(\Gamma_{i_{\hat{2}}})=3$. Let $S_i$ be the $4$-regular colored subgraph of $\Gamma_i$ generated by the vertices $T_w^j$, where $w\in W,\ 2\le j\le 5$, and edges of colors $0,1,3,4$, for $1\le i\le 7$. Figure \ref{fig:S} represents the $4$-regular colored subgraph $S_1$ of $\Gamma_1$. The cardinality of the vertex set $V(S_i)$ is $64$. Since $\{e_1,e_3,(c_1,c_2,1,1),(1,1,c_3,c_4)\}$ is a basis of $\mathbb{Z}_2^4$ (consider the vertex $v_2^1$) and $\{0,3\},\{1,3\}, \{0,4\}, \{1,4\}$-colored cycles are four-cycles (see Figure \ref{fig:T}), we have that $S_i$ is isomorphic to $S_1$, for $2\le i\le 7$. From Figure \ref{fig:T}, we have that the notations of the vertices of a $\{0,1\}$-colored cycle (resp. $\{3,4\}$-colored cycle) contains exactly four distinct subscripts. So, we can write the subgraph $S_i$ in the compact form as in Figure \ref{fig:CF}, for $1\le i\le 7$. In the compact form of $S_i$, rows represent the subscripts involved in a $\{0,1\}$-colored cycle and columns represent the subscripts involved in a $\{3,4\}$-colored cycle. The following Theorem presents the construction of a $52$-vertex crystallization $\Gamma_i^\prime$, obtained from the gem $\Gamma_i$, with regular genus $8$, for $1\le i\le 7$. In the proof of Theorem \ref{thm:RP2}, any statement whose justification is not explicitly provided can be understood to follow from Figure \ref{fig:T}.

 \begin{figure}[h!]
\tikzstyle{ver}=[]
\tikzstyle{verti}=[circle, draw, fill=black!100, inner sep=0pt, minimum width=3pt]
\tikzstyle{vert}=[circle, draw, fill=black!100, inner sep=0pt, minimum width=1pt]
\tikzstyle{edge} = [draw,thick,-]
    \centering

\begin{tikzpicture}[scale=0.4]

\begin{scope}[shift={(-16,16)}]
\foreach \x/\y/\z in
{1.5/1.5/1,1.5/-1.5/2,-1.5/-1.5/3,-1.5/1.5/4,3/3/5,3/-3/6,-3/-3/7,-3/3/8,4.5/4.5/9,4.5/-4.5/10,-4.5/-4.5/11,-4.5/4.5/12,6/6/13,6/-6/14,-6/-6/15,-6/6/16,7.5/6/17,7.5/-6/18,-7.5/-6/19,-7.5/6/20}
{\node[verti] (a\z) at (\x,\y){};}

\foreach \x/\y/\z in
{2.6/3.5/4,2.6/-3.5/3,-2.6/-3.5/13,-2.6/3.5/14,4.1/5/24,4.1/-5/23,-4.1/-5/123,-4.1/5/124,5.6/6.5/2,5.6/-6.5/234,-5.6/-6.5/1234,-5.6/6.5/12}
{\node[ver] () at (\x,\y){\tiny{$T_{\z}^1$}};}

\foreach \x/\y/\z in
{1.1/2/0,1.1/-2/34,-1.1/-2/134,-1.1/2/1}
{\node[ver] () at (\x,\y){\tiny{$T_{\z}^4$}};}

\foreach \x/\y/\z in
{8.15/6/0,8.1/-6/34,-8.3/-6/134,-8.1/6/1}
{\node[ver] () at (\x,\y){\tiny{$T_{\z}^3$}};}

\foreach \x/\y in 
{5/9,6/10,7/11,8/12,13/17,14/18,15/19,16/20}
{\path[edge,dashed] (a\x)--(a\y);}

\foreach \x/\y in 
{1/4,2/3,6/7,8/5,10/11,12/9,14/15,16/13}
{\draw [line width=2pt, line cap=round, dash pattern=on 0pt off 1.7\pgflinewidth]  (a\x) -- (a\y);}

\foreach \x/\y in 
{6/5,8/7,10/9,12/11,14/13,16/15,17/18,19/20}
{\draw[line width=2pt, line cap=rectangle, dash pattern=on 1pt off 1]
(a\x) -- (a\y);}

\foreach \x/\y in 
{5/1,6/2,7/3,8/4,9/13,10/14,11/15,12/16}
{\path[edge,dotted] (a\x)--(a\y);}

 \end{scope}

\begin{scope}[shift={(-4.5,16)}]
\foreach \x/\y/\z in
{-1/1/1,1/1/2,-1/0/3,1/0/4,-1/-1/5,1/-1/6}
{\node[vert] (\z) at (\x,\y){};}

\foreach \x/\y/\z in
{-1.5/1.5/1,1.5/1.5/2,-1.5/0/3,1.5/0/4,-1.5/-1.5/5,1.5/-1.5/6}
{\node[ver] () at (\x,\y){\tiny{$T_w^{\z}$}};}

\foreach \x/\y in
{1/2,5/6}
{\path[edge] (\x)--(\y);}

\foreach \x/\y in
{4/2,5/3}
{\path[edge,dotted] (\x)--(\y);}

\foreach \x/\y in
{3/2,5/4}
{\path[edge,dashed] (\x)--(\y);}

\end{scope}

\begin{scope}[shift={(4,16)}]

\begin{scope}[shift={(4,-2)}]
\foreach \x/\y/\z in
{1.5/1.5/1,1.5/-1.5/2,-1.5/-1.5/3,-1.5/1.5/4}
{\node[verti] (b2\z) at (\x,\y){};}

\foreach \x/\y/\z in
{1/2/4,1.1/-2/3,-1.1/-2/13,-1.1/2/14}
{\node[ver] () at (\x,\y){\tiny{$T_{\z}^2$}};}

\foreach \x/\y in 
{1/4,2/3}
{\draw [line width=2pt, line cap=round, dash pattern=on 0pt off 1.7\pgflinewidth]  (b2\x) -- (b2\y);}

\foreach \x/\y in 
{1/2,4/3}
{\draw[line width=2pt, line cap=rectangle, dash pattern=on 1pt off 1]
(b2\x) -- (b2\y);}
\end{scope}

\begin{scope}[shift={(-4,-2)}]
\foreach \x/\y/\z in
{1.5/1.5/1,1.5/-1.5/2,-1.5/-1.5/3,-1.5/1.5/4}
{\node[verti] (b3\z) at (\x,\y){};}

\foreach \x/\y/\z in
{1.1/2/24,1.1/-2/23,-1.1/-2/123,-1.1/2/124}
{\node[ver] () at (\x,\y){\tiny{$T_{\z}^2$}};}

\foreach \x/\y in 
{1/4,2/3}
{\draw [line width=2pt, line cap=round, dash pattern=on 0pt off 1.7\pgflinewidth]  (b3\x) -- (b3\y);}

\foreach \x/\y in 
{1/2,4/3}
{\draw[line width=2pt, line cap=rectangle, dash pattern=on 1pt off 1]
(b3\x) -- (b3\y);}
\end{scope}

\begin{scope}[shift={(-4,4)}]
\foreach \x/\y/\z in
{1.5/1.5/1,1.5/-1.5/2,-1.5/-1.5/3,-1.5/1.5/4}
{\node[verti] (b4\z) at (\x,\y){};}

\foreach \x/\y/\z in
{1/2/2,1.1/-2/234,-1.1/-2/1234,-1.1/2/12}
{\node[ver] () at (\x,\y){\tiny{$T_{\z}^2$}};}

\foreach \x/\y in 
{1/4,2/3}
{\draw [line width=2pt, line cap=round, dash pattern=on 0pt off 1.7\pgflinewidth]  (b4\x) -- (b4\y);}

\foreach \x/\y in 
{1/2,4/3}
{\draw[line width=2pt, line cap=rectangle, dash pattern=on 1pt off 1]
(b4\x) -- (b4\y);}
\end{scope}

\end{scope}

\begin{scope}[shift={(-6,5)}]
\begin{scope}[shift={(-15,0)}]
\foreach \x/\y/\z in
{1.5/1.5/1,1.5/-1.5/2,-1.5/-1.5/3,-1.5/1.5/4,3/3/5,3/-3/6,-3/-3/7,-3/3/8}
{\node[verti] (c1\z) at (\x,\y){};}

\foreach \x/\y/\z in
{1.1/2/13,1.1/-2/14,-1.1/-2/134,-1.1/2/1,2.6/3.5/123,2.6/-3.5/124,-2.6/-3.5/1234,-2.6/3.5/12}
{\node[ver] () at (\x,\y){\tiny{$T_{\z}^3$}};}

\foreach \x/\y in 
{1/4,2/3,6/7,8/5}
{\draw [line width=2pt, line cap=round, dash pattern=on 0pt off 1.7\pgflinewidth]  (c1\x) -- (c1\y);}

\foreach \x/\y in 
{1/2,4/3,6/5,8/7}
{\draw[line width=2pt, line cap=rectangle, dash pattern=on 1pt off 1]
(c1\x) -- (c1\y);}

\foreach \x/\y in 
{5/1,6/2,7/3,8/4}
{\path[edge] (c1\x)--(c1\y);}

\end{scope}

\begin{scope}[shift={(-5,0)}]
\foreach \x/\y/\z in
{1.5/1.5/1,1.5/-1.5/2,-1.5/-1.5/3,-1.5/1.5/4,3/3/5,3/-3/6,-3/-3/7,-3/3/8}
{\node[verti] (c2\z) at (\x,\y){};}

\foreach \x/\y/\z in
{1.1/2/23,1.1/-2/24,-1.1/-2/234,-1.1/2/2,2.6/3.5/3,2.6/-3.5/4,-2.6/-3.5/34,-2.6/3.5/0}
{\node[ver] () at (\x,\y){\tiny{$T_{\z}^3$}};}

\foreach \x/\y in 
{1/4,2/3,6/7,8/5}
{\draw [line width=2pt, line cap=round, dash pattern=on 0pt off 1.7\pgflinewidth]  (c2\x) -- (c2\y);}

\foreach \x/\y in 
{1/2,4/3,6/5,8/7}
{\draw[line width=2pt, line cap=rectangle, dash pattern=on 1pt off 1]
(c2\x) -- (c2\y);}

\foreach \x/\y in 
{5/1,6/2,7/3,8/4}
{\path[edge] (c2\x)--(c2\y);}

\end{scope}

\begin{scope}[shift={(5,0)}]
\foreach \x/\y/\z in
{1.5/1.5/1,1.5/-1.5/2,-1.5/-1.5/3,-1.5/1.5/4,3/3/5,3/-3/6,-3/-3/7,-3/3/8}
{\node[verti] (d1\z) at (\x,\y){};}

\foreach \x/\y/\z in
{1.1/2/0,1.1/-2/12,-1.1/-2/2,-1.1/2/1,2.6/3.5/4,2.6/-3.5/124,-2.6/-3.5/24,-2.6/3.5/14}
{\node[ver] () at (\x,\y){\tiny{$T_{\z}^4$}};}

\foreach \x/\y in 
{1/4,2/3,6/7,8/5}
{\draw [line width=2pt, line cap=round, dash pattern=on 0pt off 1.7\pgflinewidth]  (d1\x) -- (d1\y);}

\foreach \x/\y in 
{1/2,4/3,6/5,8/7}
{\draw[line width=2pt, line cap=rectangle, dash pattern=on 1pt off 1]
(d1\x) -- (d1\y);}

\foreach \x/\y in 
{5/1,6/2,7/3,8/4}
{\path[edge] (d1\x)--(d1\y);}

\end{scope}

\begin{scope}[shift={(15,0)}]
\foreach \x/\y/\z in
{1.5/1.5/1,1.5/-1.5/2,-1.5/-1.5/3,-1.5/1.5/4,3/3/5,3/-3/6,-3/-3/7,-3/3/8}
{\node[verti] (d2\z) at (\x,\y){};}

\foreach \x/\y/\z in
{1.1/2/3,1.1/-2/123,-1.1/-2/23,-1.1/2/13,2.6/3.5/34,2.6/-3.5/1234,-2.6/-3.5/234,-2.6/3.5/134}
{\node[ver] () at (\x,\y){\tiny{$T_{\z}^4$}};}

\foreach \x/\y in 
{1/4,2/3,6/7,8/5}
{\draw [line width=2pt, line cap=round, dash pattern=on 0pt off 1.7\pgflinewidth]  (d2\x) -- (d2\y);}

\foreach \x/\y in 
{1/2,4/3,6/5,8/7}
{\draw[line width=2pt, line cap=rectangle, dash pattern=on 1pt off 1]
(d2\x) -- (d2\y);}

\foreach \x/\y in 
{5/1,6/2,7/3,8/4}
{\path[edge] (d2\x)--(d2\y);}

\end{scope}

\end{scope}

\begin{scope}[shift={(-16,-7)}]

\begin{scope}[shift={(4,4)}]
\foreach \x/\y/\z in
{1.5/1.5/1,1.5/-1.5/2,-1.5/-1.5/3,-1.5/1.5/4}
{\node[verti] (e1\z) at (\x,\y){};}

\foreach \x/\y/\z in
{1/2/13,1.1/-2/23,-1.1/-2/2,-1.1/2/1}
{\node[ver] () at (\x,\y){\tiny{$T_{\z}^5$}};}

\foreach \x/\y in 
{1/4,2/3}
{\draw [line width=2pt, line cap=round, dash pattern=on 0pt off 1.7\pgflinewidth]  (e1\x) -- (e1\y);}

\foreach \x/\y in 
{1/2,4/3}
{\draw[line width=2pt, line cap=rectangle, dash pattern=on 1pt off 1]
(e1\x) -- (e1\y);}

\end{scope}

\begin{scope}[shift={(-4,-2)}]
\foreach \x/\y/\z in
{1.5/1.5/1,1.5/-1.5/2,-1.5/-1.5/3,-1.5/1.5/4}
{\node[verti] (e3\z) at (\x,\y){};}

\foreach \x/\y/\z in
{1.1/2/34,1.1/-2/1234,-1.1/-2/124,-1.1/2/4}
{\node[ver] () at (\x,\y){\tiny{$T_{\z}^5$}};}

\foreach \x/\y in 
{1/4,2/3}
{\draw [line width=2pt, line cap=round, dash pattern=on 0pt off 1.7\pgflinewidth]  (e3\x) -- (e3\y);}

\foreach \x/\y in 
{1/2,4/3}
{\draw[line width=2pt, line cap=rectangle, dash pattern=on 1pt off 1]
(e3\x) -- (e3\y);}
\end{scope}

\begin{scope}[shift={(-4,4)}]
\foreach \x/\y/\z in
{1.5/1.5/1,1.5/-1.5/2,-1.5/-1.5/3,-1.5/1.5/4}
{\node[verti] (e4\z) at (\x,\y){};}

\foreach \x/\y/\z in
{1/2/24,1.1/-2/14,-1.1/-2/134,-1.1/2/234}
{\node[ver] () at (\x,\y){\tiny{$T_{\z}^5$}};}

\foreach \x/\y in 
{1/4,2/3}
{\draw [line width=2pt, line cap=round, dash pattern=on 0pt off 1.7\pgflinewidth]  (e4\x) -- (e4\y);}

\foreach \x/\y in 
{1/2,4/3}
{\draw[line width=2pt, line cap=rectangle, dash pattern=on 1pt off 1]
(e4\x) -- (e4\y);}
\end{scope}

\end{scope}

\begin{scope}[shift={(5,-7.5)}]
\foreach \x/\y/\z in
{1.5/1.5/1,1.5/-1.5/2,-1.5/-1.5/3,-1.5/1.5/4,3/3/5,3/-3/6,-3/-3/7,-3/3/8,6/6/9,6/-6/10,-6/-6/11,-6/6/12,7.5/7.5/13,7.5/-7.5/14,-7.5/-7.5/15,-7.5/7.5/16,4.5/4.5/17,4.5/-4.5/18,-4.5/-4.5/19,-4.5/4.5/20}
{\node[verti] (f\z) at (\x,\y){};}

\foreach \x/\y/\z in
{1.1/2/13,1.1/-2/23,-1.1/-2/2,-1.1/2/1,5.6/6.5/1234,5.6/-6.5/34,-5.6/-6.5/4,-5.6/6.5/124,7.1/8/134,7.1/-8/234,-7.1/-8/24,-7.1/8/14}
{\node[ver] () at (\x,\y){\tiny{$T_{\z}^6$}};}

\foreach \x/\y/\z in
{2.6/3.5/123,2.6/-3.5/3,-2.6/-3.5/0,-2.6/3.5/12}
{\node[ver] () at (\x,\y){\tiny{$T_{\z}^3$}};}

\foreach \x/\y/\z in
{4.1/5/123,4.1/-5/3,-4.1/-5/0,-4.1/5/12}
{\node[ver] () at (\x,\y){\tiny{$T_{\z}^4$}};}

\foreach \x/\y in 
{17/9,18/10,19/11,20/12}
{\path[edge,dashed] (f\x)--(f\y);}

\foreach \x/\y in 
{1/4,2/3,6/7,8/5,10/11,12/9,14/15,16/13}
{\draw [line width=2pt, line cap=round, dash pattern=on 0pt off 1.7\pgflinewidth]  (f\x) -- (f\y);}

\foreach \x/\y in 
{1/2,4/3,17/18,19/20,10/9,12/11,14/13,16/15}
{\draw[line width=2pt, line cap=rectangle, dash pattern=on 1pt off 1]
(f\x) -- (f\y);}

\foreach \x/\y in 
{5/1,6/2,7/3,8/4,9/13,10/14,11/15,12/16}
{\path[edge,dotted] (f\x)--(f\y);}

\draw[edge,dashed] plot [smooth,tension=0.5] coordinates{(f1) (4.2,2.5) (f13)};

\draw[edge,dashed] plot [smooth,tension=0.5] coordinates{(f2) (4.2,-2.5) (f14)};

\draw[edge,dashed] plot [smooth,tension=0.5] coordinates{(f3) (-4.2,-2.5) (f15)};

\draw[edge,dashed] plot [smooth,tension=0.5] coordinates{(f4) (-4.2,2.5) (f16)};

\end{scope}

\end{tikzpicture}
\caption{The gem $\Gamma^1_1$ of $\mathbb{RP}^2 \times \mathbb{RP}^2$ with $80$ vertices.}\label{fig:Ex2}

\end{figure}

\begin{figure}[h!]
\tikzstyle{ver}=[]
\tikzstyle{verti}=[circle, draw, fill=black!100, inner sep=0pt, minimum width=3pt]
\tikzstyle{vert}=[circle, draw, fill=black!100, inner sep=0pt, minimum width=1pt]
\tikzstyle{edge} = [draw,thick,-]
    \centering

\begin{tikzpicture}[scale=0.4]

\begin{scope}[shift={(-16,16)}]
\foreach \x/\y/\z in
{1.5/1.5/1,1.5/-1.5/2,-1.5/-1.5/3,-1.5/1.5/4,3/3/5,3/-3/6,-3/-3/7,-3/3/8,4.5/4.5/9,4.5/-4.5/10,-4.5/-4.5/11,-4.5/4.5/12,6/6/13,6/-6/14,-6/-6/15,-6/6/16,7.5/6/17,7.5/-6/18,-7.5/-6/19,-7.5/6/20}
{\node[verti] (a\z) at (\x,\y){};}

\foreach \x/\y/\z in
{2.6/3.5/4,2.6/-3.5/3,-2.6/-3.5/13,-2.6/3.5/14,4.1/5/24,4.1/-5/23,-4.1/-5/123,-4.1/5/124,5.6/6.5/2,5.6/-6.5/234,-5.6/-6.5/1234,-5.6/6.5/12}
{\node[ver] () at (\x,\y){\tiny{$T_{\z}^1$}};}

\foreach \x/\y/\z in
{1.1/2/0,1.1/-2/34,-1.1/-2/134,-1.1/2/1}
{\node[ver] () at (\x,\y){\tiny{$T_{\z}^4$}};}

\foreach \x/\y/\z in
{8.15/6/0,8.1/-6/34,-8.3/-6/134,-8.1/6/1}
{\node[ver] () at (\x,\y){\tiny{$T_{\z}^3$}};}

\foreach \x/\y in 
{5/9,6/10,7/11,8/12,13/17,14/18,15/19,16/20}
{\path[edge,dashed] (a\x)--(a\y);}

\foreach \x/\y in 
{1/4,2/3,6/7,8/5,10/11,12/9,14/15,16/13}
{\draw [line width=2pt, line cap=round, dash pattern=on 0pt off 1.7\pgflinewidth]  (a\x) -- (a\y);}

\foreach \x/\y in 
{6/5,8/7,10/9,12/11,14/13,16/15,17/18,19/20}
{\draw[line width=2pt, line cap=rectangle, dash pattern=on 1pt off 1]
(a\x) -- (a\y);}

\foreach \x/\y in 
{5/1,6/2,7/3,8/4,9/13,10/14,11/15,12/16}
{\path[edge,dotted] (a\x)--(a\y);}

 \end{scope}

\begin{scope}[shift={(-4.5,16)}]
\foreach \x/\y/\z in
{-1/1/1,1/1/2,-1/0/3,1/0/4,-1/-1/5,1/-1/6}
{\node[vert] (\z) at (\x,\y){};}

\foreach \x/\y/\z in
{-1.5/1.5/1,1.5/1.5/2,-1.5/0/3,1.5/0/4,-1.5/-1.5/5,1.5/-1.5/6}
{\node[ver] () at (\x,\y){\tiny{$T_w^{\z}$}};}

\foreach \x/\y in
{1/2,5/6}
{\path[edge] (\x)--(\y);}

\foreach \x/\y in
{4/2,5/3}
{\path[edge,dotted] (\x)--(\y);}

\foreach \x/\y in
{3/2,5/4}
{\path[edge,dashed] (\x)--(\y);}

\end{scope}

\begin{scope}[shift={(4,16)}]

\begin{scope}[shift={(4,-2)}]
\foreach \x/\y/\z in
{1.5/1.5/1,1.5/-1.5/2,-1.5/-1.5/3,-1.5/1.5/4}
{\node[verti] (b2\z) at (\x,\y){};}

\foreach \x/\y/\z in
{1/2/4,1.1/-2/3}
{\node[ver] () at (\x,\y){\tiny{$T_{\z}^2$}};}

\foreach \x/\y/\z in
{-1.1/-2/1,-1.1/2/134}
{\node[ver] () at (\x,\y){\tiny{$T_{\z}^3$}};}

\foreach \x/\y in 
{1/4,2/3}
{\draw [line width=2pt, line cap=round, dash pattern=on 0pt off 1.7\pgflinewidth]  (b2\x) -- (b2\y);}

\foreach \x/\y in 
{1/2,4/3}
{\draw[line width=2pt, line cap=rectangle, dash pattern=on 1pt off 1]
(b2\x) -- (b2\y);}
\end{scope}

\begin{scope}[shift={(-4,-2)}]
\foreach \x/\y/\z in
{1.5/1.5/1,1.5/-1.5/2,-1.5/-1.5/3,-1.5/1.5/4}
{\node[verti] (b3\z) at (\x,\y){};}

\foreach \x/\y/\z in
{-1.1/-2/123,-1.1/2/124}
{\node[ver] () at (\x,\y){\tiny{$T_{\z}^2$}};}

\foreach \x/\y/\z in
{1.1/2/234,1.1/-2/2}
{\node[ver] () at (\x,\y){\tiny{$T_{\z}^3$}};}

\foreach \x/\y in 
{1/4,2/3}
{\draw [line width=2pt, line cap=round, dash pattern=on 0pt off 1.7\pgflinewidth]  (b3\x) -- (b3\y);}

\foreach \x/\y in 
{1/2,4/3}
{\draw[line width=2pt, line cap=rectangle, dash pattern=on 1pt off 1]
(b3\x) -- (b3\y);}
\end{scope}

\begin{scope}[shift={(-4,4)}]
\foreach \x/\y/\z in
{1.5/1.5/1,1.5/-1.5/2,-1.5/-1.5/3,-1.5/1.5/4}
{\node[verti] (b4\z) at (\x,\y){};}

\foreach \x/\y/\z in
{1/2/2,1.1/-2/234,-1.1/-2/1234,-1.1/2/12}
{\node[ver] () at (\x,\y){\tiny{$T_{\z}^2$}};}

\foreach \x/\y in 
{1/4,2/3}
{\draw [line width=2pt, line cap=round, dash pattern=on 0pt off 1.7\pgflinewidth]  (b4\x) -- (b4\y);}

\foreach \x/\y in 
{1/2,4/3}
{\draw[line width=2pt, line cap=rectangle, dash pattern=on 1pt off 1]
(b4\x) -- (b4\y);}
\end{scope}

\end{scope}

\begin{scope}[shift={(-6,5)}]
\begin{scope}[shift={(-15,0)}]
\foreach \x/\y/\z in
{1.5/1.5/1,1.5/-1.5/2,-1.5/-1.5/3,-1.5/1.5/4,3/3/5,3/-3/6,-3/-3/7,-3/3/8}
{\node[verti] (c1\z) at (\x,\y){};}

\foreach \x/\y/\z in
{-1.1/-2/134,-1.1/2/1,2.6/3.5/123,2.6/-3.5/124,-2.6/-3.5/1234,-2.6/3.5/12}
{\node[ver] () at (\x,\y){\tiny{$T_{\z}^3$}};}

\foreach \x/\y/\z in
{1.1/2/13,1.1/-2/14}
{\node[ver] () at (\x,\y){\tiny{$T_{\z}^1$}};}

\foreach \x/\y in 
{6/7,8/5}
{\draw [line width=2pt, line cap=round, dash pattern=on 0pt off 1.7\pgflinewidth]  (c1\x) -- (c1\y);}

\foreach \x/\y in 
{1/2,4/3,6/5,8/7}
{\draw[line width=2pt, line cap=rectangle, dash pattern=on 1pt off 1]
(c1\x) -- (c1\y);}

\foreach \x/\y in 
{5/1,6/2,7/3,8/4}
{\path[edge] (c1\x)--(c1\y);}

\end{scope}

\begin{scope}[shift={(-5,0)}]
\foreach \x/\y/\z in
{1.5/1.5/1,1.5/-1.5/2,-1.5/-1.5/3,-1.5/1.5/4,3/3/5,3/-3/6,-3/-3/7,-3/3/8}
{\node[verti] (c2\z) at (\x,\y){};}

\foreach \x/\y/\z in
{-1.1/-2/234,-1.1/2/2,2.6/3.5/3,2.6/-3.5/4,-2.6/-3.5/34,-2.6/3.5/0}
{\node[ver] () at (\x,\y){\tiny{$T_{\z}^3$}};}

\foreach \x/\y/\z in
{1.1/2/23,1.1/-2/24}
{\node[ver] () at (\x,\y){\tiny{$T_{\z}^1$}};}

\foreach \x/\y in 
{6/7,8/5}
{\draw [line width=2pt, line cap=round, dash pattern=on 0pt off 1.7\pgflinewidth]  (c2\x) -- (c2\y);}

\foreach \x/\y in 
{1/2,4/3,6/5,8/7}
{\draw[line width=2pt, line cap=rectangle, dash pattern=on 1pt off 1]
(c2\x) -- (c2\y);}

\foreach \x/\y in 
{5/1,6/2,7/3,8/4}
{\path[edge] (c2\x)--(c2\y);}

\end{scope}

\begin{scope}[shift={(5,0)}]
\foreach \x/\y/\z in
{1.5/1.5/1,1.5/-1.5/2,-1.5/-1.5/3,-1.5/1.5/4,3/3/5,3/-3/6,-3/-3/7,-3/3/8}
{\node[verti] (d1\z) at (\x,\y){};}

\foreach \x/\y/\z in
{1.1/2/0,1.1/-2/12,-1.1/-2/2,-1.1/2/1,2.6/3.5/4,2.6/-3.5/124}
{\node[ver] () at (\x,\y){\tiny{$T_{\z}^4$}};}

\foreach \x/\y/\z in
{-2.6/-3.5/24,-2.6/3.5/14}
{\node[ver] () at (\x,\y){\tiny{$T_{\z}^6$}};}

\foreach \x/\y in 
{1/4,2/3}
{\draw [line width=2pt, line cap=round, dash pattern=on 0pt off 1.7\pgflinewidth]  (d1\x) -- (d1\y);}

\foreach \x/\y in 
{1/2,4/3,6/5,8/7}
{\draw[line width=2pt, line cap=rectangle, dash pattern=on 1pt off 1]
(d1\x) -- (d1\y);}

\foreach \x/\y in 
{5/1,6/2,7/3,8/4}
{\path[edge] (d1\x)--(d1\y);}

\end{scope}

\begin{scope}[shift={(15,0)}]
\foreach \x/\y/\z in
{1.5/1.5/1,1.5/-1.5/2,-1.5/-1.5/3,-1.5/1.5/4,3/3/5,3/-3/6,-3/-3/7,-3/3/8}
{\node[verti] (d2\z) at (\x,\y){};}

\foreach \x/\y/\z in
{1/1.9/3,1.1/-2/123,2.6/3.5/34,2.6/-3.5/1234,-2.6/-3.5/234,-2.6/3.5/134}
{\node[ver] () at (\x,\y){\tiny{$T_{\z}^4$}};}

\foreach \x/\y/\z in
{-1.1/-2/23,-1.1/2/13}
{\node[ver] () at (\x,\y){\tiny{$T_{\z}^6$}};}

\foreach \x/\y in 
{6/7,8/5}
{\draw [line width=2pt, line cap=round, dash pattern=on 0pt off 1.7\pgflinewidth]  (d2\x) -- (d2\y);}

\foreach \x/\y in 
{1/2,4/3,6/5,8/7}
{\draw[line width=2pt, line cap=rectangle, dash pattern=on 1pt off 1]
(d2\x) -- (d2\y);}

\foreach \x/\y in 
{5/1,6/2,7/3,8/4}
{\path[edge] (d2\x)--(d2\y);}

\end{scope}

\end{scope}

\begin{scope}[shift={(-16,-7)}]

\begin{scope}[shift={(4,4)}]
\foreach \x/\y/\z in
{1.5/1.5/1,1.5/-1.5/2,-1.5/-1.5/3,-1.5/1.5/4}
{\node[verti] (e1\z) at (\x,\y){};}

\foreach \x/\y/\z in
{-1.1/-2/2,-1.1/2/1}
{\node[ver] () at (\x,\y){\tiny{$T_{\z}^5$}};}

\foreach \x/\y/\z in
{1/2/3,1.1/-2/123}
{\node[ver] () at (\x,\y){\tiny{$T_{\z}^4$}};}

\foreach \x/\y in 
{1/4,2/3}
{\draw [line width=2pt, line cap=round, dash pattern=on 0pt off 1.7\pgflinewidth]  (e1\x) -- (e1\y);}

\foreach \x/\y in 
{1/2,4/3}
{\draw[line width=2pt, line cap=rectangle, dash pattern=on 1pt off 1]
(e1\x) -- (e1\y);}

\end{scope}

\begin{scope}[shift={(-4,-2)}]
\foreach \x/\y/\z in
{1.5/1.5/1,1.5/-1.5/2,-1.5/-1.5/3,-1.5/1.5/4}
{\node[verti] (e3\z) at (\x,\y){};}

\foreach \x/\y/\z in
{1.1/2/34,1.1/-2/1234,-1.1/-2/124,-1.1/2/4}
{\node[ver] () at (\x,\y){\tiny{$T_{\z}^5$}};}

\foreach \x/\y in 
{1/4,2/3}
{\draw [line width=2pt, line cap=round, dash pattern=on 0pt off 1.7\pgflinewidth]  (e3\x) -- (e3\y);}

\foreach \x/\y in 
{1/2,4/3}
{\draw[line width=2pt, line cap=rectangle, dash pattern=on 1pt off 1]
(e3\x) -- (e3\y);}
\end{scope}

\begin{scope}[shift={(-4,4)}]
\foreach \x/\y/\z in
{1.5/1.5/1,1.5/-1.5/2,-1.5/-1.5/3,-1.5/1.5/4}
{\node[verti] (e4\z) at (\x,\y){};}

\foreach \x/\y/\z in
{-1.1/-2/134,-1.1/2/234}
{\node[ver] () at (\x,\y){\tiny{$T_{\z}^5$}};}

\foreach \x/\y/\z in
{1.1/2/124,1.1/-2/4}
{\node[ver] () at (\x,\y){\tiny{$T_{\z}^4$}};}

\foreach \x/\y in 
{1/4,2/3}
{\draw [line width=2pt, line cap=round, dash pattern=on 0pt off 1.7\pgflinewidth]  (e4\x) -- (e4\y);}

\foreach \x/\y in 
{1/2,4/3}
{\draw[line width=2pt, line cap=rectangle, dash pattern=on 1pt off 1]
(e4\x) -- (e4\y);}
\end{scope}

\end{scope}

\begin{scope}[shift={(5,-7.5)}]
\foreach \x/\y/\z in
{1.5/1.5/1,1.5/-1.5/2,-1.5/-1.5/3,-1.5/1.5/4,3/3/5,3/-3/6,-3/-3/7,-3/3/8,6/6/9,6/-6/10,-6/-6/11,-6/6/12,7.5/7.5/13,7.5/-7.5/14,-7.5/-7.5/15,-7.5/7.5/16,4.5/4.5/17,4.5/-4.5/18,-4.5/-4.5/19,-4.5/4.5/20}
{\node[verti] (f\z) at (\x,\y){};}

\foreach \x/\y/\z in
{1.1/2/13,1.1/-2/23,-1.1/-2/2,-1.1/2/1,5.6/6.5/1234,5.6/-6.5/34,-5.6/-6.5/4,-5.6/6.5/124,7.1/8/134,7.1/-8/234,-7.1/-8/24,-7.1/8/14}
{\node[ver] () at (\x,\y){\tiny{$T_{\z}^6$}};}

\foreach \x/\y/\z in
{2.6/3.5/123,2.6/-3.5/3,-2.6/-3.5/0,-2.6/3.5/12}
{\node[ver] () at (\x,\y){\tiny{$T_{\z}^3$}};}

\foreach \x/\y/\z in
{4.1/5/123,4.1/-5/3,-4.1/-5/0,-4.1/5/12}
{\node[ver] () at (\x,\y){\tiny{$T_{\z}^4$}};}

\foreach \x/\y in 
{17/9,18/10,19/11,20/12}
{\path[edge,dashed] (f\x)--(f\y);}

\foreach \x/\y in 
{1/4,2/3,6/7,8/5,10/11,12/9,14/15,16/13}
{\draw [line width=2pt, line cap=round, dash pattern=on 0pt off 1.7\pgflinewidth]  (f\x) -- (f\y);}

\foreach \x/\y in 
{1/2,4/3,17/18,19/20,10/9,12/11,14/13,16/15}
{\draw[line width=2pt, line cap=rectangle, dash pattern=on 1pt off 1]
(f\x) -- (f\y);}

\foreach \x/\y in 
{5/1,6/2,7/3,8/4,9/13,10/14,11/15,12/16}
{\path[edge,dotted] (f\x)--(f\y);}

\draw[edge,dashed] plot [smooth,tension=0.5] coordinates{(f1) (4.2,2.5) (f13)};

\draw[edge,dashed] plot [smooth,tension=0.5] coordinates{(f2) (4.2,-2.5) (f14)};

\draw[edge,dashed] plot [smooth,tension=0.5] coordinates{(f3) (-4.2,-2.5) (f15)};

\draw[edge,dashed] plot [smooth,tension=0.5] coordinates{(f4) (-4.2,2.5) (f16)};

\end{scope}

\end{tikzpicture}
\caption{The gem $\Gamma^2_1$ of $\mathbb{RP}^2 \times \mathbb{RP}^2$ with $64$ vertices.}\label{fig:Ex3}

\end{figure}

\begin{figure}[h!]
\tikzstyle{ver}=[]
\tikzstyle{verti}=[circle, draw, fill=black!100, inner sep=0pt, minimum width=3pt]
\tikzstyle{vert}=[circle, draw, fill=black!100, inner sep=0pt, minimum width=1pt]
\tikzstyle{edge} = [draw,thick,-]
    \centering

\begin{tikzpicture}[scale=0.4]

\begin{scope}[shift={(-16,16)}]
\foreach \x/\y/\z in
{1.5/1.5/1,1.5/-1.5/2,-1.5/-1.5/3,-1.5/1.5/4,3/3/5,3/-3/6,-3/-3/7,-3/3/8,4.5/4.5/9,4.5/-4.5/10,-4.5/-4.5/11,-4.5/4.5/12,6/6/13,6/-6/14,-6/-6/15,-6/6/16,7.5/6/17,7.5/-6/18,-7.5/-6/19,-7.5/6/20}
{\node[verti] (a\z) at (\x,\y){};}

\foreach \x/\y/\z in
{2.6/3.5/4,2.6/-3.5/3,-2.6/-3.5/13,-2.6/3.5/14,4.1/5/24,4.1/-5/23,-4.1/-5/123,-4.1/5/124,5.6/6.5/2,5.6/-6.5/234,-5.6/-6.5/1234,-5.6/6.5/12}
{\node[ver] () at (\x,\y){\tiny{$T_{\z}^1$}};}

\foreach \x/\y/\z in
{1.1/2/0,1.1/-2/34,-1.1/-2/134,-1.1/2/1}
{\node[ver] () at (\x,\y){\tiny{$T_{\z}^4$}};}

\foreach \x/\y/\z in
{8.15/6/0,8.1/-6/34,-8.3/-6/134,-8.1/6/1}
{\node[ver] () at (\x,\y){\tiny{$T_{\z}^3$}};}

\foreach \x/\y in 
{5/9,6/10,7/11,8/12,13/17,14/18,15/19,16/20}
{\path[edge,dashed] (a\x)--(a\y);}

\foreach \x/\y in 
{1/4,2/3,6/7,8/5,10/11,12/9,14/15,16/13}
{\draw [line width=2pt, line cap=round, dash pattern=on 0pt off 1.7\pgflinewidth]  (a\x) -- (a\y);}

\foreach \x/\y in 
{6/5,8/7,10/9,12/11,14/13,16/15,17/18,19/20}
{\draw[line width=2pt, line cap=rectangle, dash pattern=on 1pt off 1]
(a\x) -- (a\y);}

\foreach \x/\y in 
{5/1,6/2,7/3,8/4,9/13,10/14,11/15,12/16}
{\path[edge,dotted] (a\x)--(a\y);}

 \end{scope}

\begin{scope}[shift={(-4.5,16)}]
\foreach \x/\y/\z in
{-1/1/1,1/1/2,-1/0/3,1/0/4,-1/-1/5,1/-1/6}
{\node[vert] (\z) at (\x,\y){};}

\foreach \x/\y/\z in
{-1.5/1.5/1,1.5/1.5/2,-1.5/0/3,1.5/0/4,-1.5/-1.5/5,1.5/-1.5/6}
{\node[ver] () at (\x,\y){\tiny{$T_w^{\z}$}};}

\foreach \x/\y in
{1/2,5/6}
{\path[edge] (\x)--(\y);}

\foreach \x/\y in
{4/2,5/3}
{\path[edge,dotted] (\x)--(\y);}

\foreach \x/\y in
{3/2,5/4}
{\path[edge,dashed] (\x)--(\y);}

\end{scope}

\begin{scope}[shift={(4,16)}]

\begin{scope}[shift={(4,-2)}]
\foreach \x/\y/\z in
{1.5/1.5/1,1.5/-1.5/2,-1.5/-1.5/3,-1.5/1.5/4}
{\node[verti] (b2\z) at (\x,\y){};}

\foreach \x/\y/\z in
{1/2/4,1.1/-2/3}
{\node[ver] () at (\x,\y){\tiny{$T_{\z}^2$}};}

\foreach \x/\y/\z in
{-1.1/-2/1,-1.1/2/134}
{\node[ver] () at (\x,\y){\tiny{$T_{\z}^3$}};}

\foreach \x/\y in 
{1/4,2/3}
{\draw [line width=2pt, line cap=round, dash pattern=on 0pt off 1.7\pgflinewidth]  (b2\x) -- (b2\y);}

\foreach \x/\y in 
{1/2,4/3}
{\draw[line width=2pt, line cap=rectangle, dash pattern=on 1pt off 1]
(b2\x) -- (b2\y);}
\end{scope}

\begin{scope}[shift={(-4,-2)}]
\foreach \x/\y/\z in
{1.5/1.5/1,1.5/-1.5/2,-1.5/-1.5/3,-1.5/1.5/4}
{\node[verti] (b3\z) at (\x,\y){};}

\foreach \x/\y/\z in
{-1.1/-2/123}
{\node[ver] () at (\x,\y){\tiny{$T_{\z}^2$}};}

\foreach \x/\y/\z in
{-1.1/2/4}
{\node[ver] () at (\x,\y){\tiny{$T_{\z}^4$}};}

\foreach \x/\y/\z in
{1.1/-2/2}
{\node[ver] () at (\x,\y){\tiny{$T_{\z}^3$}};}

\foreach \x/\y/\z in
{1.1/2/134}
{\node[ver] () at (\x,\y){\tiny{$T_{\z}^5$}};}
\foreach \x/\y in 
{1/4,2/3}
{\draw [line width=2pt, line cap=round, dash pattern=on 0pt off 1.7\pgflinewidth]  (b3\x) -- (b3\y);}

\foreach \x/\y in 
{1/2,4/3}
{\draw[line width=2pt, line cap=rectangle, dash pattern=on 1pt off 1]
(b3\x) -- (b3\y);}
\end{scope}

\begin{scope}[shift={(-4,4)}]
\foreach \x/\y/\z in
{1.5/1.5/1,1.5/-1.5/2,-1.5/-1.5/3,-1.5/1.5/4}
{\node[verti] (b4\z) at (\x,\y){};}

\foreach \x/\y/\z in
{1/2/2,-1.1/2/12}
{\node[ver] () at (\x,\y){\tiny{$T_{\z}^2$}};}

\foreach \x/\y/\z in
{1.1/-2/134,-1.1/-2/34}
{\node[ver] () at (\x,\y){\tiny{$T_{\z}^4$}};}

\foreach \x/\y in 
{1/4,2/3}
{\draw [line width=2pt, line cap=round, dash pattern=on 0pt off 1.7\pgflinewidth]  (b4\x) -- (b4\y);}

\foreach \x/\y in 
{1/2,4/3}
{\draw[line width=2pt, line cap=rectangle, dash pattern=on 1pt off 1]
(b4\x) -- (b4\y);}
\end{scope}

\end{scope}

\begin{scope}[shift={(-6,5)}]
\begin{scope}[shift={(-15,0)}]
\foreach \x/\y/\z in
{1.5/1.5/1,1.5/-1.5/2,-1.5/-1.5/3,-1.5/1.5/4,3/3/5,3/-3/6,-3/-3/7,-3/3/8}
{\node[verti] (c1\z) at (\x,\y){};}

\foreach \x/\y/\z in
{-1.1/-2/134,-1.1/2/1,2.6/3.5/123,-2.6/3.5/12}
{\node[ver] () at (\x,\y){\tiny{$T_{\z}^3$}};}

\foreach \x/\y/\z in
{1.1/2/13,1.1/-2/14}
{\node[ver] () at (\x,\y){\tiny{$T_{\z}^1$}};}

\foreach \x/\y/\z in
{2.6/-3.5/124,-2.6/-3.5/1234}
{\node[ver] () at (\x,\y){\tiny{$T_{\z}^6$}};}

\foreach \x/\y in 
{6/7,8/5}
{\draw [line width=2pt, line cap=round, dash pattern=on 0pt off 1.7\pgflinewidth]  (c1\x) -- (c1\y);}

\foreach \x/\y in 
{1/2,4/3}
{\draw[line width=2pt, line cap=rectangle, dash pattern=on 1pt off 1]
(c1\x) -- (c1\y);}

\foreach \x/\y in 
{5/1,6/2,7/3,8/4}
{\path[edge] (c1\x)--(c1\y);}

\end{scope}

\begin{scope}[shift={(-5,0)}]
\foreach \x/\y/\z in
{1.5/1.5/1,1.5/-1.5/2,-1.5/-1.5/3,-1.5/1.5/4,3/3/5,3/-3/6,-3/-3/7,-3/3/8}
{\node[verti] (c2\z) at (\x,\y){};}

\foreach \x/\y/\z in
{-1.1/2/2,2.6/3.5/3,2.6/-3.5/4,-2.6/-3.5/34,-2.6/3.5/0}
{\node[ver] () at (\x,\y){\tiny{$T_{\z}^3$}};}

\foreach \x/\y/\z in
{-1.1/-2/234}
{\node[ver] () at (\x,\y){\tiny{$T_{\z}^6$}};}

\foreach \x/\y/\z in
{1.1/2/23,1.1/-2/24}
{\node[ver] () at (\x,\y){\tiny{$T_{\z}^1$}};}

\foreach \x/\y in 
{6/7,8/5}
{\draw [line width=2pt, line cap=round, dash pattern=on 0pt off 1.7\pgflinewidth]  (c2\x) -- (c2\y);}

\foreach \x/\y in 
{1/2,6/5,8/7}
{\draw[line width=2pt, line cap=rectangle, dash pattern=on 1pt off 1]
(c2\x) -- (c2\y);}

\foreach \x/\y in 
{5/1,6/2,7/3,8/4}
{\path[edge] (c2\x)--(c2\y);}

\end{scope}

\begin{scope}[shift={(5,0)}]
\foreach \x/\y/\z in
{1.5/1.5/1,1.5/-1.5/2,-1.5/-1.5/3,-1.5/1.5/4,3/3/5,3/-3/6,-3/-3/7,-3/3/8}
{\node[verti] (d1\z) at (\x,\y){};}

\foreach \x/\y/\z in
{1.1/2/0,1.1/-2/12,-1.1/-2/2,-1.1/2/1,2.6/3.5/4}
{\node[ver] () at (\x,\y){\tiny{$T_{\z}^4$}};}

\foreach \x/\y/\z in
{2.6/-3.5/124}
{\node[ver] () at (\x,\y){\tiny{$T_{\z}^1$}};}

\foreach \x/\y/\z in
{-2.6/-3.5/24,-2.6/3.5/14}
{\node[ver] () at (\x,\y){\tiny{$T_{\z}^6$}};}

\foreach \x/\y in 
{1/4,2/3}
{\draw [line width=2pt, line cap=round, dash pattern=on 0pt off 1.7\pgflinewidth]  (d1\x) -- (d1\y);}

\foreach \x/\y in 
{1/2,4/3,8/7}
{\draw[line width=2pt, line cap=rectangle, dash pattern=on 1pt off 1]
(d1\x) -- (d1\y);}

\foreach \x/\y in 
{5/1,6/2,7/3,8/4}
{\path[edge] (d1\x)--(d1\y);}

\end{scope}

\begin{scope}[shift={(15,0)}]
\foreach \x/\y/\z in
{1.5/1.5/1,1.5/-1.5/2,-1.5/-1.5/3,-1.5/1.5/4,3/3/5,3/-3/6,-3/-3/7,-3/3/8}
{\node[verti] (d2\z) at (\x,\y){};}

\foreach \x/\y/\z in
{1/1.9/3,1.1/-2/123,2.6/3.5/34,-2.6/3.5/134}
{\node[ver] () at (\x,\y){\tiny{$T_{\z}^4$}};}

\foreach \x/\y/\z in
{-1.1/-2/23,-1.1/2/13}
{\node[ver] () at (\x,\y){\tiny{$T_{\z}^6$}};}

\foreach \x/\y/\z in
{2.6/-3.5/1234,-2.6/-3.5/234}
{\node[ver] () at (\x,\y){\tiny{$T_{\z}^1$}};}

\foreach \x/\y in 
{6/7,8/5}
{\draw [line width=2pt, line cap=round, dash pattern=on 0pt off 1.7\pgflinewidth]  (d2\x) -- (d2\y);}

\foreach \x/\y in 
{1/2,4/3}
{\draw[line width=2pt, line cap=rectangle, dash pattern=on 1pt off 1]
(d2\x) -- (d2\y);}

\foreach \x/\y in 
{5/1,6/2,7/3,8/4}
{\path[edge] (d2\x)--(d2\y);}

\end{scope}

\end{scope}

\begin{scope}[shift={(-16,-7)}]

\begin{scope}[shift={(4,4)}]
\foreach \x/\y/\z in
{1.5/1.5/1,1.5/-1.5/2,-1.5/-1.5/3,-1.5/1.5/4}
{\node[verti] (e1\z) at (\x,\y){};}

\foreach \x/\y/\z in
{-1.1/-2/2,-1.1/2/1}
{\node[ver] () at (\x,\y){\tiny{$T_{\z}^5$}};}

\foreach \x/\y/\z in
{1/2/3,1.1/-2/123}
{\node[ver] () at (\x,\y){\tiny{$T_{\z}^4$}};}

\foreach \x/\y in 
{1/4,2/3}
{\draw [line width=2pt, line cap=round, dash pattern=on 0pt off 1.7\pgflinewidth]  (e1\x) -- (e1\y);}

\foreach \x/\y in 
{1/2,4/3}
{\draw[line width=2pt, line cap=rectangle, dash pattern=on 1pt off 1]
(e1\x) -- (e1\y);}

\end{scope}

\begin{scope}[shift={(-4,-2)}]
\foreach \x/\y/\z in
{1.5/1.5/1,1.5/-1.5/2,-1.5/-1.5/3,-1.5/1.5/4}
{\node[verti] (e3\z) at (\x,\y){};}

\foreach \x/\y/\z in
{1.1/2/34,-1.1/2/4}
{\node[ver] () at (\x,\y){\tiny{$T_{\z}^5$}};}

\foreach \x/\y/\z in
{1.1/-2/12,-1.1/-2/123}
{\node[ver] () at (\x,\y){\tiny{$T_{\z}^3$}};}

\foreach \x/\y in 
{1/4,2/3}
{\draw [line width=2pt, line cap=round, dash pattern=on 0pt off 1.7\pgflinewidth]  (e3\x) -- (e3\y);}

\foreach \x/\y in 
{1/2,4/3}
{\draw[line width=2pt, line cap=rectangle, dash pattern=on 1pt off 1]
(e3\x) -- (e3\y);}
\end{scope}
\end{scope}

\begin{scope}[shift={(5,-7.5)}]
\foreach \x/\y/\z in
{1.5/1.5/1,1.5/-1.5/2,-1.5/-1.5/3,-1.5/1.5/4,3/3/5,3/-3/6,-3/-3/7,-3/3/8,6/6/9,6/-6/10,-6/-6/11,-6/6/12,7.5/7.5/13,7.5/-7.5/14,-7.5/-7.5/15,-7.5/7.5/16,4.5/4.5/17,4.5/-4.5/18,-4.5/-4.5/19,-4.5/4.5/20}
{\node[verti] (f\z) at (\x,\y){};}

\foreach \x/\y/\z in
{1.1/2/13,1.1/-2/23,-1.1/-2/2,-1.1/2/1,5.6/6.5/1234,5.6/-6.5/34,-5.6/-6.5/4,-5.6/6.5/124,7.1/8/134,7.1/-8/234,-7.1/-8/24,-7.1/8/14}
{\node[ver] () at (\x,\y){\tiny{$T_{\z}^6$}};}

\foreach \x/\y/\z in
{2.6/3.5/123,2.6/-3.5/3,-2.6/-3.5/0,-2.6/3.5/12}
{\node[ver] () at (\x,\y){\tiny{$T_{\z}^3$}};}

\foreach \x/\y/\z in
{4.1/5/123,4.1/-5/3,-4.1/-5/0,-4.1/5/12}
{\node[ver] () at (\x,\y){\tiny{$T_{\z}^4$}};}

\foreach \x/\y in 
{17/9,18/10,19/11,20/12}
{\path[edge,dashed] (f\x)--(f\y);}

\foreach \x/\y in 
{1/4,2/3,6/7,8/5,10/11,12/9,14/15,16/13}
{\draw [line width=2pt, line cap=round, dash pattern=on 0pt off 1.7\pgflinewidth]  (f\x) -- (f\y);}

\foreach \x/\y in 
{1/2,4/3,17/18,19/20,10/9,12/11,14/13,16/15}
{\draw[line width=2pt, line cap=rectangle, dash pattern=on 1pt off 1]
(f\x) -- (f\y);}

\foreach \x/\y in 
{5/1,6/2,7/3,8/4,9/13,10/14,11/15,12/16}
{\path[edge,dotted] (f\x)--(f\y);}

\draw[edge,dashed] plot [smooth,tension=0.5] coordinates{(f1) (4.2,2.5) (f13)};

\draw[edge,dashed] plot [smooth,tension=0.5] coordinates{(f2) (4.2,-2.5) (f14)};

\draw[edge,dashed] plot [smooth,tension=0.5] coordinates{(f3) (-4.2,-2.5) (f15)};

\draw[edge,dashed] plot [smooth,tension=0.5] coordinates{(f4) (-4.2,2.5) (f16)};

\end{scope}

\end{tikzpicture}
\caption{The crystallization $\Gamma^\prime_1$ of $\mathbb{RP}^2 \times \mathbb{RP}^2$ with $52$ vertices.}\label{fig:Ex4}

\end{figure}
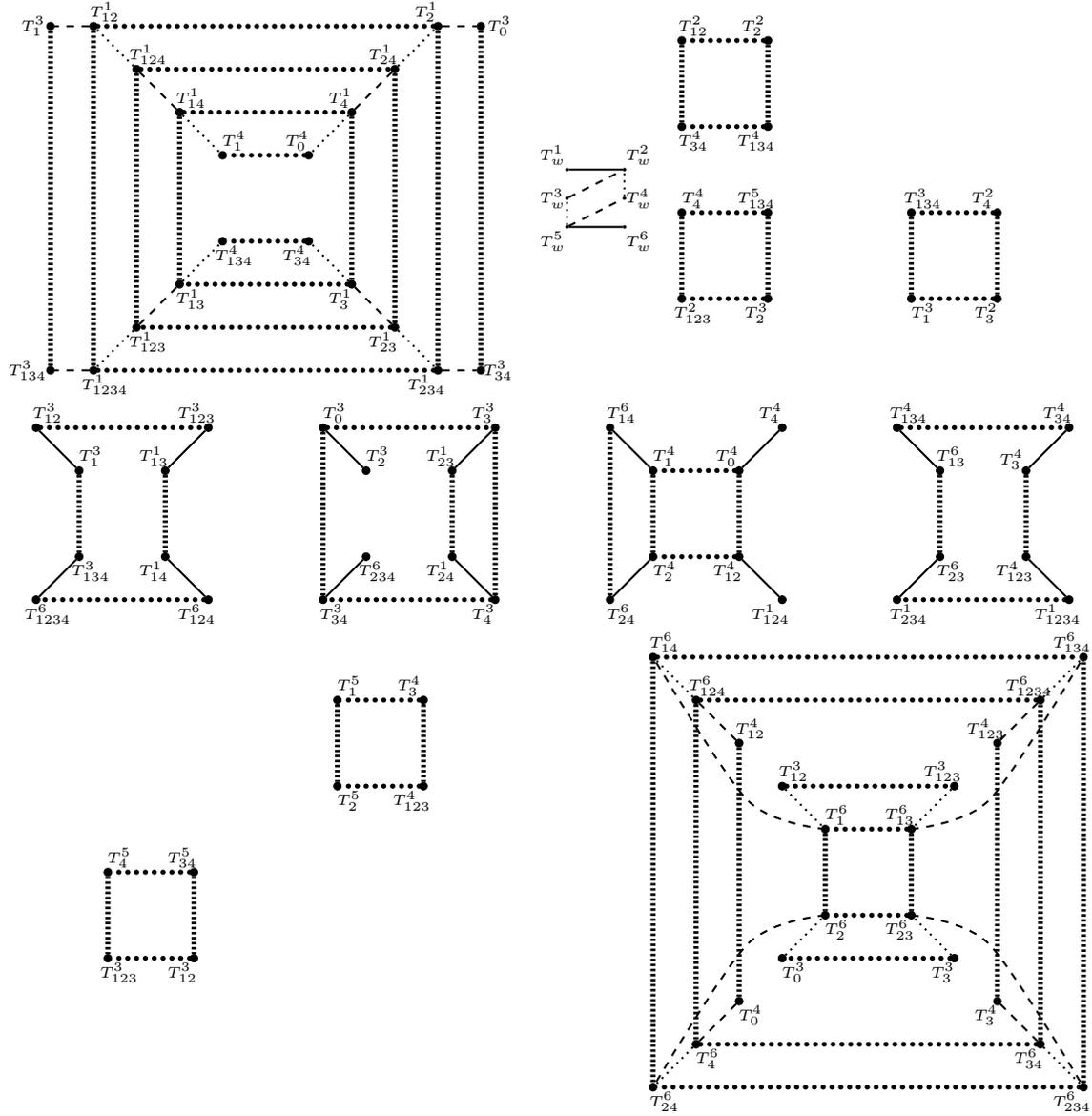

\begin{theorem}\label{thm:RP2}
The regular genus of each of the seven small covers over the simple polytope $P = \Delta^2 \times \Delta^2$ (up to D-J equivalence)  equals the minimum possible value, which is $8$. Furthermore, each of these small covers admits a genus-minimal crystallization with $52$ vertices.
\end{theorem}

\begin{proof}
    Let $\lambda:\mathcal F\to \mathbb{Z}_2^4$ be a $\mathbb{Z}_2$-characteristic function. From Lemma \ref{lemma:seven}, $\lambda=\lambda_i$ for some $1\le i\le 7$. Let $\Gamma$ ($=\Gamma_i$ for some $1\le i\le 7$) be the gem of $M^4(\lambda)$ obtained using the above construction. From Figure \ref{fig:T}, we have that $g_{\{0,2\}}=g_{\{0,3\}}=g_{\{0,4\}}=g_{\{1,3\}}=g_{\{1,4\}}=g_{\{2,4\}}=24$ and each such bi-colored cycle has length $4$. We also have that $g_{\{1,2\}}=g_{\{2,3\}}=16$ and each such bi-colored cycle has length $6$. Every $\{0,1\}$-colored cycle (resp. $\{3,4\}$-colored cycle) of the subgraph $S$ is of length $8$. Since remaining such bi-colored cycles are of length $4$, we have that $g_{\{0,1\}}=g_{\{3,4\}}=8+8=16$ of which $8$ are eight-cycles (resp. four-cycles). Again note that, in the gem $\Gamma$, we have that $g_{\hat{0}}=g_{\hat{4}}=1$, $g_{\hat{1}}=g_{\hat{3}}=2$ and $g_{\hat{2}}=3$.    
    
    From Figure \ref{fig:T}, we know that in $\Gamma$, if $T^1_{w_1}$ is connected to $T^1_{w_2}$ by an $i$-colored edge, then $T^2_{w_1}$ is connected to $T^2_{w_2}$ by an $i$-colored edge, where $i\in \{0,4\}$. It is also known that $T^1_{w}$ is connected to $T^2_{w}$ by a $2$-colored edge, for all $w\in W$. Thus, we can apply a polyhedral glue move with respect to $(\Phi_1,\Lambda_1,\Lambda_1^\prime,2)$, where $\Lambda_1$ consists of the vertices (of the form $T_w^1$) of the $\{0,4\}$-colored four-cycle containing $T_0^1$. If $T_w^1\in \Lambda_1$ for some $w\in W$, then $\Phi_1(T_w^1)=T_w^2$. Let the four-cycles generated by $\Lambda_1,\Lambda_1^\prime$ be $A_1,A_1^\prime$, respectively. Due to this polyhedral glue move, the $\{0,4\}$-colored four-cycle $A_1^\prime$ will be removed from $S$. Since $\{e_1,e_2,e_4,(c_1,c_2,1,1)\}$ is linearly independent (consider the vertex $v_2^0$), two distinct $\{0,4\}$-colored four-cycles, say $C_1$ and $C_1^\prime$, connected to $A_1$ by $3$- and $1$-colored edges in $\Gamma$, respectively, will replace $A_1^\prime$ in $S$. The cycles $C_1$ and $C_1^\prime$ are now incident with dangling $3$- and $1$-colored edges of $S$, respectively. Since vertices, from $C_1$ (resp. $C_1^\prime$), of the two induced $\{3,4\}$-colored cycles (resp. $\{0,1\}$-colored cycles) are not connected to another $\{3,4\}$-colored cycles (resp. $\{0,1\}$-colored cycles) of $S$ by $1$-colored (resp. $3$-colored) edges, we cannot apply polyhedral glue move involving color $1$ (resp. $3$) and these two induced $\{3,4\}$-colored cycles (resp. $\{0,1\}$-colored cycles). 

    Since $\{e_1,e_3,(1,1,c_3,c_4),(c_1,c_2,1,1)\}$ is linearly independent (consider the vertex $v_2^1$), each $\{0,4\}$-colored cycle, consisting of the vertices of the form $T_w^6$, has exactly one vertex $T_{w_0}^6$ such that $T_{w_0}^1\in \Lambda_1$. Let $A_2$ be the $\{0,4\}$-colored four-cycle containing $T_0^6$ and let $\Lambda_2$ be the set of vertices of $A_2$. Apply the second polyhedral glue move with respect to $(\Phi_2,\Lambda_2,\Lambda_2^\prime,2)$ (Figure \ref{fig:T}). Because $\{e_2,e_3,e_4,(1,1,c_3,c_4)\}$ is linearly independent (consider the vertex $v_2^2$), we again have a similar observation as above that we cannot apply polyhedral glue move involving color $1$ (resp. $3$) and the two induced (due to the second polyhedral glue move) $\{3,4\}$-colored cycles (resp. $\{0,1\}$-colored cycles). After these two polyhedral glue moves, we get the gem $\Gamma^1$ of $M^4(\lambda)$ with $80$ vertices, and we are left with the unique choice of polyhedral glue move involving the color $1$ (resp. $3$) and two $\{3,4\}$-colored (resp. $\{0,1\}$-colored) eight-cycles.
    
    The subscripts of the vertices of these two $\{3,4\}$-colored (resp. $\{0,1\}$-colored) eight-cycles are the entries from the third column (resp. fourth row) of the compact form of $S$ (Figure \ref{fig:CF}). Apply the polyhedral glue move with respect to $(\Phi_3,\Lambda_3,\Lambda_3^\prime,3)$ on $\Gamma^1$, where $\Lambda_3$ consists of the vertices of the form $T_w^j$, $j\in \{2,4\}$ and $w$ is an entry from the fourth row of the compact form of $S$. Due to this polyhedral glue move on $\Gamma^1$, we get the gem $\Gamma^2$ of $M^4(\lambda)$ with $64$ vertices, and all the $\{3,4\}$-colored eight-cycles of $S$ become six-cycles. Finally, applying the polyhedral glue move with respect to $(\Phi_4,\Lambda_4,\Lambda_4^\prime,1)$, where $\Lambda_4$ consists of the vertices of the form $T_w^j$, $j\in \{2,3\}$ and $w$ is an entry from the third column of the compact form of $S$, we get a crystallization $\Gamma^\prime$ of $M^4(\lambda)$ with $52$ vertices.
    
    Now, we will analyze bi-colored cycles. Let us first analyze $\{0,3\}$-colored cycles. In $\Gamma$, there are $24$ four-cycles colored by $\{0,3\}$. When applying the first polyhedral move, $2$ four-cycles colored by $\{0,3\}$ merge together, forming a four-cycle again. So, $2$ pairs of four-cycles result in $2$ four-cycles after the first polyhedral glue move. Similarly, after the second polyhedral glue move, $2$ pairs of four-cycles result in $2$ four-cycles. Thus, in $\Gamma^1$, there are $20$ four-cycles colored by $\{0,3\}$. In applying the third polyhedral glue move, $4$ four-cycles colored by $\{0,3\}$ are removed. Thus, in $\Gamma^2$, there are $16$ four-cycles colored by $\{0,3\}$. When applying the fourth polyhedral glue move, two four-cycles from each of the $3$ pairs of four-cycles get merged to form a four-cycles. Thus, in the crystallization $\Gamma^\prime$, there are $13$ four-cycles colored by $\{0,3\}$. Similarly, analyzing $\{0,4\}$-colored (resp. $\{1,4\}$-colored) cycles, we have that there are $13$ four-cycles colored by $\{0,4\}$ (resp. $\{1,4\}$) in $\Gamma^\prime$. 

    Let us now analyze $\{1,2\}$-colored (resp. $\{2,3\}$-colored) cycles. In the gem $\Gamma$, there are $16$ six-cycles colored by $\{1,2\}$ (resp. $\{2,3\}$). The vertices of a six-cycle colored by $\{1,2\}$ (resp. $\{2,3\}$) involves $3$ superscripts and $2$ subscripts. There are $8$ six-cycles colored by $\{1,2\}$ (resp. $\{2,3\}$) which involve superscripts $1,2,4$ (resp. $1,2,3$) and the other $8$ six cycles colored by $\{1,2\}$ (resp. $\{2,3\}$) involve superscripts $3,5,6$ (resp. $4,5,6$). Let $U_1=\{1,2,4\}$, $U_2=\{1,2,3\}$, $L_1=\{3,5,6\}$, and $L_2=\{4,5,6\}$. We denote a six-cycle, whose vertices involve subscripts $a,b$ and superscripts $c,d,e$, by $(a,b)^{\{c,d,e\}}$. Six-cycles colored by $\{1,2\}$ (resp. $\{2,3\}$) in $\Gamma$ are $(0,4)^{U_1},(1,14)^{U_1},(3,34)^{U_1},(13,134)^{U_1},(2,24)^{U_1}, (12,124)^{U_1},(23,234)^{U_1},(123,1234)^{U_1},(0,2)^{L_1},\\ (1,12)^{L_1},(3,23)^{L_1},(13,123)^{L_1}, (34,234)^{L_1},(134,1234)^{L_1}, (4,24)^{L_1},(14,124)^{L_1}$ (resp. $(0,2)^{U_2},\\ (34,234)^{U_2},(1,12)^{U_2},(134,1234)^{U_2},(3,23)^{U_2},(4,24)^{U_2},(13,123)^{U_2},(14,124)^{U_2},(0,4)^{L_2},\\ (3,34)^{L_2}, (12,124)^{L_2},(123,1234)^{L_2},(1,14)^{L_2},(13,134)^{L_2},(2,24)^{L_2},(23,234)^{L_2}$). Consider $\lambda=\lambda_1$. Due to the first and second glue moves, $8$ six-cycles $(0,4)^{U_1},(1,14)^{U_1},(3,34)^{U_1},(13,134)^{U_1},\\(0,2)^{L_1}, (1,12)^{L_1},(3,23)^{L_1},(13,123)^{L_1},$ (resp. $(0,2)^{U_2},(34,234)^{U_2},(1,12)^{U_2},(134,1234)^{U_2},\\(0,4)^{L_2},(3,34)^{L_2},(12,124)^{L_2},(123,1234)^{L_2}$) colored by $\{1,2\}$ (resp. $\{2,3\}$) in $\Gamma_1$ reduces to four-cycles in $\Gamma_1^1$. 
    
    Due to the third glue move, $8$ six-cycles $(13,123)^{U_2},(14,124)^{U_2},(4,24)^{U_2},(3,23)^{U_2},\\ (13,134)^{L_2},(1,14)^{L_2},(2,24)^{L_2},(23,234)^{L_2}$ colored by $\{2,3\}$ in $\Gamma_1^1$ reduces to four-cycles in $\Gamma_1^2$. Thus, there are $16$ four-cycles colored by $\{2,3\}$ in $\Gamma_1^2$. This is not the case, if we consider $\lambda=\lambda_3,\lambda_4,\lambda_6,$ or $\lambda_7$. Now, when applying the fourth glue move, $2$ four-cycles colored by $\{2,3\}$ merge together forming a four-cycle again. So, $3$ pairs of four-cycles colored by $\{2,3\}$ in $\Gamma_1^2$ result in $3$ four-cycles colored by $\{2,3\}$ in $\Gamma_1^\prime$. Hence, we have $13$ four-cycles colored by $\{2,3\}$ in the crystallization $\Gamma_1^\prime$. Similarly, analyzing $\{2,3\}$-colored cycles for $\lambda=\lambda_i$, where $2\le i\le 7$, we get that $g(\Gamma_{i_{\{2,3\}}}^\prime)=13$. There are $13$ four-cycles colored by $\{2,3\}$ in the crystallization $\Gamma_i^\prime$, for all $i\in \{1,2,5\}$, and $9$ four-cycles, $2$ six-cycles, $2$ two-cycles colored by $\{2,3\}$ in the crystallization $\Gamma_i^\prime$, for all $i\in \{3,4,6,7\}$. 

    In $\Gamma_1^1$, the $\{1,2\}$-colored cycles $(13,134)^{U_1}$ and $(13,123)^{L_1}$ are four-cycles, the $\{1,2\}$-colored cycles $(1,14)^{U_1}$ and $(14,124)^{L_1}$ are four-cycle and six-cycle, respectively, the $\{1,2\}$-colored cycles $(2,24)^{U_1}$ and $(4,24)^{L_1}$ are six-cycles, and the $\{1,2\}$-colored cycles $(23,234)^{U_1}$ and $(3,23)^{L_1}$ are six-cycle and four-cycle, respectively. When applying the third polyhedral glue move, these $\{1,2\}$-colored cycles merge together pairwise and result in four-cycle, six-cycle, eight-cycle, and six-cycle in $\Gamma_1^2$, respectively. So, there are $5$ four-cycles, $6$ six-cycles, and $1$ eight-cycle colored by $\{1,2\}$ in $\Gamma_1^2$. Since the $\{1,2\}$-colored cycles $(23,234)^{U_1}$ (induced by the third glue move), $(34,234)^{L_1},(123,1234)^{U_1},(134,1234)^{L_1}, (12,124)^{U_1},(14,124)^{L_1}$ (induced by the third glue move) are six-cycles in $\Gamma_1^2$, these cycles reduces to four-cycles, due to the fourth glue move, in the crystallization $\Gamma_1^\prime$. Hence, we have $1$ eight-cycle and $11$ four-cycles colored by $\{1,2\}$ in $\Gamma_1^\prime$. Similarly, analyzing $\{1,2\}$-colored cycles for $\lambda=\lambda_i$, where $2\le i\le 7$, we get that $g(\Gamma_{i_{\{1,2\}}}^\prime)=12$. There are $11$ four-cycles, $1$ eight-cycle colored by $\{1,2\}$ in the crystallization $\Gamma_i^\prime$, where $i\in \{1,3,6\}$, and $8$ four-cycles, $3$ six-cycles, $1$ two-cycle colored by $\{1,2\}$ in the crystallization $\Gamma_i^\prime$, where $i\in \{2,4,5,7\}$.

Let us compute the regular genus of $\Gamma^\prime$ with respect to the permutation $\varepsilon=(0,3,2,1,4)$ (cf. Subsection \ref{sec:genus}). $$\rho_{\varepsilon}(\Gamma^\prime)=1-\frac{1}{2}\left(\frac{-3}{2}(52)+g_{\{0,3\}}+g_{\{2,3\}}+g_{\{1,2\}}+g_{\{1,4\}}+g_{\{0,4\}}\right)$$ $$=1-\frac{1}{2}(-78+13+13+12+13+13)=8.$$ Thus, we have $\mathcal G(M^4(\lambda))\le \rho(\Gamma^\prime)\le \rho_{\varepsilon}(\Gamma^\prime)=8$. On the other hand, it is easy to verify that the rank of the fundamental group of $M^4(\lambda)$ is 2. Therefore, due to Proposition \ref{lbrg}, we have $\mathcal G(M^4(\lambda))\ge 8$. Thus, we get that $\mathcal G(M^4(\lambda))=\rho(\Gamma^\prime)=8$.
\end{proof}

Figures \ref{fig:Ex1}, \ref{fig:Ex2}, \ref{fig:Ex3} and \ref{fig:Ex4} represent gems $\Gamma_1,\Gamma_1^1, \Gamma_1^2$ and the crystallization $\Gamma_1^\prime$ of $\mathbb{RP}^2\times \mathbb{RP}^2$, respectively. The isomorphism signature of the crystallization $\Gamma_1^\prime$ of $\mathbb{RP}^2\times \mathbb{RP}^2$ obtained using Regina \cite{burton2013regina}, is the following.

\noindent 0vvvvLLAALAwwzwvPPzLAMMAMAPQQMLPQAQPQQQQQQQQclimpjnrqprqxuDGEH\\FLGMzxHvBJOMLzTOECICJNEUIGXDVPRMWPQQZKYRLRSPXQXXWSVWTWUUZ\\VZYYZaaaaaaaaaaaaaaaaaaaaaaaaaaaaaaaaaaaaaaaaaaaaaaaaaaaaaaaaaaaaaaaaaaaaaaaaaa\\aaaaaaaaaaaaaaaaaaaaaaaaaaaaaaaaaaaaaaaaaaaaaaaaaaaaaaaaaaaaaaaaaaaaaaaaaaaaaaaaa\\aaa

\noindent We know that the small covers $M^4(\lambda_2)$ and $M^4(\lambda_5)$ are not D-J equivalent. Since their crystallizations $\Gamma_2^\prime$ and  $\Gamma_5^\prime$ have the same isomorphism signature, we have that $M^4(\lambda_2)$ and $M^4(\lambda_5)$ are PL homeomorphic. The isomorphism signature of the crystallizations $\Gamma_2^\prime$ and  $\Gamma_5^\prime$ is the following.

\noindent 0LvLvLLzvwMwwwAzMzzPMPLQLMLPMQPAQQMQQQQQQQQQcbfgknrsunlylBDAGC\\vHwHKIKNOICROPTSGMDJWEJIFXTVTEVPLQZLKYUTUSYSUXPVURQZQXYXZV\\YWWZaaaaaaaaaaaaaaaaaaaaaaaaaaaaaaaaaaaaaaaaaaaaaaaaaaaaaaaaaaaaaaaaaaaaaaaaaa\\aaaaaaaaaaaaaaaaaaaaaaaaaaaaaaaaaaaaaaaaaaaaaaaaaaaaaaaaaaaaaaaaaaaaaaaaaaaaaaaaa\\aaa

\noindent Since the crystallizations $\Gamma_3^\prime$ and  $\Gamma_6^\prime$ have the same isomorphism signature, we have that $M^4(\lambda_3)$ and $M^4(\lambda_6)$ are PL homeomorphic. The isomorphism signature of the crystallizations $\Gamma_3^\prime$ and  $\Gamma_6^\prime$ is the following.

\noindent 0LvLvMLzLLPvLLwMzzPMPLMzzQPwAQMQQMQQQQQPQQQQcbfggkmqrsmwxAzEBt\\FuFIGILMMNQPEKCHTCUHGDWQSQUVSNJJIXRQRPVXPRWNSROYZWXWYXTV\\TZYZYZaaaaaaaaaaaaaaaaaaaaaaaaaaaaaaaaaaaaaaaaaaaaaaaaaaaaaaaaaaaaaaaaaaaaaaaa\\aaaaaaaaaaaaaaaaaaaaaaaaaaaaaaaaaaaaaaaaaaaaaaaaaaaaaaaaaaaaaaaaaaaaaaaaaaaaaaaa\\aaaaaa

\noindent Since the crystallizations $\Gamma_4^\prime$ and  $\Gamma_7^\prime$ have the same isomorphism signature, we have that $M^4(\lambda_4)$ and $M^4(\lambda_7)$ are PL homeomorphic. The isomorphism signature of the crystallizations $\Gamma_4^\prime$ and  $\Gamma_7^\prime$ is the following.

\noindent 0LvLLzvzvwzwLAwzwvAQPwMPQQLQAMQPAQQPMQQQQQQQcbfijmrsvAABtuDwJH\\zKIxGEPCERQHFHDRSCBTAyGIVOWQKMXFLJLWSWXRPSYUTRTOSTZUYXZXZY\\WVYZaaaaaaaaaaaaaaaaaaaaaaaaaaaaaaaaaaaaaaaaaaaaaaaaaaaaaaaaaaaaaaaaaaaaaaaaaa\\aaaaaaaaaaaaaaaaaaaaaaaaaaaaaaaaaaaaaaaaaaaaaaaaaaaaaaaaaaaaaaaaaaaaaaaaaaaaaaaa\\aaaa

In summary, we have seven small covers over the simple polytope $P=\Delta^2 \times \Delta^2$ up to DJ equivalence, and among these, at most four are distinct up to PL homeomorphism. The regular genus of each of the small covers over $\Delta^2 \times \Delta^2$ is $8$.

\begin{remark}{\rm

The notion of weak semi-simple crystallizations for PL $4$-manifolds was introduced in \cite{bb18}. Let $M$ be a closed connected PL $4$-manifold and $m$ be the rank of the fundamental group of $M$. A crystallization $(\Gamma,\gamma)$ of $M$ is said to be a weak semi-simple crystallization with respect to the cyclic permutation $\varepsilon=(\varepsilon_0,\varepsilon_1,\cdots, \varepsilon_4)$ if $g_{\{\varepsilon_i,\varepsilon_{i+1},\varepsilon_{i+2}\}} = m + 1$ for all $i \in \Delta_4$ (addition in subscript of $\varepsilon$ is modulo $5$). The following result from \cite{bb18} provides a necessary and sufficient condition for a closed connected PL $4$-manifold such that its regular genus realizes the lower bound in Proposition \ref{lbrg}. Let $M$ be a closed connected PL $4$-manifold and $m$ be the rank of its fundamental group. Then $\mathcal G(M)= 2\chi(M)+5m-4$ if and only if $M$ admits a weak semi-simple crystallization.

Due to this result, we get that the crystallizations obtained in this article are weak semi-simple. In particular, the crystallizations $\Gamma_1^\prime,\ G_1^\prime,$ and $ G_2^\prime$ of $\mathbb{RP}^2 \times \mathbb{RP}^2 $, $\mathbb{S}^2 \times \mathbb{S}^1 \times \mathbb{S}^1$, and $\mathbb{S}^1 \times \mathbb{S}^1 \times \mathbb{S}^1 \times \mathbb{S}^1$ are weak semi-simple with respect to the cyclic permutations $(0,3,2,1,4),$ $(0,2,4,1,3)$, and $(0,2,4,1,3)$, respectively. 
}
\end{remark}

\section{Some Questions}

In Section \ref{Section:3}, we constructed genus-minimal crystallizations of $\mathbb{S}^2\times \mathbb{S}^1\times \mathbb{S}^1$ and $\mathbb{S}^1\times \mathbb{S}^1\times \mathbb{S}^1\times \mathbb{S}^1$ using the crystallizations of $\mathbb{S}^2\times \mathbb{S}^1$ and $\mathbb{S}^1\times \mathbb{S}^1\times \mathbb{S}^1$. Hence, a natural question arises:
\smallskip

\noindent \textbf{Question 1.} Given a crystallization $(\Gamma,\gamma)$ of a $3$-manifold $M$, can a genus-minimal crystallization of $M\times \mathbb{S}^1$ be constructed?
\smallskip

In Section \ref{Section:4}, we considered $P=\Delta^2\times \Delta^2$ and constructed genus-minimal crystallizations of all seven small covers over the simple polytope $P$. One of these small covers is $\mathbb{RP}^2 \times \mathbb{RP}^2$. It is known that the orientable surface of genus $g$, $T_g=\#_g T^2$, where $T^2$ is the $2$-torus and $g\ge 1$, is a small cover over the $2(g+1)$-gon such that $(1,0)$ and $(0,1)$ are the $\mathbb{Z}_2$-characteristic vectors of the edges of the $2(g+1)$-gon, alternatingly. It is also known that the non-orientable surface of genus $g$, $U_g=\#_g \mathbb{RP}^2$, where $g\ge 1$, is a small cover over the $(g+2)$-gon such that $(1,1)$ is the $\mathbb{Z}_2$-characteristic vector of an edge, say $e$, of the $(g+2)$-gon, and $(1,0)$ and $(0,1)$ are the $\mathbb{Z}_2$-characteristic vectors of the  edges, except $e$, of the $(g+2)$-gon, alternatingly. This implies that $T_{g_1}\times T_{g_2}$, $T_{g_1}\times U_{g_2}$ and $U_{g_1}\times U_{g_2}$ are small covers over $2(g_1+1)$-gon $\times \ 2(g_2+1)$-gon, $2(g_1+1)$-gon $\times \ (g_2+2)$-gon and $(g_1+2)$-gon $\times \ (g_2+2)$-gon, respectively, with suitable extensions of $\mathbb{Z}_2$-characteristic functions (cf. Subsection \ref{smallcover}). Hence, the question is the following.

\smallskip
\noindent \textbf{Question 2.} Can genus-minimal crystallizations of $T_{g_1}\times T_{g_2}$, $T_{g_1}\times U_{g_2}$ and $U_{g_1}\times U_{g_2}$ be constructed using the similar construction technique presented in Section \ref{Section:4}?

%\bigskip

%\noindent {\bf Acknowledgement:}  The first author is supported by the Institute fellowship from the Indian Institute of Technology Delhi, India.

\smallskip

\noindent {\bf Data availability:} The authors declare that all data supporting the findings of this study are available within the article.

\smallskip

%\noindent {\bf Declarations}

\noindent {\bf Conflict of interest:} No potential conflict of interest was reported by the authors.

\end{document}